\documentclass[11pt,reqno]{amsart}
\usepackage[margin=1in,letterpaper]{geometry}
\usepackage{graphicx}
\usepackage{amssymb}
\usepackage{amsthm}
\usepackage{mathrsfs}
\usepackage{stmaryrd}
\usepackage{accents} 
\usepackage{enumitem} 
\usepackage{hyperref} 

\makeatletter\let\over\@@over\makeatother

\numberwithin{equation}{section}
\theoremstyle{plain} 
\newtheorem{theorem}{Theorem}[section] 
\newtheorem{proposition}[theorem]{Proposition} 
\newtheorem{corollary}[theorem]{Corollary}
\newtheorem{lemma}[theorem]{Lemma}
\theoremstyle{remark}
\newtheorem{remark}[theorem]{Remark}

\newcommand{\dist}{\operatorname{dist}}
\newcommand{\realpart}{\operatorname{Re}}
\newcommand{\imagpart}{\operatorname{Im}}
\newcommand{\signum}[1]{\operatorname{sgn}{#1}}
\newcommand{\sech}{\operatorname{sech}} 
\newcommand{\jump}[1]{\left\llbracket{#1}\right\rrbracket}

\newcommand{\ucenter}{u^{\mathrm{c}} }
\newcommand{\varuc}{\dot{u}^{\mathrm{c}} }
\newcommand{\vvaruc}{\ddot{u}^{\mathrm{c}} }
\newcommand{\vvvaruc}{\dddot{u}^{\mathrm{c}} }
\newcommand{\Phicr}{\Phi_{\mathrm{cr}}}

\newcommand{\Xb}{{X_{\mathrm{b}} }}
\newcommand{\Yb}{{Y_{\mathrm{b}} }}

\newcommand{\LV}{\left|}
\newcommand{\RV}{\right|}
\newcommand{\LB}{\left[}
\newcommand{\RB}{\right]}
\newcommand{\LC}{\left(}
\newcommand{\RC}{\right)}

\newcommand{\p}{\partial}
\newcommand{\R}{\mathbb{R}} 
\newcommand{\id}{\operatorname{id}}

\newcommand{\even}{\mathrm{e}}      
\newcommand{\bdd}{\mathrm{b}}       
\newcommand{\F}{{\mathscr F}}       

\newcommand{\ham}{{\mathcal H}}              
\newcommand{\symp}{\omega}                   
\newcommand{\sympep}{\omega^\epsilon}        
\newcommand{\man}{\mathcal M}                
\newcommand{\cman}{{\mathcal{W}^\epsilon}}     
\newcommand{\cmang}{\mathcal{V}^\epsilon}    
\newcommand{\csymp}{\gamma}                  
\newcommand{\csympg}{\varpi^\epsilon}                
\newcommand{\cham}{\mathcal{K}^\epsilon}     
\newcommand{\chamzero}{\mathcal{K}^0}        
\newcommand{\vf}{\mathcal V_{\ham^\epsilon}} 
\newcommand{\twodcham}{K^\epsilon}
\newcommand{\dirnu}{\nu_{\mathrm{D}}}
\newcommand{\dirF}{F_{\mathrm{D}}}
\newcommand{\nuF}{F_{\mathrm{N}}}
\newcommand{\dirmu}{\mu_{\mathrm{D}}}
\newcommand{\numu}{\mu_{\mathrm{N}}}
\newcommand{\qprime}{A}

\newcommand{\reverser}{\mathcal S}  
\newcommand{\creverser}{S}          
\newcommand{\linear}{L}             
\newcommand{\X}{{\mathcal X}}       
\newcommand{\Y}{{\mathcal Y}}       
\newcommand{\U}{{\mathcal U}}       
\newcommand{\cs}{{\mathrm{c}}}      
\newcommand{\hs}{{\mathrm{su}}}     

\newcommand{\D}{{\mathcal D}}          
\newcommand{\cm}{{\mathscr C}}         
\newcommand{\loc}{{\mathrm{loc}} }     
\newcommand{\Fcr}{F_{\mathrm{cr}} }    
\newcommand{\mucr}{\mu_{\mathrm{cr}} } 
\newcommand{\Phia}{\tilde\Phi}         

\newcommand{\flowforce}{\mathscr{S}}
  
\newcommand{\genG}{\mathcal F}
\newcommand{\genI}{\mathcal I}
\newcommand{\genU}{\mathcal U}
\newcommand{\genX}{\mathcal X}
\newcommand{\genY}{\mathcal Y}
  
\newcommand{\placeholder}{\;\cdot\;}
\newcommand{\maps}{\colon}         
\newcommand{\by}{\times}         
\newcommand{\sub}{\subset}         
\newcommand{\til}{\widetilde}
\newcommand{\n}[2][]{#1\lVert #2 #1\rVert}
\newcommand{\abs}[2][]{#1\lvert #2 #1\rvert}
\newcommand{\dell}{\partial}

\begin{document}

\title[Stratified solitary waves]{Existence and qualitative theory for stratified solitary water waves }
\date{\today}

\author[R. M. Chen]{Robin Ming Chen}
\address{Department of Mathematics, University of Pittsburgh, Pittsburgh, PA 15260} 
\email{mingchen@pitt.edu}  
\thanks{The research of the first author is supported in part by the Simons Foundation under Grant 354996 and the Central Research Development Fund No.\ 04.13205.30205 from the University of Pittsburgh}

\author[S. Walsh]{Samuel Walsh}
\address{Department of Mathematics, University of Missouri, Columbia, MO 65211} 
\email{walshsa@missouri.edu} 
\thanks{The research of the second author is supported in part by the National Science Foundation through DMS-1514910}

\author[M. H. Wheeler]{Miles H. Wheeler}
\address{Courant Institute of Mathematical Sciences, New York University, New York, NY 10012}
\email{mwheeler@cims.nyu.edu}
\thanks{The research of the third author supported in part by the National Science Foundation through DMS-1400926}

\begin{abstract}  
  This paper considers two-dimensional gravity solitary waves moving through a body of density stratified water lying below vacuum.  The fluid domain is assumed to lie above an impenetrable flat ocean bed, while the interface between the water and vacuum is a free boundary where the pressure is constant.  We prove that, for any smooth choice of upstream velocity field and density function, there exists a continuous curve of such solutions that includes large-amplitude surface waves.  Furthermore, following this solution curve, one encounters waves that come arbitrarily close to possessing points of  horizontal stagnation.

  We also provide a number of results characterizing the qualitative features of solitary stratified waves.  In part, these include bounds on the wave speed from above and below, some of which are new even for constant density flow; an a priori bound on the velocity field and lower bound on the pressure; a proof of the nonexistence of monotone bores in this physical regime; and a theorem ensuring that all supercritical solitary waves of elevation have an axis of even symmetry.  
\end{abstract}

\maketitle

\tableofcontents

\section{Introduction}

Water in the depths of the ocean has a lower temperature and higher salinity than that found near the surface.  The resulting density distribution is thus \emph{heterogeneous} or \emph{stratified}, which creates the potential for types of wavelike  motion not possible in the constant density regime.  Indeed, the density strata in the bulk are themselves free surfaces along which waves may propagate.  This can lead to a remarkable phenomenon wherein large-amplitude waves steal through the interior of the fluid while leaving the upper surface nearly undisturbed.  Field observations have revealed that these internal waves are a common feature of coastal flows, and they are believed to play a central role in the dynamics of ocean mixing (cf., e.g., \cite{helfrich2006review}).   
Stratified waves can be truly immense while traveling over vast distances:  they include the largest waves ever recorded, with amplitude measuring up to 500 meters in some instances \cite{alford2015formation}.    

In this paper, we investigate two-dimensional solitary waves moving through a heterogeneous body of water.  These are a type of traveling wave:  they consist of a spatially localized disturbance riding at a constant velocity along an underlying current without changing shape.  Solitary water waves have a long and rich mathematical history, stretching back to their discovery by Russell in 1834 (cf.\ \cite{russell1844report}).  
Most of the work on this subject has been devoted to the homogeneous and irrotational regime, that is, the density is assumed to be constant and the curl of the velocity is assumed to vanish identically.  This choice permits the use of many powerful tools from complex analysis such as conformal mappings and nonlocal formulations on the boundary.   However, stratification generically creates vorticity, and thus investigations of heterogeneous waves are most naturally made in the rotational setting.  The first rigorous existence theory for traveling waves with stratification was provided by Dubreil-Jacotin \cite{dubreil1937theoremes} over a century after Russell's discovery.  She specifically studied the small-amplitude and periodic regime by means of a new formulation of the problem that did not rely on conformal transformations.  

Solitary waves are typically more difficult to analyze than periodic waves due to compactness issues that we will elaborate below.  Stratification further complicates matters by allowing for a wealth of possible qualitative structures.  In \cite{terkrikorov1960existence,ter1963theorie}, Ter-Krikorov proved the existence of small-amplitude solitary waves with non-constant density.  
The first large-amplitude existence result for stratified solitary waves was given by Amick \cite{amick1984semilinear} and Amick--Turner \cite{amick1986global}.  This was the culmination of a burst of activity in the 1980s devoted largely to \emph{channel flows}, that is, stratified waves in an infinite strip bounded above and below by rigid walls (see also \cite{bona1983finite,turner1981internal,turner1984variational}).  Remarkably, these works came two decades before the development of an existence theory for large-amplitude periodic water waves with vorticity by Constantin and Strauss \cite{constantin2004exact}.  
What explains this seeming discrepancy is that Amick, Turner, and their contemporaries restricted their attention to waves whose velocity is constant upstream and downstream, and hence are asymptotically irrotational.  
This assumption enabled them to handle stratification without confronting the effects of vorticity in their full generality.   For further discussion of the literature, see Section~\ref{history section}.  

One of the main contributions of the present work is an existence theory for large-amplitude surface solitary waves with density stratification.  We are able to allow an arbitrary smooth density distribution and horizontal velocity profile at infinity. 
Moreover, the families we construct continue up to the appearance of an ``extreme wave'' that has a stagnation point.  This is connected to the famous Stokes conjecture, which originally pertained to periodic irrotational waves (cf., \cite{stokes1880theory,amick1982stokes}) but has since been extended to other regimes.  
For example, in the setting of irrotational solitary waves, Amick and Toland proved the existence of a continuum that limited to stagnation \cite{amick1981solitary}. For stratified solitary waves which are asymptotically irrotational,  Amick \cite{amick1984semilinear} constructed a family of solutions and proved that either it contains an extreme wave in its closure, or a certain alternative occurs that he deemed highly unusual and conjectured never happens (see \cite[Theorem~7.4]{amick1984semilinear}).  
Here, we are able to state without qualification that our continuum limits to stagnation.  
This is the first such result for rotational solitary waves, even in the constant density case (see \cite{wheeler2013solitary} and \cite[Section~6]{wheeler2015froude}).

Our existence theory is built upon a host of new theorems concerning the qualitative properties of water waves with stratification.  
We first construct a family of small-amplitude waves via center manifold reduction methods.  The full continuum is then obtained using a new global bifurcation scheme that abstracts and extends the ideas of \cite{wheeler2013solitary,wheeler2015pressure}.   
This analysis hinges critically on having a thorough understanding of the possible structures that may arise as one moves away from the small-amplitude regime, and hence the qualitative theory plays an essential role.

It is worth mentioning that a great deal of recent research has centered on the Cauchy problem for water waves in various physical regimes.
At present, a number of authors have proved results concerning the global in time well-posedness for irrotational waves with small data (see, e.g., \cite{germain2009global,wu2011global,germain2015capillary,ionescu2015gravity,alazard2013global}), or local in time existence for rotational waves or interfacial flows (see, \cite{coutand2007wellposedness,shatah2008geometry,shatah2011interface}).  Yet the stratified waves we wish to study are both colossal and long-lived.  They are also fundamentally rotational due to the baroclinic generation of vorticity.  In short, by considering the steady regime,  we are able to treat waves that lie far beyond the current limitations of  the time-dependent theory. 

Now, let us describe the setting of the problem more precisely.  We are interested in two-dimensional solitary waves with heterogeneous density $\varrho$ which travel with constant speed $c$ under the influence of gravity. Changing to a moving reference frame enables us to eliminate time dependence from the system. The wave then occupies a steady fluid domain
\begin{equation*}
  \Omega = \{ (x,y) \in \R^2 : -d < y < \eta(x) \}, 
\end{equation*}
where the a priori unknown function $\eta$ is the free surface profile and $\{ y = -d\}$ is an impermeable flat bed.   We suppose that $\varrho > 0$ in $\overline{\Omega}$, and also that the fluid is continuously stratified in the sense that $\varrho$ is smooth. Moreover, the fluid is taken to be stably stratified in that heavier fluid elements lie below lighter elements, which translates to $y \mapsto \varrho(\cdot, y)$ being non-increasing.  

A stratified water wave is described mathematically by the fluid domain $\Omega$, density $\varrho \colon \Omega \to \R_+$, velocity field $(u,v) \colon \Omega \to \R^2$, and pressure $P \colon \Omega \to \R$. The governing equations are the incompressible steady Euler system, which consists of the conservation of mass
\begin{subequations}  \label{euler}
  \begin{equation} 
    (u-c) \varrho_x + v \varrho_y  =  0 \qquad \textrm{in } \Omega, \label{mass}  
  \end{equation}
  conservation of momentum 
  \begin{align} 
    \left\{ 
    \begin{alignedat}{2} \label{momentum}
      \varrho (u-c) u_x + \varrho v u_y  & =  -P_x  &  \\
      \varrho (u-c) v_x + \varrho v v_y & =  -P_y - g \varrho & 
    \end{alignedat} 
    \right. 
    \qquad 
    \textrm{in } \Omega,
  \end{align} 
  and incompressibility
  \begin{equation} 
    u_x + v_y =  0 \qquad \textrm{in } \Omega. \label{volume}  
  \end{equation}
\end{subequations}
Here $g > 0$ is the gravitational constant of acceleration. 

The free surface is assumed to be a material line, which results in the kinematic boundary condition
\begin{subequations} \label{eulerboundary} 
  \begin{alignat}{2}
    v &=  (u-c)  \eta_x &\qquad & \textrm{on } y = \eta(x). \label{kinematicsurface} 
  \end{alignat}
  The pressure is required to be continuous over the interface
  \begin{equation} 
    P =  \displaystyle P_{\textrm{atm}} \qquad \textrm{on } y = \eta(x), \label{dynamic} 
  \end{equation}
  where $P_{\textrm{atm}}$ is the (constant) atmospheric pressure.  Finally, the ocean bed is taken to be impermeable and thus
  \begin{equation}
    v =  0 \qquad \textrm{on } y = -d. \label{kinematicbed} 
  \end{equation}
\end{subequations}

It will be important for our later reformulations to require that there is no horizontal stagnation in the flow:
\begin{equation} 
  u-c < 0 \qquad \textrm{in } \overline{\Omega} \label{nostagnation}. 
\end{equation}
One consequence of this assumption is that the streamlines, which are the integral curves of the relative velocity field $(u-c,v)$, extend from $x = -\infty$ to $x = +\infty$; see Figure~\ref{streamline figure}(b). Indeed, a simple application of the implicit function theorem shows that each streamline is the graph of a function of $x$.  

In this work we are concerned with \emph{solitary waves}, which are traveling wave solutions satisfying the asymptotic conditions
\begin{equation} 
  \label{upstream condition} 
  (u,v) \to (\mathring{u},0),
  \quad 
  \varrho \to \mathring\varrho,
  \quad
  \eta \to 0 
  \qquad \textrm{as } |x| \to \infty 
\end{equation}
uniformly in $y$. Here $\mathring{u} = \mathring{u}(y)$ is a given far-field velocity profile, and $\mathring\varrho = \mathring\varrho(y)$ is a given density function.  We point out again that one of the primary contributions of our result is that $\mathring{u}$ is allowed to be \emph{arbitrary}.  This is in marked contrast to the existing literature which requires that the velocity is constant upstream and downstream.  

Rather than using $\mathring{u}$, however, it will prove more convenient to fix a (scaled) asymptotic relative velocity $u^*\colon [-d, 0] \to \R_+$ and consider the family
\begin{equation}
  \label{one parameter shear}
  \mathring{u}(y) = c - Fu^*(y),
\end{equation}
where $F > 0$ is a dimensionless parameter which we will call the \emph{Froude number} (cf., \eqref{Ustar normalization} and \eqref{normalize} for the complete definition).  The positivity of $u^*$ is consistent with the lack of horizontal stagnation \eqref{nostagnation}.  In Section~\ref{nondim sec}, we switch to dimensionless variables so that $F$ is the only parameter appearing in the problem; we think of it as a dimensionless wave speed.  It will later be proved that there exists a critical Froude number, denoted $\Fcr$, that 
plays an important role in determining the structure of solutions.  
For constant density irrotational solitary waves, $\Fcr = 1$, but in the present context its definition is given in \eqref{def Fcr}.  We say that a solution with $F > \Fcr$ is \emph{supercritical}.   

Observe that the conservation of mass \eqref{mass} implies that the density is transported by the flow.  By fixing $\mathring{\varrho}$, we therefore determine the density throughout the fluid region once the streamlines are known.  

\begin{figure}
  \includegraphics[scale=1.1]{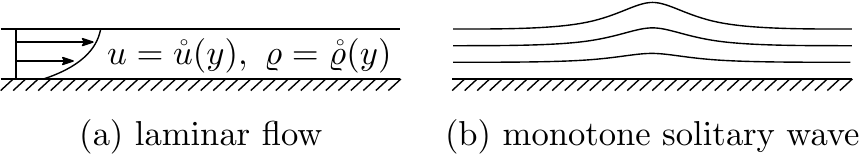}
  \caption{}
  \label{streamline figure}
\end{figure}

Finally, let us introduce some terminology for describing the qualitative features of these waves. A traveling wave is called \emph{laminar} or \emph{shear} if all of its streamlines are parallel to the bed. We say that a solitary wave is a \emph{wave of elevation} provided that at each vertical cross section of the domain, the height of every streamline (except the one corresponding to the bed) lies above its limiting height as $\abs x \to \infty$.  In particular, this means that $\eta$ is strictly positive.  A traveling wave is said to be \emph{symmetric} provided $u$ and $\eta$ are even in $x$ while $v$ is odd.  We say a symmetric wave of elevation is \emph{monotone} if the height of every streamline (except the bed) is strictly decreasing on either side of the crest line $\{x=0\}$; see Figure~\ref{streamline figure}.

\subsection{Statement of results}\label{results section}

\subsubsection*{Existence theory}
The contributions of this paper come in two parts.  The first of these is a complete large-amplitude existence theory for stratified solitary waves with an arbitrary (smooth and stable) density and smooth upstream velocity.
\begin{theorem}[Existence of large-amplitude solitary waves] \label{main existence theorem}
  Fix a H\"older exponent $\alpha \in (0,1/2]$, wave speed $c > 0$, gravitational constant $g >0$, asymptotic depth $d>0$, density function $\mathring{\varrho} \in C^{2+\alpha}([-d,0], \R_+)$, and positive asymptotic relative velocity $u^* \in C^{2+\alpha}([-d,0], \R_+)$. 
  There exists a continuous curve
  \begin{align*}
    \cm = \left\{ (u(s), v(s), \eta(s), F(s)) : s \in (0,\infty) \right\} 
  \end{align*}
of solitary wave solutions to \eqref{euler}--\eqref{upstream condition} with the regularity
  \begin{align}
    \label{(u,v,eta) regularity} 
    (u(s), v(s), \eta(s)) \in 
    C^{2+\alpha}(\overline{\Omega(s)}) \times C^{2+\alpha}(\overline{\Omega(s)}) \times C^{3+\alpha}(\R),
  \end{align}
  where $\Omega(s)$ denotes the fluid domain corresponding to $\eta(s)$.  
  The solution curve $\cm$ has the following properties.
  \begin{enumerate}[label=\rm(\alph*)] 
  \item {\rm (Extreme wave limit)} Following $\cm$, we encounter waves that are arbitrarily close to having points of (horizontal) stagnation:
    \begin{equation}
      \label{c-u to 0} 
      \lim_{s \to \infty} \inf_{\Omega(s)}  |c-u(s)| = 0.
    \end{equation}
  \item {\rm (Critical laminar flow)} The left endpoint of $\cm$ is a critical laminar flow, 
    \begin{align*}
      \lim_{s \to 0} (u(s),v(s), \eta(s), F(s)) = (c-\Fcr u^*,0,0,\Fcr).
    \end{align*}
    Here $\Fcr$ is the critical Froude number defined in \eqref{def Fcr}.  
  \item {\rm (Symmetry and monotonicity)} Every solution in $\cm$ is a wave of elevation that is symmetric, monotone, and supercritical.
  \end{enumerate}
\end{theorem}

\begin{remark}
  (i) When we say that $\cm$ is continuous, we in particular mean that the mapping $(0,\infty) \ni s \mapsto \eta(s) \in C^{3+\alpha}(\R)$ is continuous. In fact, we will show in Section~\ref{global bifurcation section} that, near any parameter value $s_0$, we can reparameterize $\cm$ so that $s \mapsto \eta(s)$ is locally real-analytic.

  The situation is more subtle for the velocity field $(u(s),v(s))$, which is defined on the $s$-dependent domain $\Omega(s)$. In Section~\ref{height equation section}, we will introduce a diffeomorphism (the Dubreil-Jacotin transform) depending on $(u(s),v(s),\eta(s))$ that sends $\Omega(s)$ to a fixed rectangular strip $R$. Composing the velocity field with the inverse of this diffeomorphism gives a continuous map  from the interval $(0,\infty)$ to the fixed function space $C^{2+\alpha}(\overline R) \times C^{2+\alpha}(\overline R)$. One consequence of this regularity is that $u(s)(x,y)$ and $v(s)(x,y)$ are jointly continuous in $(x,y,s)$ on the domain $\{(x,y,s) \in \R^3 : s >0,\ -d \le y \le \eta(s)(x)\}$. 

  (ii) In the special case of constant density, Theorem~\ref{main existence theorem} recovers and improves upon the main result of \cite{wheeler2013solitary,wheeler2015froude}.  For some choices of $u^*$, those earlier papers were forced to include the possibility that $F(s) \to \infty$ while $c-u(s)$ remains uniformly bounded away from $0$.  Here, using the qualitative results described below, we are able to give the definitive statement that \eqref{c-u to 0} holds for any $u^*$.
\end{remark}

\subsubsection*{Qualitative theory}

The second part of our results concerns the qualitative properties of solitary stratified waves.  These are used at critical junctures in the argument leading to Theorem~\ref{main existence theorem}, but are of considerable interest in their own right.  Several of them, in fact, improve substantially on the state-of-the-art for homogeneous flows.   For the time being, we give slightly weaker statements than what is eventually proved because the optimal versions are best made using a reformulation of the problem introduced in the next section.  

First, we establish a lower bound on $P$ and an upper bound on $(u,v)$ in terms of the given quantities $F$, $u^*$, $\mathring{\varrho}$, $g$, and $d$.  To our knowledge, these are the only estimates of this type for stratified steady waves; for constant density rotational waves, analogous bounds were obtained by Varvaruca in~\cite{varvaruca2009extreme}.    

\begin{proposition}[Bounds on velocity and pressure] \label{intro: bound v and P theorem} 
  The pressure and velocity fields for any solitary wave satisfy the bounds:
  \begin{gather*}
    P - P_{\mathrm{atm}} + M F\psi
    \geq 0
    \quad 
    \textup{and}
    \quad
    (u-c)^2 + v^2 \leq CF^2
    \qquad \textup{ in } \overline{\Omega},
  \end{gather*}
  where $\psi$ is the 
  pseudo stream function defined uniquely by $\nabla^\perp \psi = \sqrt{\varrho} (u-c, v)$ and $\psi|_{y = \eta} = 0$, and the constants $C$ and $M$ depend only on $u^*$, $\mathring{\varrho}$, $g$, $d$, and a lower bound for $F$.
\end{proposition}
See Proposition \ref{bound on velocity prop}, where this is proved in a non-dimensional form.  For more on the 
pseudo stream function $\psi$, see Section~\ref{stream function section}.

Second, we provide estimates for the Froude number from above and below.  Results of this type have a long history, going back at least to Starr~\cite{starr1947momentum}, who found sharp bounds for $F$ in the setting of homogeneous irrotational solitary waves; we refer the reader to the introduction to \cite{wheeler2015froude} for a detailed discussion and further references.

\begin{theorem}[Upper bound on $F$] 
  Let $(u,v,\eta,F)$ be a solution of \eqref{euler}--\eqref{upstream condition}. Then the Froude number satisfies the bound 
  \begin{equation*}
    F
    \le 
    \frac 1\pi 
    \frac{gd}{\min (u^*)^2} 
    \;
    \frac{\max \varrho}{\min \varrho}
    \;
    \frac{\sqrt{gd} }{\min_{\{x=0\}} (c-u)}.
  \end{equation*}
  Eliminating $F$ and $u^*$ in favor of $\mathring u$ yields
  \begin{align}
    \label{no more F}
    \frac{\min[\sqrt{\mathring\varrho} (c-\mathring u)]^2}{gd \min \varrho}
    \;
    \frac{\displaystyle\min_{\{x=0\}}[\sqrt\varrho (c- u)]}
    {\displaystyle\frac 1d \int_{-d}^0 \sqrt{\mathring\varrho} (c - \mathring u)\, dy}
    \le \frac 1\pi \frac{\max\varrho}{\min\varrho}.
  \end{align}
\end{theorem}

See Theorem~\ref{upper bound on F theorem}. Observe that the first two factors on the left-hand side of \eqref{no more F} are dimensionless measures of how close the flow comes to horizontal stagnation at $\pm\infty$ and along the crest line $\{x=0\}$, respectively. This is the first such estimate for waves with vorticity (with or without density stratification) that makes no additional assumptions on the shear profile $u^*$; in \cite{wheeler2015froude}, Wheeler established upper bounds for $F$ that are independent of $\inf_{\{x=0\}}(c-u)$, but impose further requirements on $u^*$. Thanks to Theorem~\ref{upper bound on F theorem}, we can avoid making similar restrictions in Theorem~\ref{main existence theorem}. 
Under different hypotheses on $u^*$ than in \cite{wheeler2015froude}, Kozlov, Kuznetsov, and Lokharu also prove estimates which imply upper bounds for $F$ for constant density waves \cite{kozlov2015bounds}. Their bounds involve the amplitude $\max \eta$, and apply not just to solitary waves but also to periodic waves as well as waves which are neither periodic nor solitary. 

We also mention that, in the special case of homogeneous irrotational waves, the argument leading to Theorem~\ref{upper bound on F theorem} can be modified 
to prove the upper bounds obtained by Starr in \cite{starr1947momentum}.

As a lower bound on $F$, in part we prove the following.

\begin{theorem}[Critical waves are laminar]\label{intro crit theorem}
  Let $(u,v,\eta,F)$ be a solitary wave solution of \eqref{euler}--\eqref{upstream condition}. If $F = \Fcr$, then $(u,v,\eta) = (c-\Fcr u^*, 0, 0)$. That is, there exist no nonlaminar solitary waves with critical Froude number.   
\end{theorem} 

In addition, we provide a partial characterization of waves of elevation in terms of the Froude number; for the complete statement, see Theorem~\ref{froude lower theorem} and the remark afterward.  Theorem~\ref{intro crit theorem} in particular implies that a continuous curve of supercritical waves cannot limit to a subcritical wave without passing through a laminar flow. The nonexistence of critical solitary waves with constant density was shown by Wheeler~\cite{wheeler2015froude} and then generalized to arbitrary constant density waves by Kozlov, Kuznetsov, and Lokharu~\cite{kozlov2015bounds}. The bound $F > \Fcr$ for waves of elevation was obtained by Wheeler~\cite{wheeler2015froude} under the assumption of constant density and Amick and Toland~\cite{amick1981solitary} and McLeod~\cite{mcleod1984froude} under the further assumption of irrotationality. Very recently, Kozlov, Kuznetsov, and Lokharu have given an essentially complete description of all constant density waves with near-critical Bernoulli constant \cite{kozlov2015conjecture}.

\begin{figure}
  \includegraphics[scale=1.1]{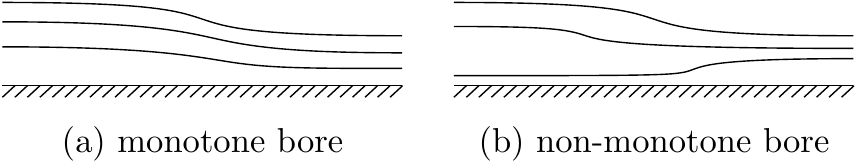}
  \caption{}
  \label{bores figure}
\end{figure}

A \emph{bore} is a traveling wave that limits to distinct laminar flows as $x \to \pm \infty$; see Figure~\ref{bores figure}. They are observed in nature (see, e.g., \cite{stoker1992water}) and have been computed numerically in various regimes (cf., e.g., \cite{turner1988broadening,lamb1998conjugate,grue2000breaking,grue2002solitary}).  Rigorous existence results for bores in multi-fluid channel flows have been obtained by Amick \cite{amick1989small}, Makarenko \cite{makarenko1992bore}, and Tuleuov \cite{tuleuov1997bore}.  
Bores play a special role in the global bifurcation theory analysis leading to Theorem~\ref{main existence theorem}.  In particular, many previous studies of stratified waves include the existence of bores as an alternative to the extreme wave limit (cf., e.g., 
\cite{amick1986global}
where they are referred to as ``surges''); this should be viewed as an indication of the system's lack of compactness.   
However, with a free upper surface, no bores exist with the property that the asymptotic height of all streamlines upstream lie at or below their asymptotic height downstream.  This is the content of the next theorem.

\begin{theorem}[Nonexistence of monotone bores] \label{intro: bores theorem}  
  Suppose that $(u,v, \eta)$ is a solution of \eqref{euler}--\eqref{nostagnation} which is a bore in the sense that 
  \begin{equation*}
    (u(x,\placeholder), v(x,\placeholder), \eta(x)) \to (\mathring{u}_\pm(\placeholder), 0, \eta_\pm), \qquad \textup{as } x \to \pm \infty
  \end{equation*}
  pointwise, where $\eta_\pm > -d$ are constants and $\mathring{u}_\pm \in C^1([-d, \eta_\pm])$.  If the limiting height of each streamline at $x = -\infty$ is no greater (or no less) than the limiting height of the same streamline at $x = \infty$, then in fact $\eta_+ = \eta_-$ and $\mathring{u}_+ \equiv \mathring{u}_-$.  
In particular, for any bounded solution of \eqref{euler}--\eqref{nostagnation}, $v$ must change signs unless it vanishes identically.  
\end{theorem}

A stronger version of Theorem~\ref{intro: bores theorem} is given in Theorem~\ref{bores theorem}; we also generalize it to include multiple fluid flows in Corollary~\ref{weak bore corollary}.  To the best of our knowledge, the nonexistence of monotone bores with a free upper surface has never been previously recorded, which is somewhat surprising in view of the large number of works devoted to studying bores in similar physical regimes.  Even in the relatively simple case of two homogeneous irrotational layers with a free surface, the usual calculations seem highly intractable 
(cf.,  \cite[Appendix~A]{dias2001free}). It was proved by Wheeler that bores (not necessarily monotone) do not exist for free surface solitary waves with constant density \cite{wheeler2015pressure}, but this argument seems to break down entirely for highly stratified waves. Moreover, even in the special case of constant density, our method is considerably more direct for monotone bores.  

Finally, we prove the following theorem characterizing the monotonicity and symmetry properties of stratified waves of elevation. 

\begin{theorem}[Symmetry] \label{intro: symmetry theorem}   
  Let $(u,v, \eta,F)$ be a supercritical wave of elevation that solves \eqref{euler}--\eqref{nostagnation} with $\n u_{C^2(\Omega)}, \n v_{C^2(\Omega)}, \n \eta_{C^3(\R)} < \infty$. Suppose that 
  \begin{equation*}
    (u,v) \to (\mathring{u},0),
    \quad
    (\nabla u,\nabla v) \to (\nabla \mathring{u},0)
     \qquad \textup{uniformly as } x \to +\infty \textup{ (or as $x \to -\infty$)}.
  \end{equation*}
  Then, after a translation, the wave is a symmetric and monotone solitary wave.
\end{theorem}
The symmetry of steady water waves has been a very active subject of research.   Indeed, the reformulated problem will turn out to be an elliptic PDE, and hence Theorem \ref{intro: symmetry theorem} falls into the larger category of results on the symmetry and monotonicity of 
positive solutions to elliptic systems.   More direct antecedents for water waves are given by Craig and Sternburg~\cite{craig1988symmetry}; Maia~\cite{maia1997symmetry}; Constantin and Escher~\cite{constantin2004symmetry}; Hur~\cite{hur2008symmetry}; Constantin, Ehrnstr\"om and Wahl\'en~\cite{constantin2007symmetry};  and Walsh~\cite{walsh2009symmetry}.  These papers are set in different physical regimes, but are all built around the method of moving planes (cf.\ \cite{gidas1979symmetry,li1991monotonicity}).  

Compared to this body of work, Theorem \ref{intro: symmetry theorem} has two distinctive features. First, we only impose asymptotic conditions upstream (or downstream), but nevertheless can conclude evenness and monotonicity. Typical moving-plane arguments begin with the far more restrictive assumption that the solutions decay in both directions. Moreover, due to the stratification, the elliptic problem that we are forced to consider has a zeroth order term with adverse sign, which significantly complicates the procedure. For the full statement and further discussion, see Theorem~\ref{symmetry theorem}.

In addition to the above theorems, a number of propositions concerning more refined monotonicity properties of solitary stratified waves are given in Section~\ref{nodal section}.  These may be of some broader interest, but are primarily important for their application in proving Theorem~\ref{main existence theorem}.  

\subsection{History of the problem} \label{history section}

The first rigorous constructions of exact nonlinear steady water waves were made in the 1920s by Nekrasov \cite{nekrasov1921steady} and Levi-Civita \cite{levi1924determinazione}. They considered the case of irrotational periodic water waves of infinite depth using conformal mappings and power series expansions. Their techniques rely heavily on both irrotationality and the smallness of the amplitude. Large-amplitude periodic irrotational waves were first constructed by Krasovski\u{\i} \cite{krasovskii1962theory}, and later by Keady and Norbury \cite{keady1978existence}, who used global bifurcation theory. These solutions were further studied by Toland \cite{toland1978existence} and McLeod \cite{mcleod1997stokes}, eventually leading to the proof of the Stokes conjecture (\cite{stokes1880theory}) by Amick, Fraenkel, and Toland \cite{amick1982stokes}. 

Solitary waves present a greater technical challenge to study. This is apparent even in the small-amplitude regime: the linearized operator at a critical laminar flow fails to be Fredholm for solitary waves, making the local bifurcation theory analysis much more subtle than in the periodic case. The existence theory for small-amplitude irrotational solitary waves begins in the mid 1950s and early 60s. Lavrentiev~\cite{Lavrentiev1954} and Ter-Krikorov~\cite{terkrikorov1960existence} used a construction based on long-wavelength limits of periodic waves. Friedrichs and Hyers~\cite{friedrichs1954existence} introduced a more direct iteration method. Later Beale~\cite{beale1977existence} provided an alternative proof using the Nash--Moser implicit function theorem. Mielke \cite{mielke1988reduction} subsequently used spatial dynamics methods, in particular the center manifold reduction technique.  Compared to Nash--Moser iteration, and other similar approximation schemes, spatial dynamics gives a fuller qualitative description of the small-amplitude solutions.  For that reason, we shall use it in this paper for the small-amplitude existence theory.  

The construction of large-amplitude solitary waves is further complicated by the loss of compactness coming from the unbounded domain. Amick and Toland \cite{amick1981periodic,amick1981solitary} circumvent this difficulty by considering a sequence of approximate problems with better compactness properties,
 constructing global curves of solutions to these approximate problems, and then taking a limit using the Whyburn lemma \cite{whyburn1964topological}. Large-amplitude irrotational solitary waves are also constructed in \cite{benjamin1990solitary}.
 
Unlike the above works, our interest in this paper lies in stratified flows, which are typically rotational. This is a consequence of the fact that vorticity is generated by the component of the density gradient that is orthogonal to the pressure gradient. Let us first discuss the constant density rotational theory, which is already significantly more involved than the irrotational case. In the presence of non-constant vorticity, the complex analytic machinery underlying much of the irrotational theory is no longer applicable; instead, one is required to explore the dynamics inside the fluid domain. In 1934 Dubreil-Jacotin \cite{dubreil1934determination} was able to use a nonconformal coordinate transformation to construct small-amplitude periodic waves with vorticity. Much later, Constantin and Strauss obtained large-amplitude periodic solutions with a general vorticity distribution through a global-bifurcation-theoretic approach \cite{constantin2004exact}. 

The first rigorous construction of small-amplitude rotational solitary waves is due to Ter-Krikorov \cite{ter1962solitary,ter1963theorie}, followed by Hur~\cite{hur2008solitary}, who generalized the method of Beale~\cite{beale1977existence}, and Groves and Wahl\'en~\cite{groves2008vorticity} using spatial dynamics.  Quite recently, Wheeler 
developed an existence theory for large-amplitude rotational solitary waves by starting from the small-amplitude solutions of Groves and Wahl\'en, and then continuing them globally \cite{wheeler2013solitary,wheeler2015froude,wheeler2015pressure}. 
In contrast to Amick and Toland~\cite{amick1981solitary,amick1981periodic}, this is done without recourse to approximate problems, relying instead on a new 
global-bifurcation-theoretic technique that permits a lack of compactness at the cost of 
additional alternatives for the structure of the solution set.  These alternatives are then winnowed down using qualitative theory, specifically the symmetry and monotonicity properties of the waves and the nonexistence of monotone bores.
We further generalize this result in the present paper, obtaining an abstract global bifurcation principle that allows us to construct large-amplitude waves.  It is important to note that, compared to the homogeneous problem considered by Wheeler, the qualitative theory here presents significant additional difficulties.  Despite this, we are able to get a result that, specialized to the constant density case, is in fact \emph{stronger} than the main theorem of \cite{wheeler2013solitary,wheeler2015froude}.

Let us now discuss the literature on stratified steady waves.  While this paper is interested in the free upper surface problem, a large part of the study of inhomogeneous waves has focused on channel flows where the fluid domain is confined between two impermeable horizontal boundaries. Two-layered systems of this type have been considered in detail in numerous studies, e.g., Amick and Turner \cite{amick1986global}
and Sun \cite{sun1995existence},
while small-amplitude channel flows with continuous stratification were constructed by Ter-Krikorov \cite{ter1963theorie}, Turner \cite{turner1981internal}, Kirchg\"assner \cite{kirchgassner1982wavesolutions}, Kirchg\"assner and Lankers~\cite{kirchgassner1993structure}, James~\cite{james1997small}, and Sun~\cite{sun2002solitary}. 
Large-amplitude existence theory for continuously stratified channel flows was provided by Bona, Bose, and Turner~\cite{bona1983finite}, Amick~\cite{amick1984semilinear}, and Lankers and Friesecke \cite{lankers1997fast}.

The subject of the present work is the free surface problem where there is no rigid lid and the upper boundary is instead a surface of constant pressure. This has many implications for the qualitative properties of the waves, and complicates the governing equations by introducing a fully-nonlinear boundary condition (see \eqref{bernoulli}). 
The first rigorous small-amplitude existence results for the free surface periodic stratified water wave problem are due to Dubreil-Jacotin~\cite{dubreil1937theoremes};
Yanowitch~\cite{yanowitch1962gravity} used a different approach to obtain similar results. The existence of large-amplitude periodic stratified waves was proved much more recently by Walsh~\cite{walsh2009stratified}, adapting the ideas of Constantin and Strauss~\cite{constantin2004exact}. 
To our knowledge, there are no prior existence results for large-amplitude solitary stratified waves with a free upper surface.

We reiterate that all of the above existence theory for solitary stratified waves is restricted to the case where the flow is uniform upstream and downstream, corresponding to $\mathring{u}$ being a constant, which precludes many physically interesting situations that involve wave-current interactions in the far field. As far as we are aware, Theorem~\ref{main existence theorem} is the first existence result of any kind for solitary stratified flows with a free upper surface and a general $\mathring{u}$.

\subsection{Plan of the article}

Let us now briefly explain the overall structure of the paper and the main challenges ahead. 

The first of these is already apparent:  we are considering a free boundary problem.  In Section~\ref{formulation section}, therefore, we begin (after non-dimensionalizing) by performing a change of variables that fixes the domain.  This is done using the Dubreil-Jacotin transformation, which recasts the Euler system as a scalar quasilinear elliptic PDE with fully nonlinear boundary conditions; we can write it abstractly
as an operator equation 
\begin{equation}
  \F(w,F) = 0,\label{intro: operator equation} 
\end{equation}
where $w$ describes the deviation of the streamlines from their asymptotic heights, and $F$ is the Froude number.  In these new coordinates, the fluid domain $\Omega$ is mapped to a fixed infinite strip $R$.  Recall that the presence of vorticity (and stratification) mean that one cannot simply consider a nonlocal problem posed on the boundary, as is typical for irrotational waves.  

Density stratification is manifested in \eqref{intro: operator equation} as a zeroth order term whose sign violates the hypotheses of the maximum principle (cf.\ \eqref{height equation} and Theorem~\ref{max principle}).  As we discuss below, maximum principle arguments are crucial to proving the existence of large-amplitude waves, and thus the ``bad sign'' is a serious issue. Curiously, this means that it is actually simpler to work with unstable densities rather than the more physically relevant kind we consider here (see also \cite{kirchgassner1993structure} where this observation is made). 

The second major difficulty is the ``singularity'' of the bifurcation point, in particular that the linearized operator at the critical laminar flow $\F_w(0,\Fcr)$ is not Fredholm (see \cite{hur2008solitary}). This completely rules out using a standard Lyapunov--Schmidt reduction to construct small-amplitude waves as in the periodic case~\cite{walsh2009stratified}. 

Perhaps the most serious obstacle, though, is a lack of compactness. For periodic waves, one can use Schauder estimates to show that bounded sets of solutions are compact (and even more, that $\F$ is ``locally proper'') \cite[Section~4]{walsh2009stratified}. With our unbounded domain $R$, however, Schauder estimates are no 
longer sufficient. In particular, there is the possibility that there exists a sequence of solutions to \eqref{intro: operator equation} for which the crest becomes progressively flatter and longer, and which therefore has no convergent subsequences; see Figure~\ref{noncompact figure}.
\begin{figure}
  \includegraphics[scale=1.1]{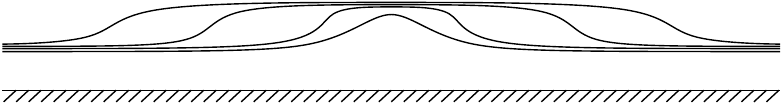} 
  \caption{A sequence of waves with increasingly broad and flat crests.}
  \label{noncompact figure}
\end{figure}

Proving Theorem~\ref{main existence theorem} therefore requires us to overcome the potential loss of the maximum principle due to the stratification, the singularity of the bifurcation point, and the loss of compactness. This process begins in Section~\ref{sec linearized}, where we establish some important preparatory results concerning the properties of the linearized operator $\F_w(w,F)$.  First, we investigate the linearized equation at a laminar flow (corresponding to $w \equiv 0$) and restricted to functions that are independent of $x$.  This leads to a Sturm--Liouville problem that, at the critical Froude number $\Fcr$, will have $0$ as its lowest eigenvalue with the remainder of the spectrum lying on the positive real axis.   Physically, $\Fcr$ divides the regimes of fast-moving (supercritical) waves that outrun all linear periodic waves and slow-moving (subcritical) waves that do not. We show that, in the supercritical regime, maximum principle arguments can be successfully carried out in many instances despite the adverse sign of the zeroth order term. More generally, we also prove that $\F_w(w,F)$ is a Fredholm operator with index $0$ when $F > \Fcr$.  This fact follows from the observation that the operator with coefficients evaluated at $x = \pm\infty$ is invertible (see \cite{wheeler2013solitary,volpert2003degree}).

Section~\ref{qualitative section} is devoted to the qualitative theory.    A priori bounds on the pressure, velocity, and Froude number are established using maximum principle arguments and several new integral identities.  
Later, these will be key to winnowing the alternatives that arise in the global bifurcation argument.  
The nonexistence of monotone bores is proved using a conjugate flow analysis, while the symmetry follows from an adapted moving planes argument.
Anticipating that these results may have interest beyond their applications to the existence theory, we 
have gathered them together in a single section rather than leaving them dispersed throughout the paper.

The next task, taken up in Section~\ref{small-amplitude section}, is to construct a family of small-amplitude solitary waves bifurcating from the background laminar flow.  As discussed above, the fact that $\F_w(0,F)$ is not Fredholm presents a serious obstruction.  Various strategies have been devised to get around this; see the discussion in Section~\ref{history section}. 
In this paper, we choose to employ spatial dynamics and the center manifold reduction method.  By treating the horizontal variable $x$ as time-like, we are able to further reformulate problem \eqref{intro: operator equation} as an infinite-dimensional Hamiltonian system.  As $F$ increases past $\Fcr$, the linearized operator has a pair of purely imaginary eigenvalues that collide at the origin and then become real. This gives rise to a two-dimensional center manifold and a corresponding reduced ODE system whose solutions can be lifted to bounded solutions of the full problem.  
For Froude numbers just above $\Fcr$, the reduced problem is, to leading order, the ODE satisfied by the solitons of the Korteweg--de Vries equation.  When $F$ is slightly subcritical, an analogous argument gives the existence of periodic stratified waves with periods limiting to infinity as $F \nearrow \Fcr$; these are heterogeneous waves of cnoidal type.  
Here we are working very much in the spirit of Groves and Wahl\'en \cite{groves2007spatial,groves2008vorticity}, who considered homogeneous density solitary waves with vorticity, and also \cite{wheeler2013solitary}, where the construction in \cite{groves2008vorticity} is linearized step by step to show that the resulting solutions are nondegenerate.

Ultimately, the center manifold analysis furnishes us with a family of small-amplitude solitary waves, $\cm_\loc$.  In Section~\ref{global bifurcation section}, we complete the proof of Theorem~\ref{main existence theorem} by continuing $\cm_\loc$ to a global curve $\cm$  using an adaptation of the method of Dancer \cite{dancer1973bifurcation,dancer1973globalstructure} and its generalization by Buffoni--Toland~\cite{buffoni2003analytic}.  
However, as mentioned above, one cannot directly apply this theory because it fundamentally requires that closed and bounded subsets of $\F^{-1}(0)$ be compact, and also that $\F_w$ be Fredholm of index $0$ at the bifurcation point.
Looking at the proofs in \cite{buffoni2003analytic}, we are first able to extract an abstract global bifurcation result, Theorem~\ref{generic global theorem}, that applies to a wider class of systems for which $\F_w$ may not be Fredholm at the bifurcation point and where $\F^{-1}(0)$ may not be locally compact (at the cost, of course, of a weaker conclusion). 
Following the strategy of \cite{wheeler2013solitary,wheeler2015pressure}, we show that, for general elliptic problems in cylinders, compactness can fail for a sequence of 
asymptotically monotone solutions only if a translated subsequence is locally converging to a monotone bore-type solution with different limits as $x \to \pm\infty$ (see Lemma~\ref{lem:compact:gen}). By applying our qualitative theory, we are finally able to exclude most of the possibilities in Theorem~\ref{generic global theorem} in our case, showing that $\cm$ can be continued up to stagnation.

Two appendices are also included.   Appendix~\ref{appendix calculations} contains  a number of proofs and calculations that are either largely standard or merely technical.  To keep the presentation self-contained, in Appendix~\ref{appendix quotes} we list several results from the literature that are used throughout the paper.   

\section{Formulation} \label{formulation section}

\subsection{Non-dimensionalization}\label{nondim sec}

In this subsection, we will choose characteristic length, velocity, and density scales in terms of the data $\mathring{\varrho}, u^*, d, g$.

The natural choice for the length scale is $d$. For the density scale we will use the density along the free surface $\{y=\eta(x)\}$:
\begin{align}
  \label{defvarrho0}
  \varrho_0 := \mathring{\varrho}(0).
\end{align}
To determine the velocity scale, we first normalize the $u^*$ appearing in \eqref{one parameter shear} so that it satisfies
\begin{align}
  \label{Ustar normalization}
  \int_{-d}^0 \sqrt{\mathring\varrho(y)}\, u^*(y) \, dy &= \sqrt{g\varrho_0 d^3}.
\end{align}
Next we consider the 
\emph{(relative) pseudo-volumetric mass flux} $m>0$ defined by
\begin{equation} 
  \label{defp0} 
  m := \int_{-d}^{\eta(x)} \sqrt{\varrho(x,y)}\left[c-u(x,y) \right] \, dy, 
\end{equation}
which is a constant independent of $x$. Sending $\abs x \to \infty$ in \eqref{defp0} and then using \eqref{upstream condition}, we find that $m$ is given in terms of $\mathring\varrho$, $F$, and $u^*$ by
\begin{align}
  \label{solitary m}
  m 
  &= 
  \int_{-d}^0 \sqrt{\mathring\varrho(y)} [c-\mathring{u}(y)] \, dy
  =
  F\int_{-d}^0 \sqrt{\mathring\varrho(y)} u^*(y) \, dy.
\end{align}
From \eqref{Ustar normalization} we see that $F$, $m$, and $\varrho_0$ are related by the simple formula
\begin{align}  \label{normalize}
  \frac{g\varrho_0 d^3}{m^2} = \frac 1{F^2},
\end{align}
and consequently the velocity scale can be chosen to be $m/(d\sqrt{\varrho_0})$.

\begin{subequations}\label{dimensionless}  
  Rescaling lengths we set
  \begin{align}
    (\tilde x, \tilde y) := \frac 1d (x,y),
    \qquad 
    \tilde \eta(\tilde x) := \frac 1d \eta(x).
  \end{align}
  Rescaling the density we define 
  \begin{align}
    \tilde \varrho(\tilde x, \tilde y) := \frac 1{\varrho_0} \varrho(x,y),
    \qquad 
    \tilde{\mathring \varrho}(\tilde y) := \frac 1{\varrho_0} \mathring \varrho(y),
  \end{align}
  and rescaling velocities we set
  \begin{gather} \label{velocity scalings}
    \begin{gathered}
      \tilde u(\tilde x, \tilde y)  := \frac {\sqrt{\varrho_0}d}m  u(x,y),
      \qquad
      \tilde v(\tilde x, \tilde y)  := \frac {\sqrt{\varrho_0}d}m v(x,y),
      \\
      \tilde c   := \frac {\sqrt{\varrho_0}d}m c,
      \qquad 
      \tilde{\mathring{u}}(\tilde y)  := \frac {\sqrt{\varrho_0}d}m \mathring{u}(y).
    \end{gathered}
  \end{gather}
  Finally, combining the length, density, and velocity scalings we set
  \begin{align}
    \tilde P(\tilde x, \tilde y) := \frac{d^2}{m^2} (P(x,y) - P_{\textrm{atm}}).
  \end{align}
\end{subequations}
In these variables, \eqref{euler} becomes
\begin{equation}
 \label{rescaled euler}
 \left\{   \begin{alignedat}{2}
   (\tilde u-\tilde{c}) \tilde\varrho_{\tilde x} + \tilde v \tilde \varrho_{\tilde y}  &=  0 \\
   \tilde\varrho (\tilde u-\tilde{c}) \tilde u_{\tilde x} + \tilde\varrho \tilde v \tilde u_{\tilde y}  & =  -\tilde P_{\tilde x}    \\ 
   \tilde\varrho (\tilde u-\tilde c) \tilde v_{\tilde x} + \tilde\varrho \tilde v \tilde v_{\tilde y} & =  -\tilde P_{\tilde y} - \frac 1{F^2} \tilde\varrho  \\
   \tilde u_{\tilde x} + \tilde v_{\tilde y} &=  0
 \end{alignedat} \right. \qquad \textrm{in } \tilde\Omega,
\end{equation}
where the rescaled fluid domain is
\begin{equation} 
  \label{def rescaled Omega} 
  \tilde\Omega := \{ (\tilde x,\tilde y) \in \R^2 : -1 < \tilde y < \tilde \eta(\tilde x) \}.
\end{equation}
Here the particularly simple formula for the dimensionless parameter appearing in 
\eqref{rescaled euler} is thanks to \eqref{normalize}.
The boundary conditions \eqref{eulerboundary} take the dimensionless form
\begin{equation}
 \label{rescaled eulerboundary} 
 \left\{ \begin{alignedat}{2}
   \tilde v &=  (\tilde u-\tilde c)  \tilde\eta_{\tilde x} &\qquad & \textrm{on } \tilde y = \tilde \eta(\tilde x),  \\
   \tilde v &=  0 && \textrm{on } \tilde y = -1,  \\
   \tilde P &=  0 &\qquad& \textrm{on } \tilde y = \tilde \eta(\tilde x),  
 \end{alignedat} \right. 
\end{equation}
and the asymptotic condition \eqref{upstream condition} becomes
\begin{equation} 
  \label{rescaled upstream condition} 
  (\tilde u,\tilde v) \to (\tilde{\mathring{u}},0),
  \quad 
  \tilde \varrho \to \tilde{\mathring\varrho},
  \quad
  \tilde \eta \to 0 
  \qquad \textrm{as } |\tilde x| \to \infty.
\end{equation}
For the solutions we construct, analogous convergence in fact holds for the first derivatives of the velocity field and density, and up to the second derivative for the surface profile.   
Combining \eqref{normalize} with \eqref{one parameter shear}, we see that 
\begin{align}
  \label{rescaled U}
  \tilde{\mathring{u}}(\tilde y) - \tilde c
  = - \frac 1{\sqrt{gd}} u^*(y),
\end{align}
so the asymptotic conditions \eqref{rescaled upstream condition} are independent of $F$.

To simplify notation we will from now on \emph{drop the tildes} on the dimensionless variables defined in \eqref{dimensionless} and the fluid domain $\tilde \Omega$.

\subsection{Stream function formulation} \label{stream function section}
The dimensionless Euler system \eqref{rescaled euler} can be turned into a scalar equation by introducing the {\it pseudo (relative) stream function} $\psi$ defined up to a constant by 
\begin{equation} 
  \psi_x = -\sqrt{\varrho}v,
  \qquad 
  \psi_y = \sqrt{\varrho} (u-c). 
  \label{defpsi} 
\end{equation}
It is apparent from these equations that the streamlines are level sets of $\psi$.  The kinematic condition  and impermeability condition in 
\eqref{rescaled eulerboundary} then imply that $\psi$ is constant on the free surface and bed; we normalize $\psi$ by setting $\psi = 0$ on $\{ y = \eta(x) \}$.  
Thanks to the definition of $m$ in \eqref{defp0} and the scaling \eqref{dimensionless}, we then have $\psi = 1$ on the bed $\{y = -1\}$.  Note also that the no stagnation condition \eqref{nostagnation} translates to
\begin{equation} 
  \psi_y < 0 \qquad \textrm{in } \overline{\Omega}. 
  \label{nostagnationpsi} 
\end{equation}

Conservation of mass \eqref{mass} implies that the density is constant on each streamline.  This permits us to introduce a \emph{streamline density function} $\rho\colon[-1,0] \to \R^{+}$ such that 
\begin{equation} 
  \label{defrho} 
  \varrho(x,y) = \rho(-\psi(x,y)) \qquad \textrm{in } \Omega.  
\end{equation}
For solitary waves, $\rho$ is determined by the flow upstream (or downstream) and hence, and soon as $u^*$ is fixed, we can recover $\rho$ from $\mathring{\varrho}$ and conversely.  In the remainder of the paper, we will use $\rho$ in place of $\mathring{\varrho}$, as it is more convenient when working in the new coordinates described below in Section~\ref{height equation section}.  Notice that stable stratification corresponds to assuming that $\rho^\prime \leq 0$.

Bernoulli's theorem states that the quantity 
\begin{equation} 
  \label{defE} 
  E := P + \frac{\varrho}{2}\left( (u-c)^2 + v^2\right) + \frac 1{F^2}\varrho y 
\end{equation}
is constant along streamlines, as can be verified by differentiation.
This and the absence of stagnation allows us to define the so-called \emph{Bernoulli function} $\beta\colon[0,1] \to \R$ by 
\begin{equation} 
  \label{defbeta} 
  \frac{dE}{d\psi}(x,y) = -\beta(\psi(x,y)) \qquad \textrm{in } \Omega. 
\end{equation}

It can be shown (see, for instance \cite[Lemma~A.2]{chen2016continuous}) that the Euler system \eqref{euler} under the assumption of no stagnation \eqref{nostagnation} is equivalent to Yih's equation (cf., \cite{yih1965dynamics}):
\begin{equation} 
  \label{yih equation} 
  \Delta \psi - \frac 1{F^2} y \rho^\prime(-\psi) + \beta(\psi) = 0 \qquad \textrm{in }  \Omega.
\end{equation}
Earlier versions of this equation were found by Dubreil-Jacotin~\cite{dubreil1937theoremes} and Long~\cite{long1953some}. The elegant form of \eqref{yih equation} is the main argument for using the pseudo-stream function defined by \eqref{defpsi} instead of the usual stream function where the factors of $\sqrt\varrho$ are absent. If $\varrho$ is constant, then \eqref{yih equation} reduces to the well-known semilinear equation $\Delta \psi = -\beta(\psi)$ for steady rotational water waves (with $\beta \equiv 0$ for irrotational waves).

Evaluating $E$ on the free surface, and recalling the normalization \eqref{velocity scalings}, the boundary conditions \eqref{eulerboundary} become
\begin{align} \label{bernoulli} 
  \left\{
  \begin{alignedat}{2}
    |\nabla \psi|^2 + \frac 2{F^2} \varrho (y+1) &= Q &\qquad& \textrm{on } y = \eta(x), \\
    \psi & = 0 && \textrm{on } y = \eta(x),  \\
    \psi &= 1 && \textrm{on } y = -1, 
  \end{alignedat} 
  \right.
\end{align}
where the constant
\begin{equation} 
  Q := 2\left(E + \frac 1{F^2} \varrho\right)\Big|_{y=\eta(x)}
  = \left( \varrho  (\mathring{u} - c)^2 + \frac 2{F^2}\varrho \right)\Big|_{y=0}.
  \label{defQ} 
\end{equation}
Finally, the asymptotic conditions \eqref{upstream condition} become
\begin{equation}
  \label{stream formulation asymptotics}
  \nabla \psi \to \left(0, \sqrt{\mathring{\varrho}}(\mathring{u} - c)\right), 
  \quad 
  \eta \to 0, 
  \quad
  \varrho \to \mathring\varrho,
  \qquad 
  \text{as }\ |x| \to \infty.
\end{equation}

\begin{remark}\label{remark bernoulli}
  For solitary waves, the behavior at infinity strongly constrains the near-field structure as well. In particular, the Bernoulli function $\beta$ can be explicitly computed in terms of the asymptotic data at $x = \pm\infty$.  
  Here, to compare our results to previous works in the literature, specifically those for periodic solutions, we
   derive an expression for the Bernoulli function $\beta$ in terms of $\mathring{u}$ and $\rho$.

  Let $\mathring{y}(p)$ be the limiting $y$-coordinate of the streamline $\{ \psi = -p \}$, and define
  \begin{align}
    \label{limiting u}
    \mathring{U}(p) := \mathring{u}(\mathring{y}(p)).
  \end{align}
  From \eqref{stream formulation asymptotics}, we see that $\mathring y$ and $\mathring U$ are related by
  \begin{equation}\label{limiting h}
    \mathring{y}(p) = \int_{-1}^p \frac{1}{\sqrt{\rho(s)}(c - \mathring{U}(s))} \, ds - 1.
  \end{equation}
  Sending $\abs x \to \infty$ in \eqref{yih equation} and applying \eqref{stream formulation asymptotics}, we see that $\beta$ is given by
  \begin{align*}
    \label{Beta explicit}
    \begin{aligned}
      \beta(-p)
      &= \frac 1{F^2} \mathring y \rho_p - \partial_y [ \sqrt\rho (\mathring U - c) ] = \left( \frac 1{F^2} \mathring y - \frac 12 (\mathring U - c)^2 \right) \rho_p
      + \rho (\mathring U - c) \mathring U_p.
    \end{aligned}
  \end{align*}
\end{remark}

\subsection{Height function formulation}\label{height equation section}

Although Yih's equation is scalar, it is still posed on an a 
priori unknown domain.  This presents a serious technical challenge for the existence theory (though not necessarily the qualitative theory) and thus we will make a change of variables that fixes the domain.  Specifically, we use the  Dubreil-Jacotin transformation \cite{dubreil1937theoremes} 
\begin{equation}
  (x,y) \longmapsto (x, -\psi(x,y)) =: (q,p),
\end{equation}
which sends the fluid domain $\Omega$ to a rectangular strip
\[
  R := \{ (q, p) \in \R^2 : \ p \in (-1, 0) \}.
\]
The free (``top'') surface $\{y=\eta(x)\}$ is mapped to $T := \{ p = 0 \}$ and the image of the bed is $B := \{ p = -1 \}$. The new coordinates $(q,p)$ are often called \emph{semi-Lagrangian variables}.  

It is also convenient to work with the new unknown $h = h(q,p)$ which gives the height above the bed of the point $(x,y) \in \overline{\Omega}$ with $x = q$ and lying on the streamline $\{\psi = -p\}$, 
\begin{equation} 
  h(q,p) := y + 1. 
  \label{defheight} 
\end{equation}
We call $h$ the \emph{height function}.  From \eqref{defheight}, it is clear that $h$ must be positive in $R \cup T$ and vanish on the bed $B$.  
In terms of $h$, the no stagnation condition now reads
\begin{equation}
  \label{no stag h}
  h_p > 0.
\end{equation}
Likewise, the asymptotic conditions for $h$ have the form
\begin{equation}\label{h asymptotics}
  h(q, p) \to H(p)
  \quad 
  h_q(q, p) \to 0, 
  \quad 
  h_p(q, p) \to H_p(p), 
  \quad 
  \text{as } |q|\to \infty.
\end{equation}

A simple calculation shows that we have the change-of-variables identities
\begin{align}
  \label{h to uv}
  h_q = \frac v{u-c}, 
  \quad 
  h_p = \frac 1{\sqrt{\varrho}(c-u)},
\end{align}
where the left-hand side is evaluated at $(q,p)$ while the right-hand side is evaluated at $(x,y) = (q,h-1)$. Combining \eqref{limiting h} with \eqref{limiting u} and \eqref{rescaled U}, we see that the upstream (and downstream) height function $H(p)$ is the solution to the ODE
\begin{equation}\label{ODE H}
  \left\{
  \begin{aligned}
    &H_p(p) = \left. 
    \frac 1
    {\sqrt{\mathring\varrho}(c-\mathring{u})}
    \right|_{y = H(p) - 1} 
    \quad \text{in }  -1 < p < 0 ,\\
    &H(-1) = 0, \qquad H(0) = d.
  \end{aligned}
  \right.
\end{equation}

One can show that Yih's equation \eqref{yih equation} and the boundary conditions \eqref{bernoulli} are equivalent to the following quasilinear PDE for $h$:
\begin{equation} \label{height equation}
  \left\{ 
  \begin{alignedat}{2}
    \left(-\frac{1+h_q^2}{2h_p^2} + \frac{1}{2H_p^2}\right)_p + \left( \frac{h_q}{h_p} \right)_q - \frac{1}{F^2}  \rho_p (h-H) &= 0 
    &\qquad& \textrm{in } R, \\
    \frac{1+h_q^2}{2h_p^2} - \frac{1}{2H_p^2} + \frac{1}{F^2}  \rho (h-1) &= 0 
    && \textrm{on } T, \\
    h &= 0 && \textrm{on } B,
  \end{alignedat} 
  \right. 
\end{equation}
see \cite{chen2016continuous}.
One can check that \eqref{height equation} is a uniformly elliptic PDE when $\inf_R h_p > 0$, with a uniformly oblique boundary condition. The sign of the zeroth order coefficient in the first equation means that \eqref{height equation} does not satisfy the hypotheses of the maximum principle (Theorem~\ref{max principle}).

\subsection{Flow force}\label{flowforce section}
The \emph{flow force}
\begin{equation}
  \flowforce(x) := \int_{-1}^{\eta(x)} \left[ P+ \varrho (u-c)^2 \right] \, dy \label{Eulerian S} 
\end{equation}
is an important quantity to consider for steady waves.
It is straightforward to check that $\flowforce$ is independent of $x$.
In semi-Lagrangian variables, we can equivalently write
\begin{equation}
  \label{DJ S}
  \flowforce := 
  \int_{-1}^0
  \left[ \frac{1-h_q^2}{2h_p^2} + \frac 1{2H_p^2}
  - \frac 1{F^2}\rho(h-H) - \frac 1{F^2} \int_0^p \rho H_p\,
  dp' \right]\, h_p\, dp.
\end{equation}
Since $H$ is fixed throughout the paper, we will view $\flowforce=\flowforce(h)$ as a functional acting on $h$.

In the spatial-dynamics formulation presented in Section~\ref{small-amplitude section},  $\flowforce$ essentially serves as the Hamiltonian.
Later, in Section~\ref{bore section}, we will be interested in $q$-independent height functions $h = K(p)$ which solve \eqref{height equation} and have $\flowforce(K) = \flowforce(H)$.
In this case we will say that $K$ and $H$ are
\emph{conjugate} (cf., \cite{benjamin1971unified, keady1975conjugate, keady1978conjugate}). 

\subsection{Function spaces and the operator equation}\label{subsec_spaces}

Finally, let us introduce the precise formulation of the problem and in particular the function spaces we will be working in. For a (possibly unbounded) domain $D\subset \R^n$, $k$ a nonnegative integer and $\alpha\in [0,1)$, we denote
\begin{align*}
  C^\infty_{\mathrm c}(D) &:= \left\{ \phi\in C^\infty(D):\ \textrm{the support of }\phi \text{ is a compact subset of } D \right\},\\
  C^{k+\alpha}(D) &:= \left\{ u\in C^k(D):\ \|\phi u\|_{C^{k+\alpha}}<\infty \ \text{ for all } \phi\in C^\infty_{\mathrm c}(D) \right\}.
\end{align*}
The spaces $C^{\infty}_{\mathrm c}(\overline{D})$ and $C^{k+\alpha}(\overline D)$ are defined analogously. 
Furthermore, let
\[
  C^{k+\alpha}_\bdd(\overline D) := \left\{ u\in C^k(D):\ \|u\|_{C^{k+\alpha}}<\infty \right\},
\]
which is a Banach space when equipped with the obvious norm. We also consider the closed subspace
\[
  C^{k+\alpha}_0(\overline D) := \left\{ u\in C^{k+\alpha}_\bdd(\overline D):\ \lim_{r\to\infty}\sup_{|x| = r} |D^j u(x)| = 0 \text{ for } 0\leq j \leq k \right\},
\]
of $C^{k+\alpha}_\bdd(\overline D)$, which is a Banach space under the $C^{k+\alpha}_\bdd(\overline D)$ norm. Finally, we write 
\[
u_n \to u \text{ in } C^{k+\alpha}_\loc(\overline D) 
\ \iff\ 
 \|\phi(u_n - u)\|_{C^{k+\alpha}(D)} \to 0 \text{ for all } \phi\in C^{\infty}_{\mathrm c}(\overline{D}).
\]

Introducing the difference 
\[
w(q, p) := h(q, p) - H(p)
\]
between the height function $h$ and its asymptotic value $H$ at $|q| = \infty$,
the asymptotics conditions \eqref{h asymptotics} become simply $w \in C^1_0(\overline R)$. Note also that $\eta(q) = w(q, 0)$. 
\begin{subequations}\label{w equation}
  In terms of $w$, the height equation \eqref{height equation} becomes
  \begin{equation} 
    \left\{ 
    \begin{alignedat}{2}
      \left(
      -\frac{1+w_q^2}{2(H_p+w_p)^2} 
      + \frac{1}{2H_p^2}
      \right)_p 
      + \left( \frac{w_q}{H_p+w_p} \right)_q 
      - \frac{1}{F^2}  \rho_p w &= 0 &\qquad& \textrm{in } R, \\
      \frac{1+w_q^2}{2(H_p+w_p)^2} - \frac{1}{2H_p^2} + \frac{1}{F^2}  \rho w &= 0 && \textrm{on } T, \\
      w &= 0 && \textrm{on } B.
    \end{alignedat} 
    \right. 
  \end{equation}
  The no stagnation assumption \eqref{nostagnation} and the boundedness of $u$ translate to
  \begin{equation}\label{lower bound w_p}
    0 < \inf_R(H_p + w_p) < \infty.
  \end{equation}
\end{subequations}
 
Define the Banach spaces $X$ and $Y = Y_1 \times Y_2$ by 
\begin{equation}
  \label{def X Y spaces} \begin{split}
    & X := \left\{ w \in C^{3+\alpha}_{\bdd,\even}(\overline R) \cap C^2_0(\overline R) : w = 0 \text{ on } B \right\},
    \\
    & Y_1 := C^{1+\alpha}_{\bdd,\even}(\overline R) \cap C^0_0(\overline R),\qquad 
    Y_2 := C^{2+\alpha}_{\bdd,\even}(T) \cap C^1_0(\overline R),
  \end{split} 
\end{equation}
where the subscript ``e'' denotes evenness in $q$. From~\eqref{h to uv} it is clear that the evenness of $w$ in $q$ is equivalent to the symmetry discussed in the paragraph before Section~\ref{results section}. 
We write \eqref{w equation} as an operator equation
\begin{subequations}\label{operator equation}
\begin{equation}\label{F equation}
\F(w,F) = 0,
\end{equation} 
where 
\begin{align*}
  \F = (\F_1,\F_2)  \maps U \subset X\times \R  \longrightarrow Y
\end{align*}
is given by 
\begin{align}
  \F_1(w,F) &:= 
  \left(
  -\frac{1+w_q^2}{2(H_p+w_p)^2} 
  + \frac{1}{2H_p^2}
  \right)_p 
  + \left( \frac{w_q}{H_p+w_p} \right)_q 
  - \frac{1}{F^2}  \rho_p w, \label{F_1}\\
  \F_2(w,F) &:= 
  \left(
\frac{1+w_q^2}{2(H_p+w_p)^2} - \frac{1}{2H_p^2} + \frac{1}{F^2}  \rho w \right)\bigg|_T. \label{F_2}
\end{align}
\end{subequations}
Here we seek solutions in the open subset 
\begin{equation}\label{F domain}
  U := 
  \left\{(w,F) \in X \times \R : 
  \inf_R(H_p + w_p) > 0,\ F > \Fcr \right\} \sub X \by \R,
\end{equation}
where ${\Fcr}$, the {\it critical Froude number}, will be defined in \eqref{def Fcr}. 
Since $\F$ is a rational function of $w$ and its derivatives, one easily checks that it is (real-)analytic with domain $U$ and values in $Y$.
All of this remains true when $\Fcr$ in the definition of $U$ is replaced by $0$; 
we will occasionally use this extended definition of $\F$.

\section{Linearized operators} \label{sec linearized}

In this section we prove several lemmas about the linearized operators $\F_w(w,F)$. The results in Section~\ref{sec SL} will be needed in Section~\ref{qualitative section}, while the work done in Section~\ref{fredholm section} lays the foundation for the existence theory in Sections~\ref{small-amplitude section} and \ref{global bifurcation section}.

\subsection{Sturm--Liouville problems} \label{sec SL}
We begin by studying the eigenvalue problem for the 
linearized operator $\F_w(0,F)$ around a laminar flow. 
Restricting to $q$-independent functions, this becomes the Sturm--Liouville problem
\begin{align}
  \label{SLprob2} \left\{
  \begin{alignedat}{2}
    \Big( \frac{\dot w_p}{H_p^3} \Big)_p
    - \mu \rho_p \dot w
    &= - \nu \frac{ \dot w}{H_p}
    &\qquad& \text{in } -1 < p < 0, \\
    \dot w &= 0 && \text{on } p=-1, \\
    -\frac{\dot w_p}{H^3_p} + \mu\rho \dot w &= 0 
    && \text{on } p = 0,
  \end{alignedat} \right.
\end{align}
where we have introduced the shorthand $\mu := 1/F^2$ and $\nu$ is the eigenvalue.

First, we consider the situation for $\nu = 0$, characterizing the smallest positive value of $\mu$ for which the system
\begin{equation}\label{SLprob} \left\{
  \begin{alignedat}{2}
    \Big(\frac{\dot w_p}{H_p^3}\Big)_p - \mu\rho_p \dot w &=0
    &\qquad& \text{in }-1 < p <0,\\
    \dot w &= 0 && \text{on } p = -1,\\
    -\frac{\dot w_p}{H_p^3} + \mu\rho \dot w &= 0 && \text{on }p = 0,
  \end{alignedat} \right.
\end{equation}
has a nontrivial solution $\dot w$. One can approach this task using variational methods, but for later convenience we instead give an argument in terms of the (unique) solution $\Phi = \Phi(p;\mu)$ of the initial value problem
\begin{equation}\label{def Phi} \left\{
  \begin{alignedat}{2}
    \Big(\frac{\Phi_p}{H_p^3}\Big)_p - \mu\rho_p \Phi &=0
    &\qquad& \text{in } -1 < p <0,\\
    \Phi &= 0 && \text{on } p = -1,\\
    \Phi_p &= 1 && \text{on } p = -1,
  \end{alignedat} \right.
\end{equation}
together with the related function
\begin{align*}
  \label{qprime def}
  \qprime(\mu) := 
  -\frac{\Phi_p(0;\mu)}{H_p^3(0)} + \mu\rho(0) \Phi(0;\mu),
\end{align*}
which is defined so that $\Phi$ solves \eqref{SLprob} for a given $\mu$ if and only if $\qprime(\mu) = 0$. Note that in the constant density case where $\rho_p \equiv 0$, $\Phi$ is independent of $\mu$ and hence $A$ is simply an affine function.

\begin{lemma}[Existence of critical Froude number] \label{lem_minimality}  There exists a unique $\mucr > 0$ such that the following holds.
  \begin{enumerate}[label=\rm(\alph*)] 
  \item \label{lem_minimality_existence} The Sturm--Liouville problem \eqref{SLprob} has a nontrivial solution $\dot w $ for $\mu = \mucr$. Moreover, we can take $\dot w = \Phicr(p) := \Phi(p;\mucr)$.
  \item \label{lem_minimality_pos} If $0 \le \mu \le \mucr$, then $\Phi(p;\mu) > 0$ for $-1 < p \le 0$ and $\Phi_p(p;\mu) > 0$ for $-1 \le p \le 0$.
  \item \label{lem_minimality_qprime} For $0 \le \mu < \mucr$, $\qprime(\mu) < 0$.
  \end{enumerate}
\end{lemma}
Recalling the definition of $\mu$, this gives us at last the definition of the critical Froude number 
\begin{equation}
  \Fcr := \sqrt{\frac{1}{\mucr}}. \label{def Fcr} 
\end{equation}
A simple calculation confirms that in the case of constant density irrotational flow, $\Fcr = 1$, as in the classical theory.  

\begin{proof}[Proof of Lemma~\ref{lem_minimality}]
  \begin{figure}
    \includegraphics[scale=1.1]{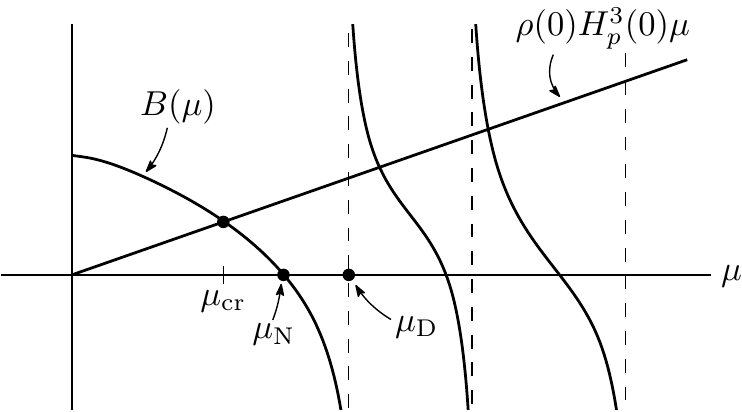}
    \caption{The function $B=B(\mu)$ from the proof of Lemma~\ref{lem_minimality}.}
    \label{figure minimality}
  \end{figure}
  Observe that \eqref{SLprob} has a nontrivial solution $\dot w \not\equiv 0$ if and only if $B(\mu) = \mu \rho(0) H_p^3(0)$, where
  \begin{align}
    \label{def B mu}
    B(\mu) := \frac{\Phi_p(0;\mu)}{\Phi(0;\mu)},
  \end{align}
  and that in this case we can take $\dot w = \Phi(\placeholder;\mu)$.
  (Note that, by uniqueness for the initial value problem, the numerator and denominator in \eqref{def B mu} cannot vanish simultaneously.)
  For $\mu = 0$ we compute
  \begin{align}
    \label{Phi 0}
    \Phi(p;0) &= \frac 1{H_p(-1)^3} \int_{-1}^p H_p(s)^3\, ds
  \end{align}
  and hence 
  \begin{align*}
    B(0) = \frac{H_p(0)^3}{\int_{-1}^0 H_p^3 \, dp} > 0.
  \end{align*}

  Clearly, $\Phi$ depends smoothly on $\mu$. Differentiating \eqref{def Phi} we see that its derivative $\Phi_\mu$ solves the inhomogeneous initial value problem
  \begin{equation}\label{Phi mu} \left\{
    \begin{alignedat}{2}
      \Big(\frac{\Phi_{\mu p}}{H_p^3}\Big)_p - \mu\rho_p \Phi _\mu &= \rho_p \Phi
      &\qquad& \text{in } -1 < p <0,\\
      \Phi_\mu &= 0 && \text{on } p = -1,\\
      \Phi_{p\mu} &= 0 && \text{on } p = -1.
    \end{alignedat} \right.
  \end{equation}
  Together \eqref{Phi mu} and \eqref{def Phi} lead to the Green's identity 
  \begin{equation} \label{Green identity Phi Phi mu} 
    \left( \frac{\Phi_{\mu p}\Phi}{H_p^3} - \frac{\Phi_p \Phi_\mu}{H_p^3} \right)\bigg|_{p=0} = \int_{-1}^0 \rho_p \Phi^2 \, dp. 
  \end{equation}

  Differentiating \eqref{def B mu} and using \eqref{Green identity Phi Phi mu}, we obtain
  \begin{align*}
    B_\mu = 
    \left( \frac{\Phi_{\mu p} \Phi - \Phi_p \Phi_\mu}{\Phi^2}\right)\bigg|_{p=0} 
    = \frac{H_p(0)^3}{\Phi(0)^2} \int_{-1}^0 \rho_p \Phi^2 \, dp <  0,
  \end{align*}
  provided, of course, that $\Phi(0) = \Phi(0;\mu) \ne 0$. 
  From this and the positivity of $B(0)$ it follows that there is a unique smallest $\mu = \mucr > 0$ such that $B(\mu) = \mu \rho(0) H_p^3(0)$ (see Figure~\ref{figure minimality}), and hence that \eqref{SLprob} has a nontrivial solution $\dot w = \Phicr$. The inequality $\qprime(\mu) < 0$ is a consequence of the inequality $B(\mu) > \mu \rho(0) H_p^3(0)$ for $0 \le \mu < \mucr$. We observe that $B(\mu)$, $\Phi(0;\mu)$, and $\Phi_p(0;\mu)$ are all strictly positive for $0 \le \mu \le \mucr$. 

  Next we will show that
  \begin{align}
    \label{uniform bound Phi}
    \Phi(p;\mu) > 0 \textup{ for $-1 < p \le 0$ and $0 \le \mu \le \mucr$}.
  \end{align}
  Toward that purpose, define the set
  \begin{align*}
    \mathcal{N} := \{ \mu \in [0,\mucr] :
    \Phi(p;\mu) > 0 \text{ for $-1 < p < 0$}\}.
  \end{align*}
  From \eqref{Phi 0} it is clear that $0 \in \mathcal{N}$. The continuous dependence of $\Phi$ on $\mu$, together with the positivity $\Phi_p > 0$ for $p=-1,0$ and $0 \le \mu \le \mucr$, show that $\mathcal{N}$ is a relatively open subset of $[0,\mucr]$. Seeking a contradiction, suppose that it is not closed. Then there exists a limit point $\bar\mu \in [0,\mucr] \setminus \mathcal{N}$. By definition, this means that there must be a point $\bar p \in (-1,0)$ where $\Phi(\bar p,\bar \mu) = 0$. On the other hand, as $\bar\mu$ is a limit point, we have
  $\Phi(p;\bar\mu) \ge 0$  for $-1 \le p \le 0$. 
  This further implies that $\Phi$ attains a local minimum at $\bar p$ and hence $\Phi_p(\bar p;\bar \mu) = 0$. As before, $\Phi(\bar p;\bar \mu) = \Phi_p(\bar p; \bar \mu) = 0$ forces $\Phi(\placeholder;\bar\mu) \equiv 0$ by uniqueness for the initial value problem, a contradiction.
  Thus we have proved that $\mathcal{N}$ is both relatively open and closed as a subset of $[0, \mucr]$.  
  It follows that $\mathcal{N} = [0, \mucr]$, which, unraveling notation, gives the bound \eqref{uniform bound Phi}.

  Fix $\mu \in [0,\mucr]$. It remains to prove the second inequality in \ref{lem_minimality_pos}, that is, $\Phi_p(p;\mu) > 0$ for $-1 \le p \le 0$. Consider the function 
  \begin{align*}
    f(p) := \frac{\Phi_p(p;\mu)}{H_p^3(p)}.
  \end{align*}
  By construction $f(-1) > 0$, and we have shown that $f(0) > 0$ as well. 
  By \eqref{def Phi} and \eqref{uniform bound Phi}, $f_p(p) = \mu\rho_p \Phi(p;\mu) \le 0$
  so that $f$ is monotonically decreasing. Thus $f > 0$ and hence $\Phi_p(p;\mu) > 0$ for $-1 \le p \le 0$, as desired.
\end{proof}

An easy extension of the proof of Lemma~\ref{lem_minimality} gives the following corollary; see Figure~\ref{figure minimality}.
\begin{corollary} \label{cor_minimality}  
  There exist $\dirmu,\numu$ with $\mucr < \numu \le \dirmu \le +\infty$ such that the following hold.
  \begin{enumerate}[label=\rm(\alph*)] 
  \item If $0 \le \mu < \numu$, then $\Phi(p;\mu) > 0$ for $-1 < p \le 0$ while $\Phi_p(p;\mu) > 0$ for $-1 \le p \le 0$. Moreover, $\Phi_p(0;\numu) = 0$.
  \item $\Phi(0;\mu) > 0$ for $0 \le \mu \le \dirmu$, while $\Phi(0;\dirmu)=0$.
  \item For $\mucr < \mu \le \numu$, $\qprime(\mu) > 0$.
  \end{enumerate}
\end{corollary}
We define the associated Froude numbers $\dirF := 1/\sqrt{\dirmu}$, $\nuF := 1/\sqrt{\numu}$. Their significance for the qualitative theory is discussed in Corollary~\ref{froude corollary}. The first, $\dirF$, is the critical Froude number for channel flow.

Now we return to the full eigenvalue problem~\eqref{SLprob2}, but with $\mu = \mucr$ fixed.
\begin{lemma}[Spectrum] \label{spectrum lemma}   
  Let $\Sigma$ denote the set of eigenvalues $\nu$ for the Sturm--Liouville problem \eqref{SLprob2} with $\mu = \mucr$.  
  \begin{enumerate}[label=\rm(\alph*)] 
  \item $\Sigma = \{ \nu_j \}_{j=0}^\infty$, where $\nu_j \to \infty$ as $j \to \infty$, and $\nu_j < \nu_{j+1}$ for all $j \geq 0$;  
  \item $\nu_0 = 0$; and  
  \item each $\nu \in \Sigma$ has geometric and algebraic multiplicity $1$.
  \end{enumerate}
\end{lemma} 
\begin{proof}
As with \eqref{SLprob}, we will analyze \eqref{SLprob2} by introducing the solution $M = M(p;\nu)$ to the initial value problem
\begin{align}
  \label{Meqn} \left\{
  \begin{alignedat}{2}
    \Big( \frac{M_p}{H_p^3} \Big)_p
    - \mucr \rho_p M
    &= - \nu\frac{ M}{H_p}
    &\qquad& \text{in } -1 < p < 0, \\
    M &= 0 && \text{on } p=-1, \\
    M_p &= 1, && \text{on } p=-1,
  \end{alignedat} \right.
\end{align}
and the associated function
\begin{equation*}
  B(\nu) := \frac{M_p(0; \nu)}{M(0; \nu)}.
\end{equation*}
Note that $w = M(\placeholder; \nu)$ solves \eqref{SLprob2} for $\nu$ provided $B(\nu) = \mucr \rho(0) H_p(0)^3$.  Also, $B$ will have a pole at each $\dirnu$ that is an eigenvalue of the Dirichlet problem corresponding to \eqref{SLprob2}:
\begin{equation}
  -\left( \frac{\dot w_p}{H_p^3} \right)_p + \mucr \rho_p \dot w =  \dirnu\frac{ \dot w}{H_p}, \qquad \dot{w}(-1) = \dot{w}(0) = 0, \qquad \dot{w} \not\equiv 0.  \label{dirichlet eigenvalue problem} 
\end{equation}
By classical theory, the set of Dirichlet eigenvalues $\Sigma_{\mathrm{D}}$ is countably infinite, contains a sequence limiting to $+\infty$, and has no finite accumulation points.  
Moreover, each $\nu_{\mathrm{D}} \in \Sigma_{\mathrm{D}}$ is simple.  We can therefore enumerate $\Sigma_{\mathrm{D}} = \{ \nu_{\mathrm D}^{(j)} \}_{j=1}^\infty$ where $\nu_{\mathrm D}^{(j)} < \nu_{\mathrm D}^{(j+1)}$ and $\nu_{\mathrm D}^{(j)} \to \infty$ as $j \to \infty$.   
Due to the sign of $\rho_p$, a priori it is possible that several of the Dirichlet eigenvalues are negative.  
However, the definition of $\mucr$ ensures that this does not occur, as the following simple argument demonstrates.  

Suppose that $\dirnu \leq 0$ and let $\dot{w}$ be a corresponding solution of \eqref{dirichlet eigenvalue problem}.  For $\delta > 0$, consider a new function 
\begin{equation}
  \dot{v}^\delta  := \frac{\dot{w}}{\delta + \Phicr},
\end{equation}
which will then satisfy
\begin{align}
  \label{v delta}
  -\left( \frac{\delta+ \Phicr}{H_p^3} \dot{v}^\delta_p \right)_p - \frac{(\Phicr)_p}{H_p^3} \dot{v}^\delta_p + \left(\mucr \rho_p \delta -  \dirnu \frac{\delta+\Phicr}{H_p} \right) \dot{v}^\delta = 0,
\end{align}
where we have used the equation \eqref{def Phi} solved by $\Phicr$ to simplify some terms.

For $\dirnu < 0$, taking $\delta$ small enough ensures that the coefficient of the zeroth order term is positive, while for $\dirnu = 0$, this coefficient vanishes as $\delta \searrow 0$.  We may therefore apply the maximum principle to the above equation if $\delta$ is sufficiently small.  But then, because $\dot{v}^\delta = 0$ on the top and bottom, it follows that $\dot{v}^\delta \equiv 0$ and hence $\dot{w} \equiv 0$. Thus $\Sigma_{\mathrm{D}} \subset (0, \infty)$ as desired.

\begin{figure}
  \includegraphics[scale=1.1]{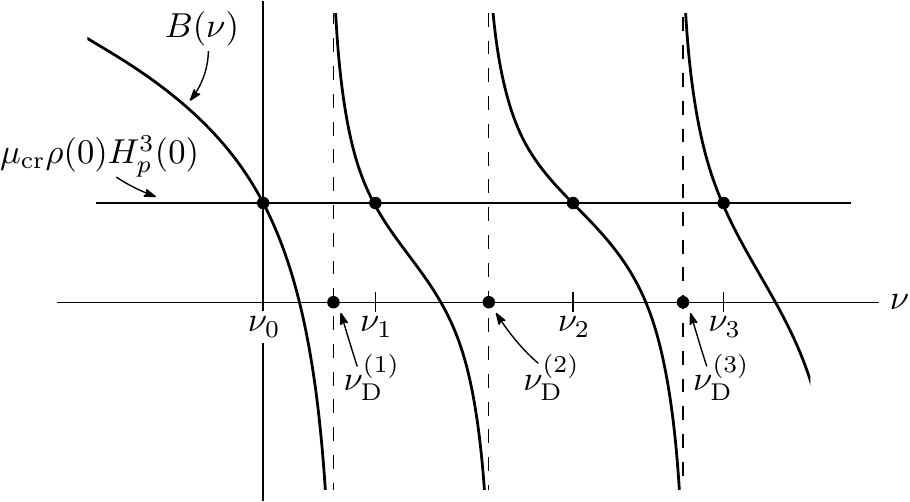}
  \caption{The function $B=B(\nu)$ from the proof of Lemma~\ref{spectrum lemma}.}
  \label{figure spectrum}
\end{figure}

Next, we differentiate \eqref{Meqn} with respect to $\nu$ to find
\begin{equation}\label{Mnueqn}
  -\left( \frac{M_{\nu p} }{ H^3_p} \right)_p 
  + \mucr \rho_p M_\nu 
  = \frac{M}{ H_p} + \nu \frac{ M_\nu }{ H_p}, \qquad M_\nu(-1; \nu) = M_{\nu p}(-1; \nu) = 0.
\end{equation}
Using this and the original equation \eqref{Meqn}, we obtain the Green's identity
\begin{equation*}
  \left.\left( \frac{M_p M_{\nu} }{ H^3_p} -\frac{MM_{\nu p} }{ H^3_p} \right) \right|^0_{p = -1} = \int^0_{-1} \frac{M^2}{ H^p} \ dp.
\end{equation*}
From this it follows that 
\begin{equation}\label{B'}
  B'(\nu) 
  = \frac{MM_{\nu p} - M_pM_\nu }{ M^2} 
  \bigg|_{p=0}
  = - \frac{H^3_p(0)}{ M^2(0)} \int^0_{-1} \frac{M^2 }{ H_p}\ dp < 0
\end{equation}
whenever $M(0;\mu) \ne 0$. Thus $B$ is a strictly decreasing function of $\nu$.

The presence of the poles at the Dirichlet eigenvalues then implies
$B(\nu) \to \pm\infty$ as $\nu \to \nu_{\mathrm D}^{(j)} \pm$ for all $j$.
Thus, for each $j \geq 1$, there is a unique $\nu_j \in (\nu_{\mathrm D}^{(j)}, \nu_{\mathrm D}^{(j+1)})$ with $B(\nu_j) = \mucr \rho(0) H_p(0)^3$; see Figure~\ref{figure spectrum}.  Moreover, since $B$ is strictly decreasing and smooth on $(-\infty, \nu_{\mathrm D}^{(1)})$, there exists at most one $\nu_0 \in (-\infty, \nu_{\mathrm D}^{(1)})$ for which the same holds true.    In fact, taking $\nu = 0$, \eqref{SLprob2} becomes \eqref{SLprob}, and hence has the nontrivial solution $\Phicr$ by Lemma~\ref{lem_minimality}.  We infer that $\nu_0$ exists and is simply $0$. Recalling the definition of $B$, we see that $\{ \nu_j \}_{j=0}^\infty$ are precisely the eigenvalues of \eqref{SLprob2}.  This proves parts (a) and (b).  

Finally, the simplicity of these eigenvalues is derived from the fact that the Sturm--Liouville problem \eqref{SLprob2} is formally self-adjoint.  
\end{proof}

\begin{remark}  
  The critical Froude numbers $\dirF, \nuF, \Fcr$ can also be defined as follows. 
  Set
  \[  L \dot{w} := -\left( \frac{\dot{w}_p}{H_p^3} \right)_p + \frac{1}{F^2} \rho_p \dot{w}\]
  and let
  $L_{\mathrm{D}}, L_\mathrm{N},L_\mathrm{R}$ denote $L$ with domain 
  \begin{align*}
    \mathcal{D}(L_{\mathrm{D}}) &:= \left\{ \dot{w} \in C^2([-1,0]) : \dot{w}(-1) = \dot{w}(0) = 0 \right\}, \\
    \mathcal{D}(L_{\mathrm{N}}) &:= \left\{ \dot{w} \in C^2([-1,0]) : \dot{w}(-1) = \dot{w}_p(0) = 0 \right\},\\
    \mathcal{D}(L_{\mathrm{R}}) &:= \left\{ \dot{w} \in C^2([-1,0]) : \dot{w}(-1) = - \frac{\dot w_p(0)}{H_p^3(0)} + \frac 1{F^2} \rho(0) \dot w(0) = 0 \right\}.
  \end{align*}
  It is easy to see that  the eigenvalues of $L_{\mathrm{D}}, L_{\mathrm{N}}, L_{\mathrm{R}}$ are strictly positive when $F$ is sufficiently large.  We can therefore define 
  \begin{align*} 
    \dirF &:= \inf{\left\{ F > 0 : \textrm{eigenvalues of $L_{\mathrm{D}}$ all positive} \right\}}, \\ 
    \nuF &:= \inf{\left\{ F > 0 : \textrm{eigenvalues of $L_{\mathrm{N}}$ all positive} \right\}}, \\
    \Fcr &:= \inf{\left\{ F > 0 : \textrm{eigenvalues of $L_{\mathrm{R}}$ all positive} \right\}}. 
  \end{align*}
  Note that if $\rho$ is a constant, then $\dirF = \nuF = 0$.  

  As a simple concrete example, consider the situation when $H_p$ and $\rho_p$ are both constants, in which case our normalization \eqref{Ustar normalization} together with \eqref{ODE H} forces $H_p \equiv 1$. Then the solution $\Phi$ of \eqref{def Phi} can be computed explicitly,
  \begin{align*}
    \Phi(p;\mu) = \frac{\sin\left( \sqrt{\mu\abs{\rho_p}} (p+1)\right)}{\sqrt{\mu\abs{\rho_p}}}.
  \end{align*}
  Using this, we easily calculate 
  \begin{align*}
    \dirF = \frac 2\pi \sqrt{\abs{\rho_p} },
    \qquad 
    \nuF &= \frac 12 \dirF = \frac 1\pi \sqrt{\abs{\rho_p} },
  \end{align*}
  while $\Fcr=1/\sqrt{\mucr}$ where $\mucr > 0$ is the smallest positive solution of 
  \begin{align*}
    \frac {\sqrt{\abs{\rho_p}}}{\sqrt{\mucr}}
    &= 
    \tan \Big( \sqrt{\mucr \abs{\rho_p} } \Big).
  \end{align*}
  In the limit as $\rho_p \to 0$, these formulas simplify to $\dirF = \nuF = 0$ and $\Fcr = 1$. 

  Likewise, the solution $M$ of \eqref{Meqn} has the explicit form
  \begin{align*}
    M(p;\nu)
    = \frac{\sin\left(\sqrt{\nu + \mucr \abs{\rho_p} }(p+1)\right)}
    {\sqrt{\nu  + \mucr \abs{\rho_p} }},
  \end{align*}
  so that the eigenvalues $\nu \ge 0$ of the related Sturm--Liouville
  problem \eqref{SLprob2} are 
  solutions of the algebraic equation
  \begin{align*}
    \sqrt{\nu + \mucr \abs{\rho_p} }
    = \mucr \tan \Big(\sqrt{\nu + \mucr \abs{\rho_p}}\Big).
  \end{align*}
\end{remark}

\subsection{Fredholm and invertibility properties} \label{fredholm section}

We now move on to the linearized operators $\F_w(w,F)$ for $(w,F) \in U$. Not surprisingly, the operators $\F_w(0,F)$ obtained by linearizing about the trivial solution $w=0$ play a special role; here it is crucial that $F > \Fcr$ so that we can use Lemma~\ref{lem_minimality}.

Consider the problem $\F_w(0,F)\dot w = (f_1,f_2)$, i.e.,
\begin{align}\label{linear} \left\{
  \begin{alignedat}{2}
    \Big( \frac{\dot w_p}{H_p^3} \Big)_p
    + \Big( \frac {\dot w_q}{H_p} \Big)_q
    + \frac {1}{F^2}\rho_p \dot w
    &= f_1
    &\quad& \textup{ in } R,\\
    - \frac{\dot w_p}{H_p^3} + \frac{1}{F^2} \rho \dot w
    &= f_2
    &\quad& \textup{ on } T,\\
    \dot w &= 0 && \textup{ on } B.
  \end{alignedat} \right.
\end{align}
We will view $\F_w(0,F)$ as a map $X \to Y$ but also as a map $\Xb \to \Yb$, where 
the spaces
\begin{align*}
  \Xb := \left\{ u\in C^{3+\alpha}_\bdd(\overline R):\ u|_B = 0 \right\},
  \quad 
  \Yb := C^{1+\alpha}_\bdd(\overline R) \times C^{2+\alpha}_\bdd(T).
\end{align*}
are like $X$ and $Y$ but without evenness or decay at infinity.

Both of the coefficients in front of $\dot w$ in \eqref{linear} have the ``bad'' sign in the sense that they do not satisfy the hypotheses of the maximum principle (cf.\ Theorem~\ref{max principle}). We can get around this, however, by using a slight variation of the function $\Phi$ from Lemma~\ref{lem_minimality}, namely the function $\Phia$ defined by 
\begin{equation}\label{def Phia} \left\{
  \begin{alignedat}{2}
    \Big(\frac{\Phia_p}{H_p^3}\Big)_p - \frac{1}{F^2} \rho_p \Phia &=0
    &\qquad& \text{in } -1 < p <0,\\
    \Phia &= \epsilon && \text{on } p = -1,\\
    \Phia_p &= 1 && \text{on } p = -1,
  \end{alignedat} \right.
\end{equation}
where $0 < \epsilon \ll 1$ is a constant (depending on $F$) to be determined.
\begin{lemma}\label{lem Phia}
  If $F > \Fcr$, then, for $\epsilon > 0$ sufficiently small,
  \begin{align}
    \label{Phia pos}
    \Phia > 0 \textup{ for } {-1} < p \le 0,
    \qquad 
    \Phia_p > 0 \textup{ for } {-1} \le p \le 0,
  \end{align}
  and 
  \begin{align}
    \label{Phia top}
    -\frac{\Phia_p}{H_p^3} + \frac {1}{F^2}\rho \Phia < 0 
    \textup{ on } p=0.
  \end{align}
  \begin{proof}
    Comparing \eqref{def Phia} to \eqref{def Phi}, we see that $\Phia = \Phi$ when $\epsilon = 0$. Thus, from Lemma~\ref{lem_minimality} we know that \eqref{Phia top} and the second inequality in \eqref{Phia pos} hold for $\epsilon$ sufficiently small. Indeed, since $\Phi_p > 0$ for $-1 \le p \le 0$, we in fact have $\Phi_p \ge \delta > 0$ for some constant $\delta$, and therefore $\Phia_p \ge \delta/2$, say, for sufficiently small $\epsilon$. Because $\Phia(-1) = 0$, the first inequality in \eqref{Phia pos} follows by integrating the second.
  \end{proof}
\end{lemma}
    
Fix $F > \Fcr$ and let $\Phia$ be the function whose existence is guaranteed by Lemma~\ref{lem Phia}. Making the change of dependent variable
\begin{align*}
  \dot w =: \Phia v,
\end{align*}
a calculation shows that $\F_w(0,F)\dot w = (f_1,f_2)$ is equivalent to
\begin{equation} \label{v equation} \left\{
  \begin{alignedat}{2}
    \left( \frac{v_p}{H_p^3} \right)_p +  \left( \frac{v_q}{H_p} \right)_q &= \frac{f_1}{\Phia} 
    &\qquad& \textrm{in } R,\\
    - \frac{v_p}{H_p^3} + \frac{1}{\Phia} \left( - \frac{\Phia_p}{H_p^3} + \frac{1}{F^2} \rho \Phia \right) v  &= \frac{f_2}{\Phia} 
    && \textrm{on } T,\\
    v &= 0 && \textrm{on } B.
  \end{alignedat} \right.
\end{equation}
The elliptic operator in \eqref{v equation} has no zeroth order term, and from \eqref{Phia top} the coefficient in front of $v$ in the 
boundary condition on $T$ has the ``good'' sign. 
At this point we can essentially follow the plan of \cite[Appendix A]{wheeler2013solitary}, and hence defer the proofs to Appendix~\ref{appendix calculations}.
Indeed, the arguments in that earlier work actually allow for the ``bad'' sign in the boundary condition, and so some of the analysis can even be simplified for our case.

\begin{lemma}[Strong invertibility] \label{strong invertibility lemma}
  For $F > \Fcr$, $\F_w(0,F) \colon \Xb \to \Yb$ is invertible.  
\end{lemma}

\begin{lemma}[Fredholm index $0$] \label{fredholm lemma} 
  For $(w,F) \in U$, $\F_w(w,F)$ is Fredholm with index $0$ both $\Xb \to \Yb$ and $X \to Y$.
\end{lemma}

\section{Qualitative properties} \label{qualitative section} 

\subsection{Bounds on the velocity and pressure} \label{bounds velocity pressure section}

Very little is known about the distribution of pressure in traveling waves with vorticity, even less so for those with density stratification.  Our first result gives an a priori lower bound on the pressure in a stratified steady wave.  This generalizes the best currently available lower bound for constant density rotational waves, 
due to Varvaruca~\cite{varvaruca2009extreme}.  
Bernoulli's law then furnishes an upper estimate on the magnitude of the relative velocity.

\begin{proposition}[Bounds on velocity and pressure] \label{bound on velocity prop} 
  The pressure and velocity fields for any solitary wave satisfy the bounds
  \begin{equation}
    \label{pressure and velocity bound}
    P + M \psi \geq 0 
    \qquad \textup{and} \qquad 
    (u-c)^2 + v^2 \leq C
    \qquad \textup{in } \overline{\Omega},
  \end{equation}
  where the constants $C$ and $M$ depend only on $u^*$, $\mathring\varrho$, $g$, $d$, and a lower bound for $F$.
\end{proposition}
Converting to dimensional variables we obtain Proposition~\ref{intro: bound v and P theorem}.
\begin{proof}[Proof of Proposition~\ref{bound on velocity prop}]
  Recall from Remark~\ref{remark bernoulli} that $\beta$ and hence $E$ are completely determined by $u^*$, $\rho$, $g$, and $d$.

  Let $f := P + M \psi$, where $M$ is a constant to be determined, and assume that $F > F_0$ for some fixed lower bound $F_0$. A tedious calculation using Yih's equation~\eqref{yih equation} shows that $f$ satisfies the elliptic equation
  \begin{equation}
    \begin{split}
      \label{eqn:elliptic} 
      \Delta f - b_1 f_x - b_2 f_y &=
      \frac{2F^{-2} \rho (2M + \Delta \psi) \psi_y - 2 F^{-4} \rho^2}{|\nabla \psi|^2} -
      \left[
      (2M +\Delta \psi) M - \frac 1{F^2} \rho_p\psi_y  \right],
    \end{split} 
  \end{equation}
  where the coefficients $b_1$ and $b_2$ are given by
  \begin{equation}
    b_1 := 
    2\frac{ \psi_x (2M + \Delta \psi)}{ |\nabla \psi|^2},
    \qquad 
    b_2 
    :=
    2\frac{\psi_y \left(2M + \Delta \psi \right) - 2 F^{-2}\rho}{|\nabla \psi|^2}.
  \end{equation} 

  From Bernoulli's law \eqref{bernoulli} and \eqref{defQ}, we see that 
  \begin{equation}
    \frac 1{F^2}\eta < \frac{(\mathring{u}(0)-c)^2}2 
    = \frac{u^*(0)^2}{2gd}.
  \end{equation}
  Yih's equation \eqref{yih equation} then gives the estimate
  \begin{equation}
    \Delta \psi = -\beta + \frac1{F^2} y \rho_p 
    \geq
    -\|\beta_+\|_{L^\infty} -  
    \frac {u^*(0)^2}{2gd} \| \rho_p \|_{L^\infty}.
  \end{equation}
  Thus, for 
  \begin{equation}
    M > M_1 := 
    \frac 12 \|\beta_+\|_{L^\infty} +  
    \frac {u^*(0)^2}{4gd}\| \rho_p \|_{L^\infty},
  \end{equation}
  we have $\Delta \psi + 2M \geq 0$. Using this fact in \eqref{eqn:elliptic} and estimating quite crudely yields the inequality  
  \begin{equation}
    \Delta f - b_1 f_x - b_2 f_y \leq - \left[ M^2 - \frac 1{F^2} \rho_p \psi_y \right].
  \end{equation} 
  Putting 
  \begin{equation}
    M_2 := \frac 1{F_0}\|  \psi_y \|_{L^\infty}^{1/2} \| \rho_p \|_{L^\infty}^{1/2},
  \end{equation}
  and taking $M :=  \max\{M_1, M_2\}$ therefore guarantees that 
  \begin{equation}
    \Delta f - b_1 f_x - b_2 f_y \leq 0 \qquad \textrm{in } \Omega.\label{f supersolution} 
  \end{equation}
  We claim that, with this choice of $M$, $f \ge 0$. At $x = \pm\infty$, the pressure is hydrostatic and hence positive, so $\liminf_{|x| \to \infty} f(x, \cdot) \ge 0$. On the surface, $f= 0$, and on the bed $\{y=-1\}$,
  \begin{align*}
    f_y = P_y + M \psi_y = -\frac 1{F^2}\rho + M \psi_y < 0 \qquad \textrm{on }  y = - 1,
  \end{align*}
  so that $f$ cannot be minimized there. As $f$ is a supersolution of the elliptic problem \eqref{f supersolution}, the claim follows from the maximum principle.

  Using Bernoulli's law again, the inequality $f \ge 0$ means
  \begin{align*}
    -\frac 12  |\nabla \psi|^2 - \frac 1{F^2}\rho y + E = P \ge -M \psi \ge -M,
  \end{align*}
  and hence, after rearranging,
  \begin{align}
    \label{eqn:nice}
    |\nabla \psi|^2 
    \le 2M - \frac{2}{F^2} \rho y +  2E 
    \le 2M + \frac{2}{F^2_0} + 2E,
  \end{align}
  where we have used the inequalities $y > -1$ and $0 < \rho \le 1$ satisfied by the dimensionless variables $\rho$ and $y$.

  First suppose that $M_2 \geq M_1$ so that $M = M_2$.  Then taking the supremum of the left-hand side of \eqref{eqn:nice} and dropping the $\psi_x$ term, we find 
  \begin{align*}
    \| \psi_y \|_{L^\infty}^2 
    & \leq 
    \left(\frac 12\| \psi_y \|_{L^\infty}^2 + \frac{3}{2F^{4/3}_0} \| \rho_p \|_{L^\infty}^{2/3} \right)
    + \frac{2 }{F^2_0} + 2\| E\|_{L^\infty}
  \end{align*}
  and hence 
  \begin{equation}
    \label{case 1 estimate psi_y} 
    \| \psi_y \|_{L^\infty}^2 < 
    \frac{3}{F^{4/3}_0} \| \rho_p \|_{L^\infty}^{2/3} + \frac{4 }{F_0^2} + 4\| E \|_{L^\infty}.
  \end{equation}
  In particular, $M_2$ is controlled by
  \begin{align*}
    M_2
    &\le 
    \frac 1{F_0}  \| \rho_p \|^{1/2}_{L^\infty}
    \left(
    \frac{3}{F^{4/3}_0} \| \rho_p \|_{L^\infty}^{2/3} + \frac{4 }{F_0^2} + 4\| E \|_{L^\infty}\right)^{1/4}.
  \end{align*}
  Plugging $M=M_2$ back into \eqref{eqn:nice} then yields a bound on the full gradient:
  \begin{align}
    \|\nabla \psi\|_{L^\infty}^2 
    \le 
    \frac 2{F_0}  \| \rho_p \|^{1/2}_{L^\infty}
    \left(
    \frac{3}{F^{4/3}_0} \| \rho_p \|_{L^\infty}^{2/3} + \frac{4 }{F_0^2} + 4\| E \|_{L^\infty}\right)^{1/4}
    + \frac{2}{F^2_0} + \n E_{L^\infty},
  \end{align}
  On the other hand, if $M_1 \geq M_2$ so that $M = M_1$, \eqref{eqn:nice} gives immediately that 
  \begin{equation*}
    \| \nabla \psi \|_{L^\infty}^2 \leq 2 M_1 + \frac 2{F_0^2} + 2 \| E \|_{L^\infty}.
    \qedhere
  \end{equation*}
\end{proof}

The estimates \eqref{pressure and velocity bound} can be translated into bounds on the semi-Lagrangian quantities.

\begin{corollary}[Bounds on $w$ and $\nabla w$] \label{bound on w and w_q cor}  
  There exist positive constants $C_*$ and $\delta_*$ so that any solitary wave with $F \ge \Fcr$ satisfies
  \begin{equation}
    \label{uniform bound w nabla w} 
    \inf_{R} \left( w_p + H_p \right) > \delta_* \qquad \textup{and} \qquad 
    \| w \|_{C^1(R)} < C_*(1+\| w_p \|_{C^0(R)}).
  \end{equation}
\end{corollary}
\begin{proof}
  Letting $F_0 = \Fcr$ in Proposition~\ref{bound on velocity prop} and using \eqref{h to uv}, we have
  \begin{align*}
    \frac 1{h_p^2} + \frac {h_q^2}{h_p^2} ={\rho \over F^2} \LB (u-c)^2 + v^2 \RB < C,
  \end{align*}
  for some constant $C$. 
  Dropping the second term on the left-hand side yields
  \begin{align*}
    \inf_R (H_p+w_p)
    = \inf_R h_p 
    \ge \frac 1{\sqrt C} =: \delta_*,
  \end{align*}
  while dropping the first yields
  \begin{align}
    \label{wq vs wp}
    \abs{w_q} = \abs{h_q} < \sqrt{C} h_p = \sqrt{C} (H_p + w_p)
    \le C_2(1 + \abs{w_p}).
  \end{align}
  Taking the supremum over $R$ of \eqref{wq vs wp}, we conclude that $\n{w_q}_{C^0(R)} \le C_2 (1+ \n{w_p}_{C^0(R)})$. The full bound on $\n w_{C^1(R)}$ then follows from writing $w(q,p) = \int_{-1}^p w_p(q,p')\, dp'$.
\end{proof}

\subsection{Bounds on the Froude number} \label{bounds on F section} 
The objective of this subsection is to derive a priori estimates from above and below for the Froude number.  These are of general interest to studies of solitary waves, but are of special importance to our arguments in Section~\ref{global bifurcation section}.  The earliest work on this topic we are aware of was carried out by Starr \cite{starr1947momentum}, who formally derived sharp lower and upper bounds for homogeneous irrotational solitary waves.
We refer to \cite[Section~1.2]{wheeler2015froude} for a detailed historical discussion, but emphasize that our upper bound is new even for constant density waves with vorticity. In particular, while our estimate involves a measure of stagnation that does not appear in \cite{wheeler2015froude}, it does not require any additional assumptions on $u^*$.

\subsubsection{Lower bound}
First we will prove that there are no nontrivial solutions $(w,F)$ of \eqref{w equation} with critical Froude number $F = \Fcr$. We call this a lower bound because it will imply that the continuous curve of solutions that we construct in Section~\ref{global bifurcation section} can only reach a subcritical wave with $F < \Fcr$ by first passing through the critical laminar flow $(0,\Fcr)$. 
In the special case of constant density, our argument reduces to that in \cite{wheeler2015froude}, and further implies the inequality $F > \Fcr$ for waves of elevation. To get a similar result with non-constant density, we need to assume $F > \nuF$ where $\nuF < \Fcr$; see Corollary~\ref{froude corollary}. This extra assumption is related to the additional complexity of the Sturm--Liouville problem studied in Section~\ref{sec SL}. We note that the proof that $F \ne \Fcr$ in \cite{wheeler2015froude} was extended in \cite{kozlov2015conjecture} to constant density waves which are not necessarily solitary or even periodic.

The main ingredient in our argument is the following integral identity involving the functions $\Phi$ and $\qprime$ defined at the start of Section~\ref{sec SL}, as well as the free surface profile $\eta(\placeholder) = w(\placeholder,0)$. For two homogeneous and irrotational layers, this identity yields (B.8) in \cite{dias2001free}, at least formally. 
\begin{lemma} \label{froude lower lemma}
  For any solution $(w,F) \in X \times \R$ of the height equation \eqref{w equation} we have
  \begin{align}
    \label{froude lower identity}
    \begin{aligned}
      \int_{-M}^M \int_{-1}^0
      \frac{H_p^3 w_q^2 + (H_p+2h_p)w_p^2}{2h_p^2H_p^3}
      \Phi_p\Big(p; \frac 1{F^2}\Big) \, dp\, dq
      + 
      \qprime\Big( \frac 1{F^2}\Big)
      \int_{-M}^M \eta \, dx
      \longrightarrow 0
    \end{aligned}
  \end{align}
  as $M \to \infty$. 
  \begin{proof}
    Multiplying \eqref{height equation} by $\Phi$, integrating by parts, and using the equation \eqref{def Phi} satisfied by $\Phi$, we obtain
    \begin{align*}
      0
      &=
      \int_{-M}^M \int_{-1}^0
      \bigg[
      \Big( - \frac{1+h_q^2}{2h_p^2}
      + \frac 1{2H_p^2} \Big)_p \Phi
      + \Big( \frac{h_q}{h_p} \Big)_q \Phi
      - \frac 1{F^2}\rho_p (h-H) \Phi\bigg] \, dp\, dq
      \\
      &=
      \int_{-M}^M \int_{-1}^0
      \bigg[
      \Big(  \frac{1+h_q^2}{2h_p^2}
      - \frac 1{2H_p^2} \Big) \Phi_p
      - \Big(\frac{\Phi_p}{H_p^3} \Big)_p(h-H) 
      \bigg]
      \, dp\, dq\\
      &\qquad + \int_{-\infty}^\infty 
      \Big( - \frac{1+h_q^2}{2h_p^2}
      + \frac 1{2H_p^2} \Big) \Phi
      \bigg|_{p=0}
      \, dq
      + \int_{-1}^0 \frac{h_q}{h_p} \Phi\, dp \bigg|_{q=-M}^{q=M},
    \end{align*}
    so that integrating by parts once more yields
    \begin{align}
      \label{after once more}
      \begin{aligned}
        0
        &=
        \int_{-M}^M \int_{-1}^0
        \bigg[
        \Big( \frac{1+h_q^2}{2h_p^2}
        - \frac 1{2H_p^2} \Big) \Phi_p
        + \frac{\Phi_p}{H_p^3} (h_p-H_p) 
        \bigg]
        \, dp\, dq\\
        &\qquad+ \int_{-M}^M \bigg[
        \Big( - \frac{1+h_q^2}{2h_p^2}
        + \frac 1{2H_p^2} \Big) \Phi
        - \frac{\Phi_p}{H_p^3} (h-H) 
        \bigg]_{p=0}\, dq
        + \int_{-1}^0 \frac{h_q}{h_p} \Phi\, dp \bigg|_{q=-M}^{q=M}.
      \end{aligned}
    \end{align}
    Rewriting the first integrand in \eqref{after once more} as
    \begin{align*}
      \left[
      \Big(  \frac{1+h_q^2}{2h_p^2}
      - \frac 1{2H_p^2} \Big) 
      + \frac{1}{H_p^3} (h_p-H_p) \right]
      \Phi_p 
      &= 
      \frac{H_p^3 w_q^2 + (H_p+2h_p)w_p^2}{2h_p^2H_p^3}
      \Phi_p,
    \end{align*}
    and applying the boundary conditions in \eqref{height equation} and \eqref{def Phi}, we are left with 
    \begin{align*}
      &\int_{-M}^M \int_{-1}^0
      \frac{H_p^3 h_q^2 +  (H_p+2h_p)w_p^2}{2h_p^2H_p^3}
      \Phi_p \, dp\, dq
      + 
      \qprime\Big(\frac 1{F^2}\Big)
      \int_{-M}^M \eta \, dx\\
      & \qquad\qquad = - \int_{-1}^0 \frac{h_q}{h_p} \Phi\, dp \bigg|_{q=-M}^{q=M}
      \longrightarrow 0 
    \end{align*}
    as $M \to \infty$ as desired.
  \end{proof}
\end{lemma}

\begin{theorem}\label{froude lower theorem}
  Let $(w,F) \in X \times \R$ be a solution of the height equation \eqref{w equation}. 
  \begin{enumerate}[label=\rm(\roman*)]
  \item \label{elevation2froude} If $w > 0$ on $T$, and if $F$ is such that $\Phi_p \ge 0$, then $\qprime(1/F^2) < 0$.
  \item \label{no critical} If $F = \Fcr$, then $w \equiv 0$. That is, there exist no nontrivial solitary waves with critical Froude number.   
  \end{enumerate}
\end{theorem} 
\begin{proof}
  To prove \ref{elevation2froude}, suppose that $w > 0$ on $T$ (i.e.,~$\eta > 0$) and that $\Phi_p > 0$. Then both of the integrals in \eqref{froude lower identity} are strictly positive for all $M > 0$. Thus, in order for the limit in \eqref{froude lower identity} to hold, the coefficient $\qprime(1/F^2)$ must be negative. 
 
  To prove \ref{no critical}, observe that $A(1/\Fcr^2) = 0$ by Lemma~\ref{lem_minimality}(i) so that, for $F = \Fcr$, \eqref{froude lower identity} reduces to
  \begin{align}
    \label{monotone contradiction}
    \int_{-M}^M \int_{-1}^0
    \frac{H_p^3 w_q^2 +  (H_p+2h_p)w_p^2}{2h_p^2H_p^3}
    (\Phicr)_p \, dp\, dq
    \longrightarrow 0 
  \end{align}
  as $M \to \infty$. Since $(\Phicr)_p > 0$ by Lemma~\ref{lem_minimality}(ii), the left-hand side of \eqref{monotone contradiction} is a nonnegative, nondecreasing function of $M > 0$. Thus the limit in \eqref{monotone contradiction} forces the left-hand side to vanish for all $M$, which in turn forces $w_q, w_p \equiv 0$ and hence $w \equiv 0$.
\end{proof}
\begin{corollary}\label{froude corollary}
  Let $(w,F) \in X \times \R$ be a solution of the height equation \eqref{w equation}. If $w > 0$  on $T$, and if $F \ge \nuF$ then in fact $F > \Fcr > \nuF$. Here $\nuF^2 = 1/\numu$ is defined in Corollary~\ref{cor_minimality}. Thus waves of elevation are either supercritical with $F > \Fcr$ or quite subcritical in that $F < \nuF < \Fcr$.
  \begin{proof}
    This follows immediately from Theorem~\ref{froude lower theorem}\ref{elevation2froude} and Corollary~\ref{cor_minimality}.
  \end{proof}
\end{corollary}

\subsubsection{Upper bound}
Next we will prove an upper bound on the Froude number. Unlike \cite{wheeler2015froude} in the case of 
constant density, our estimate does not require any additional assumptions on the asymptotic height function $H$. Our argument is more closely related to those given by Starr \cite{starr1947momentum} and Keady and Pritchard \cite{keady1974bounds}. While they are able to obtain an upper bound $F < \sqrt 2$, the presence of vorticity seems to unavoidably introduce a term like $\n{h_p(0,\placeholder)}_{L^\infty}$ 
which measures how close the wave is to stagnation on the line beneath the crest. This additional term also means that our bound is not a strict generalization of \cite{wheeler2015froude}. For constant density waves, Kozlov, Kuznetsov, and Lokharu \cite{kozlov2015conjecture} have obtained yet another distinct upper bound on $F$ --- really on the
Bernoulli constant $Q$ defined in \eqref{defQ} --- in terms of the amplitude $\max \eta$. Their method involves a detailed characterization of one-parameter families of laminar flows with the same vorticity function, and hence seems particularly difficult to generalize to the stratified case.

Our upper bound is a consequence of the following integral identity.
\begin{lemma} \label{second identity lemma} 
  For any solution $(w,F) \in X \times \R$ of the height equation \eqref{w equation}, 
  \begin{align}
    \label{second F identity}
    \frac 1{F^2}\left[\int_{-1}^0 \abs{\rho_p} w(0,p)^2 \, dp
    + \rho(0) \eta(0)^2\right]
    = 
    \int_{-1}^0 
    \frac{w_p^2}{H_p^2 h_p}(0,p)\, dp.
  \end{align}
\end{lemma}
\begin{proof} 
  Comparing the flow force $\flowforce$  \eqref{DJ S} at $q=0$ and $q=\pm\infty$, we find
  \begin{align*}
    \flowforce(h) 
    &= 
    \int_{-1}^0
    \left( \frac 1{2h_p^2} + \frac 1{2H_p^2}- \frac{1}{F^2} \rho (h-H) - \int_0^p \frac{1}{F^2} \rho H_p\,
    dp' \right)\, h_p\, dp\\
    &= 
    \int_{-1}^0
    \left( \frac 1{2H_p^2} + \frac 1{2H_p^2}- \int_0^p \frac{1}{F^2} \rho H_p\,
    dp' \right)\, H_p\, dp
    = \flowforce(H),
  \end{align*}
  where, here and in what follows, all integrals are taken at $q=0$. Grouping terms and integrating by parts, we get 
  \begin{align}
    \notag
    0&=
    \int_{-1}^0
    \left [
    \left( \frac 1{2h_p^2} - \frac 1{2H_p^2} \right) H_p
    + \left( \frac 1{2h_p^2} + \frac 1{2H_p^2} \right) w \right] \, dp \\
    \notag
    & \qquad -\int_{-1}^0 \left[ 
    \frac{1}{F^2} \rho ww_p
    +\frac{1}{F^2} \rho wH_p
    +\left(\int_0^p \frac{1}{F^2} \rho H_p\, dp'\right) w_p
    \right ] \, dp\\
    &=
    \label{rearrange me}
    \int_{-1}^0
    \left [
    \left( \frac 1{2h_p^2} - \frac 1{2H_p^2} \right) H_p
    + \left( \frac 1{2h_p^2} + \frac 1{2H_p^2} \right) w_p
    \right ] \, dp
    -\frac{1}{F^2} \int_{-1}^0\rho ww_p \, dp.
  \end{align}
  Integrating by parts again, the second integral in \eqref{rearrange me} becomes \begin{align*}
    \frac 1{F^2} \int_{-1}^0\rho ww_p \, dp
    &= 
    \frac 1{2F^2} \int_{-1}^0 \abs{\rho_p} w^2\, dp
    + \frac 1{2F^2} \rho w^2 \bigg|_{p=0},
  \end{align*}
  while some algebra shows that the first integral collapses to
  \begin{align*}
    \int_{-1}^0
    \left [
    \left( \frac 1{2h_p^2} - \frac 1{2H_p^2} \right) H_p
    + \left( \frac 1{2h_p^2} + \frac 1{2H_p^2} \right) w_p
    \right ] \, dp
    &= 
    -\int_{-1}^0 
    \frac{w_p^2}{2H_p^2 h_p}\, dp,
  \end{align*}
  leaving us with \eqref{second F identity} as desired.
\end{proof}

\begin{theorem}[Upper bound on $F$] \label{upper bound on F theorem} 
  Let $(w,F) \in X \times \R$ be a solution of the height equation \eqref{w equation} and set $h = w + H$.  Then the Froude number $F$ satisfies the bound 
  \begin{align}
    \label{eqn:hurray}
    F^2 
    \le 
    \frac 1\pi
    \n{H_p}_{L^\infty}^2
    \n\rho_{L^\infty}
    \n{h_p(0,\placeholder)}_{L^\infty}.
  \end{align}
\end{theorem}
\begin{proof}
  Rewriting \eqref{second F identity} slightly as 
  \begin{align*}
    \frac 1{2F^2} 
    \int_{-1}^0 \rho w w_p\, dp
    = 
    \int_{-1}^0 
    \frac{w_p^2}{H_p^2 h_p}\, dp,
  \end{align*}
  we crudely estimate the left-hand side by
  \begin{align*}
    \frac 1{2F^2} 
    \int_{-1}^0 \rho w w_p\, dp
    \le \frac 1{F^2} \n\rho_{L^\infty} 
    \n {w(0,\placeholder)}_{L^2} \n{w_p(0,\placeholder)}_{L^2}
    \le \frac 1{\pi F^2} \n\rho_{L^\infty} \n{w_p(0,\placeholder)}_{L^2}^2
  \end{align*}
  and the right-hand side by
  \begin{align*}
    \int_{-1}^0 
    \frac{w_p^2}{H_p^2 h_p}\, dp
    \ge \left(\min_p \frac 1{H_p^2} \right)
    \left(\min_{\{q=0\}} \frac 1{h_p}\right)
    \n{w_p(0,\placeholder)}_{L^2}^2.
  \end{align*}
  Together, these estimates imply
  \begin{align*}
    \frac 1{\pi F^2} \n\rho_{L^\infty} \n{w_p(0,\placeholder)}_{L^2}^2
    \ge 
    \left(\min_p \frac 1{H_p^2} \right)
    \left(\min_{\{q=0\}} \frac 1{h_p}\right)
    \n{w_p(0,\placeholder)}_{L^2}^2,
  \end{align*}
  and hence
  \begin{align*}
    F^2 
    \le 
    \frac 1\pi
    \n{H_p}_{L^\infty}^2
    \n\rho_{L^\infty}
    \n{h_p(0,\placeholder)}_{L^\infty},
  \end{align*}
  which completes the proof.  
\end{proof}

The identity \eqref{second F identity} can also be used in the case of homogeneous and irrotational flow to recover Starr's and Keady--Pritchard's upper bound of $\sqrt{2}$ for the Froude number.  This is possible simply because the right-hand side of \eqref{second F identity} can be simplified for such waves.  
The presence of non-constant stratification forces us to do more crude estimates when proving Theorem \ref{upper bound on F theorem}.

\subsection{Nonexistence of monotone bores} \label{bore section}

By a bore we mean a solution $h$ of the height equation \eqref{height equation} satisfying
\begin{equation}
  h(q, p) \to H_\pm(p) \textrm{ as } q \to \pm\infty, 
  \label{bore asymptotics} 
\end{equation}
pointwise in $p$, where $H_-$ and $H_+$ represent distinct laminar flows. 
Outside of hydrodynamics, traveling waves of this type are often called fronts.  If one adopts the spatial dynamics viewpoint, thinking of $q$ as the time variable, this is equivalent to having a heteroclinic orbit connecting the rest points $H_-$ and $H_+$.  

Bores are of particular importance to this work because they represent a barrier to proving that $\F$ has certain compactness properties.  Observe that, a priori, 
there may be a sequence of even solitary waves of elevation where the crest flattens and expands into an infinite shelf; see Figure~\ref{noncompact figure}.  Naturally, this scenario precludes the existence of a subsequential limit in $C_0^2(R)$, and hence would mean that $\F^{-1}(0)$ is not locally compact.   Using a translation argument, it is shown in Section~\ref{global bifurcation section} that this would also imply the existence of a bore which is \emph{monotone} in that $H \le H_+ \le H_-$ (cf.\ Lemmas~\ref{lem:compact:gen} and \ref{lem:compact}); see Figure~\ref{bores figure}.

As mentioned in the introduction, bores exist for stratified waves in many regimes. However, as we prove in this subsection, one can completely rule out the existence of \emph{monotone} bores for \emph{free surface} stratified waves. 

\begin{theorem}[Nonexistence of monotone bores] \label{bores theorem}  
  Suppose that $h \in C^2_\bdd(\overline R)$ is a bore solution of the height equation \eqref{height equation} with $\inf_R h_p > 0$, and let $H_\pm$ be as in \eqref{bore asymptotics}. If 
  \begin{align*}
    \textup{$H_+ \ge H_- = H$ on $[-1,0]$}
    \quad  \textup{ or } \quad
    \textup{$H_+ \le H_- = H$ on $[-1,0]$},
  \end{align*}
  then $H_+ = H_- = H$. The same result holds with the roles of $H_-$ and $H_+$ reversed.
\end{theorem}

The proof of Theorem~\ref{bores theorem} relies on the following integral identity.
\begin{lemma} \label{bore identity}
  Suppose that $h = K(p)$ is a solution of the height equation \eqref{height equation} with $K \in C^2([-1,0])$ and $K_p > 0$. If $H$ and $K$ are conjugate in that $\flowforce(K) = \flowforce(H)$, then
  \begin{align}
    \int_{-1}^0 \frac{(K_p-H_p)^3}{H_p^2 K_p^2} \, dp = 0.
    \label{auxiliary condition} 
  \end{align}
\end{lemma}
\begin{proof}
  Let $K$ be given as above. Since $K$ is a $q$-independent solution of the height equation, it satisfies the ODE
  \begin{equation}
    -\left( \frac{1}{2K_p^2} \right)_p + \left( \frac{1}{2H_p^2} \right)_p - \frac{1}{F^2} \rho_p (K-H) = 0 \qquad \textrm{in } (-1,0)\label{K eq} 
  \end{equation}
  and vanishes on the bed, and on the top 
  \begin{equation}
    -\frac{1}{2K_p(0)^2} + \frac{1}{2H_p(0)^2} = \frac{1}{F^2} \rho(0)(K(0) - H(0)).\label{K top cond} 
  \end{equation}
  Multiplying \eqref{K eq} by $J:= K-H$ and integrating by
  parts, we find
  \begin{align*}
    0 &= 
    \int_{-1}^0 \left[\left(
    - \frac 1{2K_p^2}
    + \frac 1{2H_p^2} \right)_p J
    - \frac{1}{F^2} \rho_p J^2\right]\, dp \\
    &= 
    \int_{-1}^0 \left(
    \frac 1{2K_p^2}
    - \frac 1{2H_p^2} \right) J_p\, dp
    - \int_{-1}^0 \frac{1}{F^2} \rho_p J^2 \, dp
    + \left(
    - \frac 1{2K_p^2}
    + \frac 1{2H_p^2} \right) J\bigg|_{p=0}.
  \end{align*}
  Using the boundary conditions \eqref{K top cond}, this simplifies to 
  \begin{align}
    \label{eqn:useme}
    \int_{-1}^0 \frac{1}{F^2} \rho_p J^2\, dp &= 
    \int_{-1}^0 \left(
    \frac 1{2K_p^2}
    - \frac 1{2H_p^2} \right) J_p\, dp
    + \frac{1}{F^2}  \rho J^2 \Big|_{p=0}.
  \end{align}

  Since $K$ and $H$ have the same flow force, we can argue as in the proof of Lemma~\ref{second identity lemma} but with $h(q,\placeholder)$ replaced by $K$ to obtain an analogue of \eqref{rearrange me}:
  \begin{align*}
    0
    &= 
    \int_{-1}^0
    \left [
    \left( \frac 1{2K_p^2} - \frac 1{2K_p^2} \right) K_p
    + \left( \frac 1{2K_p^2} + \frac 1{2K_p^2} \right) J_p
    \right ] \, dp
    -\frac{1}{F^2} \int_{-1}^0\rho JJ_p \, dp.
  \end{align*}
  Integrating by parts in the last integral we get
  \begin{align*}
    0
    &=
    \int_{-1}^0
    \left [
    \left( \frac 1{2K_p^2} - \frac 1{2H_p^2} \right) H_p
    + \left( \frac 1{2K_p^2} + \frac 1{2H_p^2} \right) J_p
    \right ] \, dp
    +\frac{1}{2F^2} \int_{-1}^0 \rho_p J^2\, dp
    -\frac{1}{2F^2} \rho J^2 \bigg|_{p=0}.
  \end{align*}
  Substituting \eqref{eqn:useme}, the boundary terms cancel, leaving us
  with
  \begin{align}
    0
    &=
    \int_{-1}^0
    \left [
    \left( \frac 1{2K_p^2} - \frac 1{2H_p^2} \right) H_p
    + \left( \frac 1{2K_p^2} + \frac 1{2H_p^2} \right) J_p
    +\frac 12 \left( \frac 1{2K_p^2} - \frac 1{2H_p^2} \right) J_p 
    \right ] \, dp \nonumber \\
    &=
    \frac 14\int_{-1}^0
    \frac{J_p^3}{H_p^2 K_p^2}
    \, dp
  \end{align}
  as desired.
\end{proof}

In order to apply Lemma~\ref{bore identity} to prove Theorem~\ref{bores theorem}, we need to know that the asymptotic states $H_\pm$ in the statement of the theorem are themselves solutions to the height equation with the expected regularity and flow force, based solely on the pointwise limit in \eqref{bore asymptotics}. This is the content of the following technical lemma.
\begin{lemma}\label{flow force lemma}
  Let $h \in C^2_\bdd(\overline R)$ be a solution of the height equation \eqref{height equation} with $\inf_R h_p > 0$ which is a bore in that \eqref{bore asymptotics} holds for some functions $H_\pm$. Then $H_\pm \in C^2([-1,0])$ are $q$-independent solutions to the height equation \eqref{height equation} and $\flowforce(H_+) = \flowforce(H_-)$.
\end{lemma}
\begin{proof}
  Consider the left-translated sequence $\{h_n\}$ defined by
  $h_n:= h(\placeholder+n,\placeholder)$. Via a standard argument, we can extract a subsequence so that $h_n \to h_+$ in $C^{1+1/2}_\loc(\overline R)$ for some $h_+ \in C^2_\bdd(\overline R)$. Letting $\inf_R h_p =: \delta > 0$, we easily check that $\inf_R (h_+)_p \ge \delta$ and $\flowforce(h_+) = \flowforce(H)$. Moreover, $h_+$ solves the (divergence form) height equation \eqref{height equation} in the weak sense. Since $h_+$ has the additional regularity $h_+ \in C^2_\bdd(\overline R)$, it is also a classical solution. Finally, comparing $h_n \to h_+$ in $C^{1+1/2}_\loc(\overline R)$ with \eqref{bore asymptotics}, we see that $h_+(q,p) = H_+(p)$, which completes the proof for $H_+$. Arguing similarly with right translations we obtain the same results for $H_-$.
\end{proof}

\begin{proof}[Proof of Theorem~\ref{bores theorem}]
  Set $K = H_+$, and assume first that $K \ge H$. By Lemma~\ref{flow force lemma}, we know that $K \in C^2([-1,0])$ solves the height equation \eqref{height equation} with $K_p > 0$ and that $\flowforce(K) = \flowforce(H)$. In particular, as in the proof of Lemma~\ref{bore identity}, $K$ satisfies \eqref{K eq} and \eqref{K top cond}. A simple consequence of \eqref{K eq} is that  
  \begin{equation}
    \left( -\frac{1}{2K_p^2} + \frac{1}{2H_p^2} \right)_p = \frac{1}{F^2} \rho_p (K-H) \leq 0,
  \end{equation}
  hence the quantity in parentheses on the left-hand side above is nonincreasing in $p$. From the boundary condition \eqref{K top cond} on $T$, 
  we then have
  \begin{equation}
    -\frac{1}{2K_p^2} + \frac{1}{2H_p^2} \geq \frac{1}{F^2} \rho(0)(K(0) -H(0)) \geq 0
  \end{equation}
  for all $p$ and hence 
  $K_p \geq H_p$ on $[-1,0]$.
  But now Lemma~\ref{bore identity} implies that $K_p - H_p \equiv 0$. Since $H(-1) = K(-1) = 0$, it follows that $H \equiv K$. A similar argument shows that the same holds true if $K \leq H$.
\end{proof}

The above proof is built on two facts. First, since $K_p, H_p > 0$, the equation $J = K - H$ satisfies is elliptic. Thus  if $J$ is nonnegative (or nonpositive) it is a supersolution (or subsolution).
For the free surface problem, the boundary condition on $\{ p = 0\}$ then allows us to infer that $J_p$ cannot change signs.  By contrast, for a channel flow, the boundary condition at the top will be inhomogeneous Dirichlet, which does not permit us to draw the same conclusion.  Indeed, by Rolle's theorem, $J_p$ \emph{must} change signs in the interior.  The second building block is the integral identity \eqref{auxiliary condition}.
An analogous identity was discovered by Lamb and Wan \cite[Appendix A]{lamb1998conjugate} for solitary stratified flows in a channel with uniform velocity at infinity, and for general stratified solitary waves in a channel by Lamb \cite[Appendix A]{lamb2000three}.  Again, this does not lead to a contradiction because  $J_p$ is not single-signed in the regimes studied by these authors.

A look at the proof of Theorem~\ref{bores theorem} shows that we do not need to assume $H_- = H$, but only that $\flowforce(H)=\flowforce(H_-)=\flowforce(H_+)$ and that either $H_+ \ge H$ or $H_+ \le H$, and similarly for $H_-$. This leads to the following corollary, which we will need in Section~\ref{global bifurcation section}. 
\begin{corollary}\label{bores corollary}
  Suppose that $h \in C^2_\bdd(\overline R)$ is a bore solution of the height equation \eqref{height equation} with $\inf_R h_p > 0$, and let $H_\pm$ be as in \eqref{bore asymptotics}. If $\flowforce(H_+)=\flowforce(H_-)=\flowforce(H)$, and if 
  \begin{align*}
    \textup{$H_+ \ge H$ on $[-1,0]$}
    \quad  \textup{ or } \quad
    \textup{$H_+ \le H$ on $[-1,0]$},
  \end{align*}
  then $H_+ \equiv H$. The same result holds with $H_+$ replaced by $H_-$.
\end{corollary}

Theorem~\ref{bores theorem} can also be generalized to the case of a free surface flow with multiple layers.  These are not the main subject of this paper, but given that many authors have performed conjugate flow analysis for multi-fluid flows, we wish to emphasize that the nonexistence result above is not a consequence of the smoothness of $\rho$.  

Let us briefly recall the governing equations for the layer-wise continuous stratification; a more thorough discussion is given, for example, in \cite{chen2016continuous}.  The fluid domain is assumed to be partitioned into finitely many immiscible strata
\[
  \label{def Omega_i} {\Omega} = \bigcup_{i=1}^N {\Omega_i}, \quad \Omega_i := \{ (x, y) \in \Omega : \eta_{i-1}(x) < y < \eta_i(x) \} 
\]
with the density $\varrho$ smooth in each layer:
\begin{equation}
  \varrho \in  C^{1+\alpha}(\overline{\Omega_1}) \cap \cdots \cap C^{1+\alpha}(\overline{\Omega_N}).\label{regularity weak rho} 
\end{equation}
The indexing convention is  that $\eta_0 := -1$, $\eta_N := \eta$ and $\Omega_i$ lies beneath $\Omega_{i+1}$ for $i = 1, \ldots, N-1$. Euler's equation will hold in the strong sense in each interior layer and the pressure is assumed to be continuous in $\overline{\Omega}$.  

It is possible even in this setting to perform the Dubreil-Jacotin transform. Letting $R_i := \{ (q,p) \in \R \times (p_{i-1}, p_i) \}$ be the image of $\Omega_i$, one finds that the height function $h$ will have the regularity 
\begin{equation}
  h \in C^{0+\alpha}(\overline{R}) \cap C^{1+\alpha}(\overline{R_1}) \cap \cdots \cap C^{1+\alpha}(\overline{R_N})\label{regularity weak h} 
\end{equation}
and solve the height equation \eqref{height equation} in the sense of distributions.  Notice that $\rho_p$ will include Dirac $\delta$ masses on the internal interfaces; the continuity of the pressure precisely guarantees that the quantity on the left-hand side of the height equation is equal to $0$ in the distributional sense.  

In particular, a laminar flow $K$ will satisfy \eqref{K eq} in the strong sense on each $[p_i, p_{i+1}]$, along with top boundary condition \eqref{K top cond}.  The continuity of the pressure translates to the requirement that 
\begin{equation}
  \label{weak K jump cond}
  \jump{{1\over 2K^2_p} - {1\over 2H^2_p}}_i +  \frac 1{F^2}\jump{\rho}_i (K-H) = 0 \qquad \textrm{on } \{p = p_i\},
\end{equation} 
for $i = 1, \ldots, N-1$.  Here, $\jump{f}_i := f(p_i+) - f(p_i - )$ denotes the jump across the $i$-th layer.
One can see quite easily that \eqref{weak K jump cond} is equivalent to viewing \eqref{K eq} as being satisfied in the sense of distributions.  

A bore in this context means that \eqref{bore asymptotics} holds for distinct laminar flows $H_\pm$ with regularity \eqref{regularity weak h}. 

\begin{corollary}\label{weak bore corollary}  
  Suppose that $h$ has the regularity \eqref{regularity weak h} and is a distributional solution of the height equation \eqref{w equation} for an Eulerian density as in \eqref{regularity weak rho}.  If $h$ is a bore with $H_\pm$ as in \eqref{bore asymptotics}, $\inf_{R_i }h_p > 0$ for $i = 1, \ldots, N$,  and 
  \begin{equation}
    H_+ \geq H_- = H ~\textrm{on } [-1,0] 
    \quad \textrm{or} \quad 
    H_+ \leq H_- = H ~ \textrm{on } [-1,0],
  \end{equation} 
  then $H_+ = H_- = H$. 
\end{corollary}
\begin{proof}  
  This follows from a straightforward adaptation of the proof of Theorem~\ref{bores theorem} and Lemma~\ref{bore identity}. Let $K = H_+$, and assume that $K \ge H$. The inequality $K_p \ge H_p$ holds by the same reasoning, as $K \geq H$ implies that $\rho_p (K-H)$ is nonpositive as a distribution.  The computations of $\flowforce$ can be carried out as before, only there will now be boundary terms on the interior interfaces when integrating by parts.  However, the jump conditions \eqref{weak K jump cond} cause them to cancel out completely, and so we arrive again at \eqref{auxiliary condition}.
\end{proof}

\subsection{Symmetry} \label{sec:symmetry}

Next, we turn our attention to the question of even symmetry for solitary stratified water waves.  The tool we use for this is the classical moving planes method introduced by Alexandroff \cite{alexandrov1962spheres} and then further developed by Serrin \cite{serrin1971symmetry}, Gidas, Nirenberg, and Ni~\cite{gidas1979symmetry}, Berestycki and Nirenberg \cite{berestycki1988monotonicity}, and many others (see also \cite{berestycki1991method} for a survey).  Specifically, we use a version due to C.~Li~\cite{li1991monotonicity} that treats fully nonlinear problems on cylindrical domains, and base our approach on the work of Maia~\cite{maia1997symmetry}, who applied Li's ideas to study the case of multiple-layered channel flows with uniform velocity at infinity.

\begin{theorem}[Symmetry] \label{symmetry theorem} 
  Let $(w,F) \in C^2_\bdd(\overline{R}) \times \R$ be a solution of equation \eqref{w equation} that is a wave of elevation 
  \begin{equation} 
    w> 0 \qquad \textup{in } R \cup T, \label{symmetry:wave of elevation} 
  \end{equation}
  supercritical, and satisfies the upstream (or downstream) condition
  \begin{equation}\label{symmetry: w localized}
    w, Dw, D^2 w \to 0 \quad \text{uniformly as } \quad q \to -\infty ~(\text{or } +\infty). \\
  \end{equation}
  Then, after a translation, $w$ is a symmetric and monotone solitary; i.e., there
  exists $q_* \in \R$ such that $q \mapsto w(q,\cdot)$ is even about $\{q = q_*\}$ and 
  \begin{equation}  
    {\pm w_q} > 0 \qquad \textrm{for } {\pm(q_*-q)} > 0, ~ -1 < p \leq 0. \label{symmetry: monotonicity} 
  \end{equation}
\end{theorem}
\begin{remark}\label{remark symmetry}
  It is worth emphasizing that Theorem~\ref{symmetry theorem} only assumes that $w \to 0$ as $q \to -\infty$ instead of the much stronger condition $w \in C^2_0(\overline R)$ which we impose elsewhere, and which is the standard hypothesis for moving planes arguments. This is in line with the hypotheses of Maia's symmetry result \cite{maia1997symmetry} for channel flows, which states that either there is an axis of symmetry or the solution is monotone in the entire strip. 
  We are able to rule out the second possibility in our setting using Theorem~\ref{bores theorem} on the nonexistence of monotone bores.
\end{remark}
\begin{remark}
  While Theorem~\ref{symmetry theorem} is of independent interest, we note that it is used in the proof of Theorem~\ref{main existence theorem} only in so far as it permits us to infer that the small-amplitude waves of elevation constructed in Section~\ref{small-amplitude section} are in fact monotone.
\end{remark}

To apply the moving plane method, we consider the usual reflected functions  
\[ h^\lambda(q,p) := h(2\lambda-q, p),\]
where $\{q=\lambda\}$ is the axis of reflection. We let $v^\lambda$ denote the difference
\[ v^\lambda := h^\lambda - h.\]
Thus $\{q=q_*\}$ is an axis of even symmetry if and only if $v^{q_*}$ vanishes identically.  We will work with the $\lambda$-dependent sets 
\begin{equation}
  R^\lambda := \{ (q,p) \in R : q < \lambda \}, 
  \quad 
  T^\lambda := \{ (q, 0) : q < \lambda \}, 
  \quad 
  B^\lambda := \{ (q,-1) : q < \lambda\}.
\end{equation}

If $h$ is a solution of the height equation \eqref{height equation}, then for each $\lambda$, $v^\lambda$ solves the PDE
\begin{equation}\label{difference equation v}
  \left\{ 
  \begin{alignedat}{2}
    \mathscr{L} v^{\lambda} &= 0 &\qquad& \textrm{in } R^\lambda, \\
    \mathscr{B} v^\lambda &= 0 && \textrm{on } T^\lambda, \\
    v^\lambda &= 0 && \textrm{on } B^\lambda,
  \end{alignedat}
  \right.
\end{equation}
where $\mathscr{L}$ is defined as 
\begin{equation}\label{defn L}
  \begin{split}
    \mathscr{L} := & {1\over h^\lambda_p}\p_q^2 - {2h^\lambda_q \over (h^\lambda_p)^2}\p_q\p_p + {1+ (h^\lambda_q)^2 \over (h^\lambda_p)^3}\p_p^2 + {h_{pp}(h^\lambda_q + h_q) - 2h^\lambda_p h_{pq} \over (h^\lambda_p)^3} \p_q \\
    & + \left( \beta(-p) -  \frac{1}{F^2} \rho_p h \right)\left[ (h^\lambda_p)^2 + h^\lambda_p h_p + h^2_p \right]  \frac{1}{(h^\lambda_p)^3} \p_p - \frac{1}{F^2} \rho_p,
  \end{split}
\end{equation}
and $\mathscr{B}$ is given by
\begin{equation}\label{defn B}
  \mathscr{B} := {h^\lambda_q + h_q \over 2h^2_p} \p_q - {(h^\lambda_p + h_p)(1+(h^\lambda_q)^2) \over 2h^2_p (h^\lambda_p)^2} \p_p + \frac{1}{F^2} \rho.
\end{equation}
See \cite{walsh2009symmetry} for the details of the calculation.  Indeed, from \cite[Lemma~1]{walsh2009symmetry} we know that $\mathscr{L}$ is uniformly elliptic. 

As in Section~\ref{fredholm section}, the coefficients in $\mathscr{L}$ and $\mathscr{B}$ may have ``bad'' signs which do not satisfy the hypotheses of the maximum principle (cf.\ Theorem~\ref{max principle}). Since the maximum principle is the main tool in classical moving planes arguments, this is a major obstacle. For supercritical waves, we can overcome this obstacle  by using a slight modification of the function $\Phi$ defined in \eqref{def Phi}. Fix $0 < \varepsilon \ll 1$ and let $\Psi = \Psi(p; F, \varepsilon)$ be 
the solution of the initial value problem
\begin{equation} 
  \label{aux problem Psi} 
  \left( \frac{\Psi_p}{H_p^3} \right)_p - \frac{1}{F^2} (\rho_p- \varepsilon) \Psi = 0 \textrm{ in } (-1, 0), 
  \qquad \Psi(-1) = \varepsilon, ~ \Psi_p(-1) = 1.
\end{equation}
\begin{lemma}\label{lem Psi}
  If $F > \Fcr$, then, for $\varepsilon > 0$ sufficiently small,
  \begin{align}
    \label{Psi pos}
    \Psi > 0 \textup{ for } {-1} < p \le 0,
    \qquad 
    \Psi_p > 0 \textup{ for } {-1} \le p \le 0,
  \end{align}
  and 
  \begin{align}
    \label{sign robin Psi}
    -\frac{\Psi_p}{H_p^3} + \frac {1}{F^2}\rho \Psi < 0 
    \textup{ on } p=0.
  \end{align}
  \begin{proof}
    The proof is identical to the proof of Lemma~\ref{lem Phia}.
  \end{proof}
\end{lemma}

As in the discussion before Lemma~\ref{strong invertibility lemma}, the change of variables $v =: \Psi u$ transforms the linear elliptic equation $\F_w(0,F) v = (f_1,f_2)$ into an equation for $u$ whose zeroth order coefficients have the 
correct sign for the applications of the maximum principle and Hopf maximum principle (Theorem~\ref{max principle}). The change of variables $v =: \Phi u$ used in Section~\ref{fredholm section} makes the zeroth order coefficient in the elliptic operator vanish; here we define $\Psi$ in a slightly more complicated way in order for this coefficient to be strictly negative.

\begin{lemma} \label{first step moving planes lemma}  
  Under the hypotheses of Theorem~\ref{symmetry theorem}, there exists $K > 0$ such that 
  \begin{equation} 
    v^\lambda \ge 0 \textup{ in } R^\lambda \quad \textup{for all } \lambda < -K, \label{v lambda sign cond} 
  \end{equation}
  and 
  \begin{equation} 
    h_q \geq 0 \textup{ in } R^\lambda \quad \textup{for all } \lambda < -K. \label{v lambda monotonicity prop} 
  \end{equation}
\end{lemma}
\begin{proof}
  Let $\Psi$ be defined as in \eqref{aux problem Psi}.  Recall that this means in particular that 
  \begin{equation}
    \left(  -\frac{1}{H_p^3} \Psi_p + \frac{1}{F^2} \rho \Psi\right)\bigg|_{p=0} < 0 \qquad \textrm{for } 0 \leq \frac{1}{F^2} < \frac{1}{(F_{\mathrm{cr}})^2}.
  \end{equation}
  Also, by taking $\varepsilon > 0$ sufficiently small, we may assume that $\Psi > \varepsilon$ on $(-1,0)$.

  We are therefore justified in defining $u^\lambda$ by $v^\lambda = u^\lambda \Psi$.  
  A calculation shows that $u^\lambda$ satisfies
  \begin{equation}\label{difference equation w}
    \left\{ 
    \begin{alignedat}{2}
      \tilde{\mathscr{L}} u^\lambda &= 0 &\qquad& \textrm{in } R^\lambda,\\
      \tilde{\mathscr{B}} u^\lambda &= 0 && \textrm{on } T^\lambda, \\
      u^\lambda &= 0 && \textrm{on } B^\lambda,
    \end{alignedat}
    \right.
  \end{equation}
  where
  \begin{align*}
    \tilde{\mathscr{L}} u^\lambda &:= \Psi (\mathscr{L}_{\textrm{P}} u^\lambda) + \left[ {2(1 + (h^\lambda_p)^2) \over (h^\lambda_p)^3} \Psi_p \right] u^\lambda_p - \left[ {2h^\lambda_p \over (h^\lambda_p)^2} \Psi_p \right] u^\lambda_q 
    + Z
    u^\lambda, \\
    \tilde{\mathscr{B}} u^\lambda &:= \Psi \mathscr{B}u^\lambda + (\mathscr{B}_{\textrm{P}}  \Psi) u^\lambda,
  \end{align*}
  with the zeroth order coefficient $Z$ given by
  \begin{align*}
    Z &:= {1+ (h^\lambda_q)^2\over (h^\lambda_p)^3}\Psi_{pp} - \frac{1}{F^2} \rho_p \Psi  \\ & \qquad +  \left(h_{qq}(h^\lambda_p + h_p) - 2h_q h_{pq} + \left(\beta(-p) - \frac 1{F^2} \rho_ph\right)\left[ (h^\lambda_p)^2 + h^\lambda_p h_p + h^2_p \right] \right) \frac{\Psi_p}{(h^\lambda_p)^3},
  \end{align*}
  and $\mathscr{L}_{\textrm{P}}$ and $\mathscr{B}_{\textrm{P}}$ are the principal parts of $\mathscr{L}$ and $\mathscr{B}$,
  \begin{align*}
    \mathscr{L}_{\textrm{P}} &:= {1\over h^\lambda_p}\p_q^2 - {2h^\lambda_q \over (h^\lambda_p)^2}\p_q\p_p + {1+ (h^\lambda_q)^2 \over (h^\lambda_p)^3}\p_p^2,\\
    \mathscr{B}_{\textrm{P}} &:=  {h^\lambda_q + h_q \over 2h^2_p} \p_q - {(h^\lambda_p + h_p)(1+(h^\lambda_q)^2) \over 2h^2_p (h^\lambda_p)^2} \p_p.
  \end{align*}

  We claim that there exists some $K>0$ large enough so that  $u^\lambda\ge0$ in $R^\lambda$ for all $\lambda \leq -K$, which then implies \eqref{v lambda sign cond}. 
  Assume on the contrary that no matter how large $K$ is, there exists some $\lambda_0 \leq -K$ such that $u^{\lambda_0}$ takes on a negative value in $R^{\lambda_0}$.  By hypothesis, $h$ is a wave of elevation, and so clearly the same is true of $h^\lambda$ for any $\lambda$.  Now $u^\lambda$ vanishes on the vertical line segment $\{ q = \lambda\}$, and
  \begin{equation}
    u^\lambda = {h^\lambda - h \over \Psi} >  {H - h\over \Psi},
  \end{equation}
  with the right-hand side of the above inequality limiting to $0$ as $q \to -\infty$.  Thus if $u^{\lambda_0}$ is negative somewhere in $R^{\lambda_0}$ then there must exist some $(q_0, p_0) \in R^{\lambda_0} \cup T^{\lambda_0} $ such that
  \begin{equation}\label{min u}
    u^{\lambda_0}(q_0, p_0) = \inf_{R^{\lambda_0}} u^{\lambda_0} < 0.
  \end{equation}
We consider separately two cases.

  \bigskip
  {\bf Case 1:} $(q_0, p_0) \in R^{\lambda_0}$.  
  Then $u^{\lambda_0}$ attains its (global) minimum at an interior point, 
  so
  \begin{equation*}
    \nabla u^{\lambda_0}(q_0, p_0) = 0,
  \end{equation*}
  and hence
  \begin{align}
    \label{nice grad}
    0 = \nabla u^{\lambda_0}(q_0,p_0)  = \left[{\nabla v^{\lambda_0} \over \Psi} - { \Psi_p \over \Psi^2} v^{\lambda_0}\right](q_0,p_0).
  \end{align}
    In light of \eqref{symmetry: w localized}, for each $\delta > 0$, we may take $K$ sufficiently large so that 
  \begin{equation}\label{asymp h}
    \|w\|_{C^2(R^{-K})} = \|h - H\|_{C^2(R^{-K})} < \delta,
  \end{equation}
  from which we have the chain of inequalities 
  \begin{align}
    \label{nice chain}
    H(p_0) < h^{\lambda_0}(q_0,p_0) < h(q_0,p_0) < H(p_0) + \delta.
  \end{align}
  Moreover \eqref{asymp h} and \eqref{nice grad} lead to the bounds on
  $v^{\lambda_0}$ and $\nabla v^{\lambda_0}$ at $(q_0, p_0)$:
  \begin{align*}
    | v^{\lambda_0}(q_0,p_0)| < \delta, \qquad | \nabla v^{\lambda_0}(q_0,p_0)| = \left| {\Psi_p(p_0) \over \Psi(p_0)}v^{\lambda_0}(q_0,p_0) \right| < C\delta,
  \end{align*}
  where $C$ depends only on $\varepsilon$. Therefore in terms of $h^{\lambda_0}$ we have that at $(q_0, p_0)$,
  \begin{equation}\label{asymp hlambda}
    \big\lvert h^{\lambda_0} - H \big\rvert < \delta, 
    \quad 
    \big\lvert \nabla h^{\lambda_0} - \nabla H \big\rvert < C\delta.
  \end{equation}
  From \eqref{asymp h} and \eqref{asymp hlambda} we conclude that 
  \begin{equation*}
    Z(p_0) =  \left( {1\over H^3_p}\Psi_{pp} - {3H_{pp} \over H^4_p} \Psi_p - \frac{1}{F^2} \rho_p \Psi\right)(p_0) +\mathcal{O}(\delta).
  \end{equation*}
  But, recalling the ODE satisfied by $\Psi$ \eqref{aux problem Psi}, we know that
  \[
  {1\over H^3_p}\Psi_{pp} - {3H_{pp} \over H^4_p} \Psi_p - \frac{1}{F^2} \rho_p\Psi = -\varepsilon \frac{1}{F^2}\Psi < 0.
  \]
  Therefore, by taking $K$ sufficiently large, and hence $\delta$ sufficiently small, we can guarantee that $Z < 0$ at $p_0$.  Applying the maximum principle to \eqref{difference equation w} at $(q_0, p_0)$ then leads to a contradiction.

  \bigskip
  {\bf Case 2:} $(q_0, p_0) \in T^{\lambda_0}$, that is $p_0 = 0$. Applying the Hopf lemma, we see that
  \begin{equation}\label{hopf}
  u_q^{\lambda_0}(q_0, 0) = 0, \quad  u_p^{\lambda_0}(q_0, 0) < 0.
  \end{equation}
  From the first of these equalities it follows that
  \begin{equation} 
    h^{\lambda_0}_q(q_0,0) = h_q(q_0,0).\label{h_q lambda = h_q} 
  \end{equation}
  Moreover, arguing as in the previous case, \eqref{nice chain} and \eqref{asymp h} show that for $K$ large enough, 
  \begin{equation}\label{asymp hlambda 2}
    \big\lvert h(q_0, 0) - H(0) \big\rvert < \delta, \quad \big\lvert h_p(q_0, 0) - H_p(0) \big\rvert < \delta, \quad \big\lvert h^{\lambda_0}(q_0, 0) - H(0) \big\rvert < \delta.
  \end{equation}
  Since both $h$ and $h^{\lambda_0}$ solve the height equation \eqref{height equation}, we can evaluate the boundary condition at $(q_0, 0)$ to obtain
  \begin{equation*}
    \frac{1+h_q^2}{2h_p^2}  + \frac{1}{F^2} \rho h = \frac{1+(h^\lambda_q)^2}{2(h^\lambda_p)^2}  + \frac{1}{F^2} \rho h^\lambda.
  \end{equation*}
  Together, \eqref{h_q lambda = h_q} and \eqref{asymp hlambda 2} furnish the estimate
  \begin{equation}\label{asymp hlambda 3}
    \big\lvert h_p^{\lambda_0}(q_0, 0) - h_p(q_0, 0) \big\rvert < C\delta.
  \end{equation}

    Combing \eqref{asymp hlambda 2} and \eqref{asymp hlambda 3}, the boundary condition in \eqref{difference equation w} can then be written as
    \begin{equation}\label{bc w}
      \tilde{\mathscr{B}}u^{\lambda_0} = \left( - {(h^\lambda_p + h_p)(1+(h^\lambda_q)^2) \over 2h^2_p (h^\lambda_p)^2} u^{\lambda_0}_p \right) \Psi + \left[  -\frac{1}{H_p^3} \Psi_p + \frac{1}{F^2} \rho \Psi + \mathcal{O}(\delta) \right] u^{\lambda_0} = 0
    \end{equation}
    on $T^\lambda$. In view of \eqref{sign robin Psi}, this means that for $K$ sufficiently large, the coefficient of $u^{\lambda_0}$ above is negative, and therefore from \eqref{min u} we know that $u_p^{\lambda_0}(q_0,0) > 0$, which is a contradiction to \eqref{hopf}.  

    \bigskip
    Thus we have confirmed that there exists a $K$ sufficiently large so that \eqref{v lambda sign cond} is satisfied.  
    The monotonicity property \eqref{v lambda monotonicity prop} then follows immediately:  for any $\lambda \leq -K$, we have $v^\lambda(\lambda, p) = 0$ and therefore $v^\lambda_q(\lambda, p) \leq 0$, which implies that  $h_q(\lambda, p) \geq 0$. 
  \end{proof}

Following the moving plane method, we complete the proof of Theorem \ref{symmetry theorem} by employing a continuation argument.  
The general procedure is quite standard, having been worked out in similar contexts by Hur~\cite{hur2008symmetry}, in the case of constant density water, and Maia~\cite{maia1997symmetry}, for stratified channel flow. Our situation is different in that, unlike Hur, we allow a priori for there to be no decay in one direction, and, unlike Maia, we have at our disposal Theorem~\ref{bores theorem}.
\begin{proof}[Proof of Theorem~\ref{symmetry theorem}]
  Let $\lambda_*$ be defined by
  \begin{equation*}
    \lambda_* := \sup{\left\{\lambda_0:\ 
    v^{\lambda} > 0 \text{ in } R^\lambda 
    \text{ for all } \lambda < \lambda_0 \right \}},
  \end{equation*}
  where the above set is nonempty in light of the previous lemma.  
  
\bigskip
  {\bf Case 1.} $\lambda_* < +\infty$. By continuity we know that $v^{\lambda_*} \geq 0$ on $R^{\lambda_*}$. Looking at the elliptic system \eqref{difference equation v} that $v^{\lambda_*}$ satisfies in $R^{\lambda_*}$, we can therefore apply the strong maximum principle to conclude that either $v^{\lambda_*} > 0$ or $v^{\lambda_*} \equiv 0$ in $R^{\lambda_*}$.

  We claim that $v^{\lambda_*} \equiv 0$ in $R^{\lambda_*}$. Seeking a contradiction, assume instead that  $v^{\lambda_*} > 0$ in $R^{\lambda_*}$.  
  The maximality of $\lambda_*$ implies that there exist sequences $\{ \lambda_k \}$ and $\{ (q_k, p_k) \}$ with  $\lambda_k \searrow \lambda_*$ and $(q_k, p_k) \in \overline{R^{\lambda_{k}}}$ such that
  \[
  v^{\lambda_k}(q_k, p_k) = \inf_{R^{\lambda_k}} v^{\lambda_k} < 0.
  \]
  Since $v^{\lambda_k} = 0$ on $B^{\lambda_k}$ and $\{ q = \lambda_k \}$, the strong maximum principle forces $(q_k, p_k) \in T^{\lambda_k}$. This trivially implies 
  \begin{equation}\label{deriv min}
    v^{\lambda_k}_p(q_k, 0) \ge 0, \qquad  \textrm{and} \qquad v^{\lambda_k}_q(q_k, 0) = 0.
  \end{equation}

  Next, we prove that $\{q_k\}$ is bounded from below.  Were this not true, then we have $q_k < -K$ for all $k$ sufficiently large, where $K$ is chosen as in \eqref{v lambda sign cond}. Consider once more the function $u^{\lambda_k} := v^{\lambda_k}/\Psi$ introduced in Lemma~\ref{first step moving planes lemma}. 
  Clearly $u^{\lambda_k}$ satisfies \eqref{difference equation w} in $R^{\lambda_k}$, but from \eqref{bc w} we see that $\tilde{\mathscr{B}} u^{\lambda_k}(q_k, 0) > 0$, a contradiction.

  Thus $\{q_k\}$ is indeed bounded from below by $-K$.    It is also obviously bounded from above by $\lambda_1$.  Up to a subsequence, therefore, 
  \begin{equation}
    (q_k, 0) \to (q_*, 0) \in \overline{T^{\lambda_*}} \quad \text{as } k \to \infty
  \end{equation}
  for some $q_* \in [-K, \lambda_*]$.   The fact that $v^{\lambda_*} > 0$ in $R^{\lambda_*}$ 
  forces
  \[
  \lim_{k\to\infty} v^{\lambda_k}(q_k, 0) = v^{\lambda_*}(q_*, 0) = 0.
  \]
  If $q_* < \lambda_*$, then we have by continuity that
  \begin{align}
    \label{good q*}
    v^{\lambda_*}(q_*, 0) =  v^{\lambda_*}_q(q_*, 0) = 0,
  \end{align}
  and, furthermore, from the Hopf lemma, $v^{\lambda_*}_p(q_*, 0) < 0$. Recalling the definition of the boundary operator $\mathscr{B}$ in \eqref{defn B}, these inequalities show that $\mathscr{B} v^{\lambda_*}(q_*,0) > 0$, which is impossible because $v^{\lambda_*}$ solves \eqref{difference equation v}.

  Therefore $(q_*,0)$ must be a corner point of $R^{\lambda_*}$, i.e., $q_* = \lambda_*$. From \eqref{good q*}, we then see $h^{\lambda_*}_q(\lambda_*, 0) = 0$. We can rewrite the top boundary condition in \eqref{difference equation v} as
  \begin{equation}\label{bc rewritten}
    {(h^\lambda_p)^2(h^\lambda_q + h_q)} v^{\lambda}_q - {(h^\lambda_p + h_p)(1+(h^\lambda_q)^2) } v^{\lambda}_p + 2\frac{1}{F^2} \rho h^2_p (h^\lambda_p)^2 v = 0.
  \end{equation}
  Letting $\lambda = \lambda_*$, differentiating the above equality with respect to $q$, and then evaluating it at $(\lambda_*, 0)$ we obtain
  \[
  2 h_p(\lambda_*,0)v^{\lambda_*}_{qp}(\lambda_*,0) = 0,
  \]
  where we have used the identities
  \[
  h^{\lambda_*}_q(\lambda_*, 0) = -h_q(\lambda_*, 0), \quad h^{\lambda_*}_p(\lambda_*, 0) = h_p(\lambda_*, 0), 
  \quad 
  h^{\lambda_*}_{qp}(\lambda_*, 0) = -h_{qp}(\lambda_*, 0).
  \]
  From the no stagnation assumption \eqref{no stag h}, we can therefore conclude
  \[
  v^{\lambda_*}_{qp}(\lambda_*, 0) = 0.
  \]
  Moreover, because $v^{\lambda_*}(\lambda_*, \cdot)$ vanishes identically, it follows that
  \[
  v^{\lambda_*}_p(\lambda_*, 0) = v^{\lambda_*}_{pp}(\lambda_*, 0) = 0.
  \]
  Finally, by using the PDE to express $v^{\lambda_*}_{qq}$ in terms of $v^{\lambda_*}, v^{\lambda_*}_q, v^{\lambda_*}_p, v^{\lambda_*}_{qp}$ and $v^{\lambda_*}_{pp}$, we see that
  \[
  v^{\lambda_*}_{qq}(\lambda_*, 0) = 0.
  \]

  The previous paragraph demonstrates that $v^{\lambda_*}$, together with all its derivatives up to second order,  vanishes at the corner point $(\lambda_*, 0)$. Since $v^{\lambda_*}$ solves the PDE in the domain $R^{\lambda_*}$, this violates the Serrin edge point lemma (see Theorem~\ref{max principle}). Therefore we must have $v^{\lambda_*} \equiv 0$ in $R^{\lambda_*}$, and hence $h$ (and $w$) are symmetric about the axis $q_* := \lambda_*$. This proves the first part of the theorem.

  Next, consider the strict monotonicity of $h$. For $\lambda < \lambda_*$ we have from the definition of $\lambda_*$ that $v^{\lambda} > 0$ in $R^\lambda$. As $v^{\lambda}$ vanishes on the right boundary of $R^\lambda$, it attains its minimum on $\overline{R^\lambda}$ there.  The Hopf maximum principle then implies 
  \begin{equation}
    h_q(\lambda, p) = -{1\over 2} v^{\lambda}_q(\lambda, p) > 0 \quad \text{ for } \lambda < \lambda_*, ~ -1  < p < 0.
    \label{symmetry: hq pos on right boundary} 
  \end{equation}

  Naturally, on the top boundary $T^{\lambda_*}$, $h_q \geq 0$ by continuity. Seeking a contradiction, suppose that $h_q$ vanishes at some point $(\lambda, 0) \in T^{\lambda_*}$.  Then, from differentiating the boundary condition \eqref{bc rewritten} with respect to $q$, evaluating the result at $(\lambda, 0)$, and using the identities
  \begin{gather*}
    h_q(\lambda, 0) = - h^{\lambda}_q(\lambda, 0) = 0, \qquad  v^\lambda(\lambda, 0) = v^\lambda_q(\lambda, 0) = 0,\\
    h_p(\lambda, 0) = h^\lambda_p(\lambda, 0), \qquad h_{qp}(\lambda, 0) = -h^\lambda_{qp}(\lambda, 0),
  \end{gather*}
  we conclude  
  \[
  2 h_p(\lambda,0)v^{\lambda}_{qp}(\lambda,0) = 0,
  \]
  and hence 
  \[
  v^{\lambda}_{qp}(\lambda,0) = 0.
  \]
  Note that, as before, we have
  \[
  v^{\lambda}_p(\lambda, 0) = v^{\lambda}_{pp}(\lambda, 0) = 0.
  \]
  Thus, using the PDE to solve for $v^\lambda_{qq}$, we get $v^\lambda_{qq}(\lambda,0) = 0$. But $v^{\lambda}$ satisfies \eqref{difference equation v} and so the Serrin edge point lemma is violated.  We infer, therefore, that $h_q > 0$ on $T^{\lambda_*}$.  Combining this with \eqref{symmetry: hq pos on right boundary} confirms that
  $h_q  > 0$ in $R^{\lambda_*} \cup T^{\lambda_*}$ as desired.
  The same holds for $w_q$. As $\{ q = \lambda_*\}$ is an axis of even symmetry, the proof \eqref{symmetry: monotonicity} is complete.

\bigskip
  {\bf Case 2.} If $\lambda_* = +\infty$, then $v^{\lambda} \ge 0 \text{ in } R^\lambda \text{ for all } \lambda$. Since $v^\lambda$ solves \eqref{difference equation v} in $R^\lambda$, an application of the strong maximum principle ensures that $v^\lambda > 0$ in $R^\lambda$ for all $\lambda$. Then, arguing as in the preceding paragraph, we see that $h_q > 0$ in $R^\lambda \cup T^\lambda$ for any $\lambda$, and hence $h_q>0$ in $R$, indicating that $h$ is a strictly monotone bore solution. In particular, the pointwise limits $h(q,p) \to H_\pm(p)$ as $q \to \pm\infty$ in \eqref{bore asymptotics} hold, with $H_- = H$ and $H_ + > H$ for $p > -1$, violating Theorem~\ref{bores theorem}. Therefore we must have $\lambda_* < +\infty$, and thus the previous argument completes the proof of the theorem.
\end{proof}

\subsection{Asymptotic monotonicity and nodal properties}\label{nodal section}

In Section~\ref{small-amplitude section}, we will establish the existence of a curve of small-amplitude supercritical waves of elevation bifurcating from the critical laminar flow. By the results in Section~\ref{sec:symmetry}, these waves are symmetric and monotone.
The global bifurcation theory will require us to show that these properties persist away from the point of initial bifurcation.  In particular, this is a critical component of the proof that the solution set $\F^{-1}(0)$ is locally compact (cf.\ Lemma~\ref{lem:compact}).  

Notice, however, that monotonicity is neither an open nor closed property in the natural topology of our function spaces.  It will therefore be necessary to exploit some structure of the equation itself, which will mostly come in the form of maximum principle arguments.  We will actually consider a collection of sign conditions \eqref{elevation nodal properties} that together imply monotonicity but also define a subset of $\mathscr{F}^{-1}(0)$ that is both open and closed in an appropriate sense. 

An analysis of this kind is a standard part of many global bifurcation arguments, where conditions of the type \eqref{elevation nodal properties} are referred to as ``nodal properties.''  Typically, they are used to rule out undesirable topological alternatives such as the existence of closed loops of solutions. In our theory, however, they play a much different role as we will see in Section~\ref{global bifurcation section}.

Before discussing the nodal properties, we first divide $R$ into two regions: a finite rectangle and two semi-infinite tails where $w$ is small.  The finite rectangle can be dealt with in the same way as periodic solutions were in~\cite[Section~5]{walsh2009stratified}.  For the tail, we will follow the strategy of~\cite[Section~2.2]{wheeler2015pressure} by considering small solutions of \eqref{w equation} in the half strip
\[ R^+ := \{ (q,p) \in R : q > 0 \}\]
with boundary components 
\[  L^+ := \{ (0, p) : p \in [-1,0] \}, \qquad T^+ := \R_+ \times \{ 0 \}, \qquad B^+ := \R_+ \times \{ -1 \}.\]

\begin{proposition}[Asymptotic monotonicity] \label{monotonicity small-amplitude proposition} 
  There exists $\delta >0$ such that, if $w\in C_{\mathrm{b}}^3({R^+}) \cap C_0^1(\overline{R^+})$ is a solution of equation \eqref{w equation} in $R^+$, $F > \Fcr$, and 
  \[   
    \| w \|_{C^2(R^+)} < \delta,
  \]
  then $w_q$ exhibits the following monotonicity property:
  \begin{equation} 
    \textup{If } {\pm w_q} \leq 0 \textup{ on } L^+, 
    \textup{ then } {\pm w_q} < 0 \textup{ in } R^+ \cup T^+. 
    \label{monotonicity property} 
  \end{equation}
\end{proposition}

\begin{proof}[Proof of Proposition~\ref{monotonicity small-amplitude proposition}]
We will use a maximum principle argument applied to $h_q=w_q$.  Due to the translation invariance, we can ``quasilinearize'' the height equation by differentiating it with respect to $q$. 
This leads to a linear elliptic PDE for $v := h_q$ with coefficients depending on $h$: 
\begin{equation} \label{hq height equation}
  \left\{ 
  \begin{alignedat}{2}
    \left(-\frac{h_q v_q}{h_p^2} + \frac{(1+h_q^2)v_p}{h_p^3}\right)_p + \left( \dfrac{v_q}{h_p} - \frac{h_q v_p}{h_p^2} \right)_q -  \frac{1}{F^2}  \rho_p v  
    &= 0 && \textrm{in } R^+ \\ 
    \frac{h_q v_q}{2h_p^2} - \frac{(1+h_q^2) v_p}{h_p^3} +  \frac{1}{F^2}  \rho v &= 0 &\quad& \textrm{on } T^+ \\
    v &= 0 &\quad& \textrm{on } B^+. 
  \end{alignedat} 
  \right. 
\end{equation}
Let $v =: \Psi u$, where $\Psi$ is chosen according Lemma~\ref{lem Psi}.
We can then rewrite \eqref{hq height equation} in terms of $u$.  For instance, the equation in the interior becomes
\begin{align} 
  \label{interior u}
  \begin{aligned}
    0 & =  \left( \left( \frac{\Psi}{h_p^3} + \frac{\Psi h_q^2}{h_p^3}\right) u_p - \frac{\Psi h_q}{h_p^2} u_q \right)_p + \left( \frac{\Psi}{H_p} u_q - \frac{\Psi h_q}{h_p^2} u_p \right)_q \\
    & \qquad +  \frac{(1+h_q^2) \Psi}{h_p^3} u_p - \frac{h_q \Psi_p}{h_p^2} u_q  + \left( \left( \frac{(1+h_q^2) \Psi_p}{h_p^3}\right)_p -\left( \frac{h_q \Psi_p}{h_p^2} \right)_q - \frac{1}{F^2}  \rho_p \Psi  \right) u.   
  \end{aligned}
\end{align}
Taking $\|h - H\|_{C^2(R^+)} = \|w\|_{C^2(R^+)}$ sufficiently small ensures that this represents a uniformly elliptic operator acting on $u$ on the set $R^+$.  The key point is the zeroth order coefficient, which can be rewritten as 
\begin{align*}  
  -\left( \frac{h_q \Psi_p}{h_p^2} \right)_q + \left( \left(\frac{1}{h_p^3} - \frac{1}{H_p^3}\right) \Psi_p + \frac{h_q^2}{h_p^3} \Psi_p\right)_p  -\frac{1}{F^2} \varepsilon \Psi,
\end{align*}
has the correct sign to apply the maximum principle
since the first two terms on the right-hand side can be controlled by the final term by taking $\| h - H \|_{C^2}$ small.  Thus \eqref{interior u} has
the form
\[ 
  \sum_{i,j} \partial_i (a_{ij} \partial_j u) + \sum_i b_i \partial_i u + c u = 0,
\]
where $a_{ij}$ is a symmetric positive definite matrix and $c < 0$, both with uniform bounds for $\| h - H\|_{C^2}$ sufficiently small.  

Likewise, the boundary conditions on $T^+$ in \eqref{hq height equation} can be written
\begin{align*} 
  - \frac{1+h_q^2}{h_p^3} \Psi u_p + \frac{h_q}{h_p^2} \Psi u_q 
  + \left( -\frac{\Psi_p}{H_p^3} + \frac 1{F^2} \rho \Psi - \left( \frac{1}{h_p^3} - \frac{1}{H_p^3}\right)\Psi_p \right) u
  = 0.
\end{align*}
By \eqref{sign robin Psi}, the coefficient of $u$ is negative for $\| h - H \|_{C^1}$ small enough.  Thus this represents a uniformly oblique boundary condition with the correct sign,
and we can apply the maximum principle to conclude that \eqref{monotonicity property} holds.  
\end{proof}

With Proposition~\ref{monotonicity small-amplitude proposition} in hand, we now consider 
the ``nodal properties''
\begin{subequations} \label{elevation nodal properties}
  \begin{alignat}{2}
    w_{q}  &< 0 &\qquad& \textrm{in } R^+ \cup T^+, \label{elevation nodal: hq} \\
    w_{qq}  &< 0 && \textrm{on } L^+, \label{elevation nodal: hqq} \\
    w_{qp}  &< 0 && \textrm{on } B^+,   \label{elevation nodal: hqp}\\
    w_{qqp} & < 0 && \textrm{at } (0,-1), \label{elevation nodal: corner hqqp} \\
    w_{qq}  &< 0 && \textrm{at } (0,0). \label{elevation nodal: corner hqq}
  \end{alignat} 
\end{subequations}

\begin{lemma}  \label{nodal: monotone implies nodal lemma} 
  Let $(w,F)$ be a solution of equation \eqref{w equation} where $w \in C^3_{\bdd,\even}(\overline{R}) \cap C^1_0(\overline{R})$ is monotone in the sense that $w_q < 0$ in $R^+ \cup T^+$. Then $w$ exhibits the nodal properties \eqref{elevation nodal properties}.  
\end{lemma}
\begin{proof}  
As in the proof of Proposition~\ref{monotonicity small-amplitude proposition}, $v := w_q = h_q$ satisfies the uniformly elliptic PDE \eqref{hq height equation}. The strategy is now to use various maximum principle arguments to derive \eqref{elevation nodal properties}.  Again, the zeroth order term in \eqref{hq height equation} comes with an adverse sign. Unlike in Proposition~\ref{monotonicity small-amplitude proposition}, we are not assuming that $w$ is small. Instead, we are saved by the following observation.   First, $w$ is even in $q$ and vanishes identically on the bed, hence
\[ 
  v = 0 \qquad \textrm{on } L^+ \cup B^+.
\]
Second, by assumption we have that 
\[ 
  v < 0 \qquad \textrm{in } R^+ \cup T^+.
\]
Together, these two statements mean that we \emph{can} apply the Hopf lemma, and the Serrin edge point lemma (Theorem~\ref{max principle}\ref{edge point}) 
at the points on $\overline{L^+ \cup B^+}$ where $v$ achieves its maximum value.

With that in mind, let us consider in order the nodal properties \eqref{elevation nodal properties}.  First, \eqref{elevation nodal: hq} is satisfied by hypothesis.  
The Hopf lemma applied to $L^+$ and $B^+$ shows that \eqref{elevation nodal: hqq} and \eqref{elevation nodal: hqp} hold, respectively.  

Now consider the corners.  Since $v$ vanishes identically on $L^+ \cup B^+$, we have that 
\[ 
  v = v_p = v_q = v_{pp} = v_{qq} = 0
  \qquad \text{at $(0,-1)$}.
\]
Thus, by the Serrin edge point lemma, \eqref{elevation nodal: corner hqqp} holds.  

Likewise, at the upper left corner point $(0,0)$ we have 
\[ v(0,0) = v_p(0,0) = 0.\]
Using these facts, differentiating the top boundary condition in \eqref{height equation} twice in $q$, and evaluating at $(0,0)$, one arrives at   
\[  \frac{1}{2h_p^2} v_q^2 + \frac{1}{h_p^3} v_{qp} + \frac 1{F^2} \rho v_q = 0 \qquad \textrm{at } (0,0).\]
Seeking a contradiction, suppose that $v_q(0,0) = w_{qq}(0,0) = 0$.  Then from the evenness of $w$ and the line above we have 
\[ v, \, v_p, \, v_q, \, v_{pp},  \, v_{qp} = 0 \qquad \textrm{at } (0,0).\]
Finally, we can use \eqref{hq height equation} to write $v_{qq}$ in terms of the quantities above, which reveals that $v_{qq}(0,0) = 0$   as well.  This is a clear contradiction of the Serrin edge point lemma. We infer, therefore, that \eqref{elevation nodal: corner hqq} holds and the proof is complete.  
\end{proof}

Notice that the hypotheses of Lemma~\ref{nodal: monotone implies nodal lemma} do not require that the solution be supercritical.  

\begin{lemma}[Open property] \label{nodal open lemma} 
  Let $(w,F), (\tilde w, \tilde F)$ be two supercritical solutions of \eqref{w equation} with $w,\tilde w \in C^3_{\bdd,\even}(\overline{R}) \cap C^1_0(\overline{R})$. If $w$ satisfies the nodal properties \eqref{elevation nodal properties}, then there exists $\varepsilon = \varepsilon(w) > 0$ such that 
  \begin{align*}
    \| w - \tilde{w} \|_{C^3(R)} + \abs{F-\tilde F}< \varepsilon 
  \end{align*}
  implies that $\tilde{w}$ also satisfies \eqref{elevation nodal properties}.  
\end{lemma}
\begin{proof}
  According to Lemma~\ref{nodal: monotone implies nodal lemma}, to prove that $\tilde{w}$ satisfies \eqref{elevation nodal properties} we need only confirm that $\tilde{w}_q < 0$ in $R^+ \cup T^+$. With that in mind, we begin by dividing $R^+$ into two overlapping regions, a finite extent rectangle and a ``tail'':
  \[ R_1^+ := \{ (q,p) \in R^+ : q < 2K \}, \qquad R_2^+ := \{ (q,p) \in R^+ : q > K \},\]
  where $K > 0$ is to be determined.  The top, bottom, and left boundary components we likewise write as $T_{1,2}^+$, $B_{1,2}^+$, and $L_{1,2}^+$.

  First consider the finite rectangle $R_1^+$.  Arguing as in \cite[Section~5]{walsh2009stratified}, it is possible to show that for each $K > 0$, there exists $\varepsilon_K > 0$ such that, if $\| w - \tilde{w} \|_{C^3(R^+)} + \abs{ F - \tilde F} < \varepsilon_K$, then $\tilde{w}_q < 0$ in $R_1^+ \cup T_1^+$.  This is simply because the finite rectangle behaves exactly the same as in the periodic case.  

  On the other hand, because $ w \in C_0^2(R)$, we can choose $K$ to be large enough so that 
  \[ \| w \|_{C^2(R_2^+)} < \frac{\delta}{2},\]
  where $\delta$ is given as in the hypotheses of Proposition~\ref{monotonicity small-amplitude proposition}.  
  By letting $\varepsilon := \min\{ \delta/2, \varepsilon_K\}$, we have $\tilde{w}_q < 0$ in $R_1^+ \cup T_1^+$, which in particular means that $\tilde{w}_q \leq 0$ on $L_2^+$.  Applying Proposition~\ref{monotonicity small-amplitude proposition} allows us to conclude that $\tilde{w}_q < 0$ in $R_2^+ \cup T_2^+$.  Since $R_1^+ \cup R_2^+ = R^+$, the proof is complete.   \end{proof}

\begin{lemma}[Closed property] \label{nodal closed lemma} Let $\{ (w_n, F_n) \} \subset U$ be a sequence of solutions to \eqref{w equation} and suppose that there exists a solution $(w, F) \in U$ of \eqref{w equation} with 
\begin{equation}
  (w_n, F_n) \to (w, F) \qquad \textup{in } C_{\mathrm{b}}^3(\overline{R}) \by \R.
\end{equation}
If each $w_n$ satisfies the nodal properties \eqref{elevation nodal properties}, then $w$ also satisfies \eqref{elevation nodal properties} unless $w \equiv 0$.  
\end{lemma}
\begin{proof}  
  Let $\{ (w_n, F_n)\}$ and $(w,F)$ be given as above.  Again, Lemma~\ref{nodal: monotone implies nodal lemma} shows that it is enough to confirm that $w_q < 0$ on $R^+ \cup T^+$.   Let $v := w_q$.  Simply by continuity, we know that $v \leq 0$ on $\overline{R^+}$, and we have already seen that $v$ satisfies the uniformly elliptic PDE \eqref{hq height equation} in $R^+$.  As $v$ vanishes identically on $L^+ \cup B^+$, the maximum principle implies that either (i) $v < 0$ in $R^+ \cup T^+$, or else (ii) there exists some point $(q_0, 0) \in T^+$ such that $v(q_0, 0) = 0$.  Here we are once more using the facts that $\sup_{R^+} v = 0$ and $v \to 0$ as $q \to \infty$.

  Thanks to Lemma~\ref{nodal: monotone implies nodal lemma},
  possibility (i) implies that $w$ satisfies the nodal properties, so consider possibility (ii).  The boundary condition on the top in \eqref{hq height equation} written in terms of $v$ is 
  \[ \frac{v v_q}{2h_p^2} - \frac{(1+v^2) v_{p}}{h_p^3} +  \frac{1}{F^2} \rho v = 0 \qquad \textrm{on } T^+.\]
  Were (ii) to hold, then evaluating the above line at $(q_0,0)$ would give that $v_{p}(q_0,0) = 0$.   This contradicts the Hopf lemma unless $v$ vanishes identically in $\overline{R^+}$, but this is equivalent to saying $w\equiv 0$.  
\end{proof}

\section{Small-amplitude existence theory} \label{small-amplitude section}

In this section we construct small-amplitude waves with nearly critical Froude numbers $F \approx \Fcr$. It will be convenient to introduce a small parameter $\epsilon$ defined by 
\begin{equation}
  \label{def epsilon}
  \epsilon := 
   \frac 1{\Fcr^2} - \frac{1}{F^2} 
  =  \mucr - \frac 1{F^2}.
\end{equation}
Note that with this convention, $\epsilon > 0$ corresponds to a supercritical Froude number.  Solving \eqref{def epsilon} for $F$, we also introduce the notation
\begin{align}
  \label{def F epsilon}
  F = F^\epsilon := \left( \frac 1{\Fcr^2} - \epsilon \right)^{-1/2}.
\end{align}

The main results of this section are summarized in the following theorem. In addition, we prove the existence of a family of subcritical periodic (cnoidal) waves, see Remark~\ref{periodic remark}. 
\begin{theorem}[Small-amplitude solitary waves]\label{small amplitude theorem}
  There exists $\epsilon_* > 0$ and a curve
  \begin{align*}
    \cm_\loc = \{ (w^\epsilon,F^\epsilon) : \epsilon \in (0,\epsilon_*)\}
    \sub U
  \end{align*}
  of solutions to $\F(w,F) = 0$ with the following properties:
  \begin{enumerate}[label=\rm(\roman*)]
  \item {\rm (Continuity)} The map $\epsilon \mapsto w^\epsilon$ is continuous from $(0,\epsilon_*)$ to $X$, with $\n {w^\epsilon}_X \to 0$ as $\epsilon \to 0$.
  \item {\rm (Invertibility)} The linearized operator $\F_w(w^\epsilon,F^\epsilon)$ is invertible for each $\epsilon \in (0,\epsilon_*)$.
  \item {\rm (Uniqueness)} If $w \in X$ satisfies $w > 0$ on $T$ and if $\n w_X$ is sufficiently small, then, for any $\epsilon \in (0,\epsilon_*)$, $\F(w,F^\epsilon) = 0$ implies $w = w^\epsilon$.
  \item \label{small amplitude theorem elevation} {\rm (Elevation)} The waves $(w^\epsilon,F^\epsilon)$ are waves of elevation in that $w^\epsilon > 0$ on $R \cup T$.
  \item {\rm (Analyticity)} The curve $\cm_\loc$ is real analytic in the sense that
    $\epsilon \mapsto w^\epsilon$ is real analytic.
  \end{enumerate}
\end{theorem}
 
We will prove Theorem~\ref{small amplitude theorem} incrementally.  First, in Lemma~\ref{small elevation lemma}, we use the center manifold reduction method to construct the family $(w^\epsilon, F^\epsilon)$ (cf.\ Theorem~\ref{center manifold theorem}).  
Doing this requires proving a number of preparatory lemmas that show the linearized problem has the requisite spectral behavior and the nonlinearity is quadratic near the origin. Continuity and invertibility then follow from a straightforward adaptation of the arguments in   \cite[Theorem~4.1]{wheeler2013solitary}.  Elevation and uniqueness are proved in Lemma~\ref{small elevation lemma} and Lemma~\ref{uniqueness lemma}, respectively.  These results are stitched together, along with an argument for the existence of an analytic reparameterization, in Section~\ref{small amplitude proof section}.

\subsection{Hamiltonian formulation}
Following \cite{groves2008vorticity}, we will relate our nonlinear operator equation to a Hamiltonian system $(\mathcal{M},\symp,\ham^\epsilon)$ in which the horizontal variable $q$ plays the role of time. The Hamiltonian function $\ham^\epsilon$ will turn out to be the flow force $\flowforce$ defined in Section~\ref{flowforce section} (cf.\ \cite{benjamin1966internal,benjamin1984impulse}).  For the position and momentum variables, we will use $w$ and $r := w_q/(w_p+H_p)$, respectively.  

With that in mind, in this section we think of $w = w(q,p)$ as a $C^1$ mapping $q \mapsto w(q,\cdot)$ taking values in a Hilbert space of $p$-dependent functions, and similarly for $r$. 
In particular, we will work with 
\begin{align*}
  \X &:= \{ (w,r) \in H^1(-1,0) \by L^2(-1,0) : w(-1) = 0 \}, \\
  \Y &:= \{ (w,r) \in H^2(-1,0) \by H^1(-1,0) : w(-1) = 0 \}.
\end{align*}
Here we are abusing notation somewhat by suppressing the dependence of $(w,r)$ on $q$; this will be a common practice throughout the section.  We also define  
\begin{align*}
  \man := \{ (w,r) \in \Y : w_p + H_p > 0 \text{ for } -1 \le p \le 0\}.
\end{align*}
Clearly $\man$ is an open subset of $\Y$ containing the origin.  Since $\Y$ is dense in $\X$ and the inclusion $\Y \hookrightarrow \X$ is smooth, $\man$ is therefore a so-called \emph{manifold domain} of $\X$. The symplectic form $\symp \maps \X \by \X \to \R$ is defined by 
\begin{align*}
  \symp\big( (w_1,r_1), (w_2,r_2) \big)
  :=
  \int_{-1}^0 (r_2 w_1 - r_1 w_2) \, dp,
\end{align*}
and finally the Hamiltonian $\ham^\epsilon \in C^\infty(\man,\R)$ is given by 
\begin{align}
  \notag
  \ham^\epsilon(w, r) 
  &:= \int_{-1}^0 \left[ 
  \int_0^p \frac{1}{(F^\epsilon)^2} \rho H_p \, dp^\prime 
  - \frac{1}{2H_p^2} 
  + \frac{1}{2} r^2 
  - \frac{1}{2(w_p+H_p)^2} 
  + \frac{1}{(F^\epsilon)^2} \rho w 
  \right] (w_p + H_p) \, dp \\
  \label{def normalized Hamiltonian}
  & \qquad 
  + \int_{-1}^0 \left[ 
  H_p \int_0^p  \frac{1}{(F^\epsilon)^2} \rho H_p \, dp^\prime 
  - \frac{1}{H_p} \right] \, dp.
\end{align}
The second integral in \eqref{def normalized Hamiltonian} does not depend on $w$ or $r$, and has been added to ensure $\ham^\epsilon(0,0) = 0$.  
Comparing with \eqref{DJ S}, we see that  $\ham^\epsilon$ is essentially $\flowforce$ viewed as a functional acting on $w$ and $r=w_q/(w_p+H_p)$.

A calculation shows that the domain $\D(\vf)$ of the Hamiltonian vector field $\vf$ corresponding to $(\man,\symp,\ham^\epsilon)$ is 
\begin{align}
  \label{Dvf}
  \begin{aligned}
    \D(\vf) = \bigg\{ (w, r ) \in \man &:
    r(-1) = 0,\ \\
    &\left(\frac{r^2}2 + \frac 1{2(w_p+H_p)^2}  - \frac 1{2H_p^2} 
    + \frac 1{(F^\epsilon)^2} \rho w\right)\bigg|_{p=0} = 0
    \bigg\},
  \end{aligned}
\end{align}
while Hamilton's equations are
\begin{equation}
  \label{hamiltons equation} \left\{ 
  \begin{aligned}
    w_q &= (w_p+H_p) r, \\
    r_q &= 
    \left( 
    \frac{r^2}2
    + \frac 1{2(w_p+H_p)^2}
    - \frac 1{2H_p^2}
    \right)_p
    + \frac 1{(F^\epsilon)^2} \rho_p w.
  \end{aligned} \right. 
\end{equation}
It is easy to see that \eqref{hamiltons equation} together with 
$(w,r) \in \D(\vf)$ and $r = w_q/(w_p+H_p)$
is formally equivalent to the height equation \eqref{w equation}. We note that these equations are \emph{reversible} in that, if $(w,r)(q)$ is a solution, then so is $\reverser(w,r)(-q)$, where 
\begin{align*}
  \reverser(w,r) := (w,-r) 
\end{align*}
is called the \emph{reverser}.

Linearizing the Hamiltonian system $(\man,\symp,\ham^0)$ with $\epsilon = 0$ about the equilibrium $(w,r) = (0,0)$ yields the linear problem
\begin{align*}
  (\dot w, \dot r)_q 
  = \linear (\dot w, \dot r)
\end{align*}
where $\linear \maps \D(\linear) \sub \X \to \X$ is the closed operator 
\begin{align*}
  \linear(\dot w, \dot r) := \left(
  H_p \dot r, \ 
  - \left(\frac{\dot w_p}{H_p^3} \right)_p + \frac 1{\Fcr^2} \rho_p \dot w 
  \right)
\end{align*}
with domain
\begin{align*}
  \D(\linear) := \left\{ (\dot w,\dot r) \in \Y :
  \dot r(-1) = 0,\ 
  \left(- \frac{\dot w_p}{H_p^3} + \frac 1{\Fcr^2} \rho \dot w\right)
  \bigg|_{p=0} = 0 \right\}.
\end{align*}

The operator $\linear$ is related to the Sturm--Liouville problem studied in Section~\ref{sec SL}. Using the results from that section we obtain the following.
\begin{lemma}[Spectral properties of $\linear$] \label{hamiltonian spectral lemma} 
  \hfill
  \begin{enumerate}[label=\rm(\roman*)]
  \item The spectrum of $\linear$ consists of an eigenvalue at $0$ with algebraic multiplicity 2, together with simple eigenvalues $\pm\sqrt{\nu_j}$, where $\nu_j$ are the nonzero eigenvalues of the corresponding Sturm--Liouville problem \eqref{SLprob2}. The eigenvector and generalized eigenvector associated with the $0$ eigenvalue are $(\Phicr,0)$ and $(0,\Phicr/H_p)$.
  \item \label{spectral bound} There exists $\Xi > 0$ and $C >0$ such that
    \begin{align*}
      \n u_\Y \le C \n{(\linear-i\xi I)u}_{\X},
      \qquad 
      \n u_\X \le \frac C{\abs \xi} \n{(\linear-i\xi I)u}_{\X},
    \end{align*}
    for all $u \in \D(\linear)$ and $\xi \in \R$ with $\abs \xi > \Xi$.
  \end{enumerate}
\end{lemma}

\begin{proof}
We begin with part (i).  Notice that $\lambda$ is an eigenvalue of $\linear$ provided that there exists a \ $(w, r) \in \D(\linear) \setminus \{0\}$ satisfying 
\[   H_p r = \lambda w, \qquad  -\left( \frac{{w}_p}{H_p^3} \right)_p + \frac{1}{\Fcr^2} \rho_p w  = \lambda r. \]
The boundary conditions encoded in the definition of $\D(\linear)$ along with the above equation imply that $w$ is a weak solution of the Sturm--Liouville problem \eqref{SLprob2} with $\nu = \lambda^2$.  By Lemma~\ref{spectrum lemma}, then, the eigenvalues of $\linear$ are precisely of the form $\pm \sqrt{\nu}$, for $\nu \in \Sigma$.  In particular, $0$ is an eigenvalue for $\linear$ with multiplicity $2$ while the rest of the spectrum consists of simple nonzero real eigenvalues.

We defer the proof of (ii) to Appendix~\ref{appendix: small amplitude}.
\end{proof}

\subsection{Further change of variables}

As in \cite{groves2008vorticity}, before applying a center manifold reduction to $(\man,\symp,\ham^\epsilon)$, we will perform a change of dependent variables (in a neighborhood of the origin in $\Y$) which flattens $\D(\vf)$. Set
\begin{align*}
  f(w,r) 
  &:=  
  \left(\frac{r^2}2 + \frac 1{2(w_p+H_p)^2}  - \frac 1{2H_p^2} 
  + \frac 1{(F^\epsilon)^2} \rho w\right)
  - 
  \left(- \frac{ w_p}{H_p^3} + \frac 1{(F^\epsilon)^2} \rho  w\right)\\
  &=  
  \frac{r^2}2 + \frac 1{2(w_p+H_p)^2}  - \frac 1{2H_p^2} 
  + \frac{ w_p}{H_p^3},
\end{align*}
so that the nonlinear boundary condition in the definition \eqref{Dvf} of $\D(\vf)$ can be written as 
\begin{align*}
  \left(- \frac{w_p}{H_p^3} + \frac 1{(F^\epsilon)^2} \rho w\right)
  \bigg|_{p=0} = f(w,r)\Big|_{p=0}.
\end{align*}
We replace $w$ with the new unknown $\xi$ defined by
\begin{align*}
  \xi := \zeta + \epsilon \rho(0) H_p^3(0) (1+p) \int_p^0 \zeta(s)\, ds,
\end{align*}
where
\begin{align*}
  \zeta := w + H_p^3(0) (1+p) \int_p^0 f(w,r)(s)\, ds.
\end{align*}
A calculation then shows that, at $p=0$,
\begin{align*}
  -\frac{\xi_p}{H_p^3}
  + \frac 1{\Fcr^2} \rho \xi
  = 
  -\frac{\zeta_p}{H_p^3}
  + \frac 1{(F^\epsilon)^2} \rho \zeta
  = 
  -\frac{w_p}{H_p^3} 
  + \frac 1{(F^\epsilon)^2} \rho w - f(w,r),
\end{align*}
and hence that $\xi$ satisfies the linearized boundary condition
\[  
-\frac{\xi_p(0)}{H_p^3(0)}
+ \frac 1{\Fcr^2} \rho(0) \xi(0) = 0 
\]
if and only if $(w,r) \in \D(\vf)$.

Denoting the change of variables mapping by $G^\epsilon(w,r) := (\xi,r)$, we have the following lemma.
\begin{lemma} \label{G lemma} Restricted to a sufficiently small neighborhood of the origin in $\R \by \Y$,
  \begin{enumerate}[label=\rm(\roman*)]
  \item $G^\epsilon$ is a diffeomorphism onto its image, with $G^\epsilon$ and $(G^\epsilon)^{-1}$ depending smoothly on $\epsilon$; 
  \item the derivative $DG^\epsilon(w,r) \maps \Y \to \Y$ extends to an isomorphism $\widehat{DG}^\epsilon(w,r) \maps \X \to \X$. This isomorphism and its inverse depend smoothly on $(w,r,\epsilon)$; and
  \item $G^\epsilon$ is a near-identity transformation in that $G^\epsilon(0)= 0$ for all $\epsilon$ and $DG^0(0)=\id$.
  \end{enumerate}
\end{lemma}
This follows from a standard argument that we omit; see, for example, \cite[Lemma~4.1]{groves2001spatial}. 
Applying the change of variables $u = G^\epsilon(w,r)$ transforms our Hamiltonian system $(\man,\symp,\ham^\epsilon)$ into 
$(\man,\symp_*,\ham_*^{\epsilon})$, where now $\symp_*$ is the (position-dependent) symplectic form
\begin{align*}
  \symp_*|_u(v_1,v_2) 
  := \symp\left( 
  \widehat{DG}^\epsilon\big((G^\epsilon)^{-1}(u)\big) v_1,
  \ 
  \widehat{DG}^\epsilon\big((G^\epsilon)^{-1}(u)\big) v_2
  \right)
\end{align*}
while
\begin{align*}
  \ham_*^\epsilon(u) := \ham\big( (G^\epsilon)^{-1}(u)\big).
\end{align*}
Hamilton's equations can be written abstractly as 
\begin{align}
  \label{u equation}
  u_q = \linear u + N^\epsilon(u).
\end{align}
Note that the reverser $\reverser$ is unchanged in these coordinates,
\begin{align*}
  G^\epsilon \circ \reverser \circ (G^\epsilon)^{-1} = \reverser.
\end{align*}

The following technical lemma details how the spatial dynamics formulation of the problem in \eqref{u equation} relates to the original operator equation $\F(w,F) = 0$.
\begin{lemma}\label{change of vars lem}\hfill
  \begin{enumerate}[label=\rm(\roman*)]
  \item \label{elliptic to dynamic} Let $(w,F^\epsilon)$ be a solution of $\F(w,F^\epsilon)=0$ with $\n w_X$ and $\abs \epsilon$ sufficiently small, and set 
    \begin{align}
      \label{u from w}
      u = G^\epsilon\left(w, \frac{w_q}{H_p+w_p}\right).
    \end{align}
    Then $u \in C^2_0(\R,\X) \cap C^1_0(\R,\U)$ solves Hamilton's equations $u_q = Lu + N^\epsilon(u)$ and is reversible in that $u(-q)=\reverser u(q)$.
  \item \label{dynamic to elliptic} Conversely, suppose that $u \in C^4_0(\R,\X) \cap C^3_0(\R,\U)$ satisfies $u_q = Lu + N^\epsilon(u)$ and $u(-q)=\reverser u(q)$. Then $u$ is related by \eqref{u from w} to some $w \in X$ solving $\F(w,F^\epsilon) = 0$. Moreover, the correspondence $(\epsilon,u) \mapsto w$ is continuous $\R \by C^4_0(\R,\X) \cap C^3_0(\R,\U) \to X$.
  \item \label{linear elliptic to dynamic} With $u$ and $w$ as above, suppose that $\dot w \in X$ is a nontrivial solution of the linearized problem $\F_w(w,F^\epsilon)\dot w = 0$. Then there is an associated nontrivial solution $\dot u \in C^2_\bdd(\R,\X) \cap C^1_\bdd(\R,\D(\linear))$ of the linearized problem $\dot u_q = L \dot u + DN^\epsilon(u) \dot u$ satisfying $\dot u(-q) = \reverser \dot u(q)$.
  \end{enumerate}
\end{lemma}
\begin{proof}
  Given Lemma~\ref{G lemma}, the proof of \ref{elliptic to dynamic} is straightforward and hence omitted. Likewise the proof of \ref{linear elliptic to dynamic} proceeds as in \cite[Lemma~4.4]{wheeler2013solitary}. The proof of \ref{dynamic to elliptic} is similar to that in \cite[Lemma~4.3]{wheeler2013solitary}, but is given in Appendix~\ref{appendix calculations} for the reader's convenience.
\end{proof}

\subsection{Center manifold reduction}
Up to this point, we have succeeded in transforming the original Hamiltonian system into one that is suitable for analysis via the center manifold reduction method.  In particular, by transitioning to $(\man,\symp_*,\ham_*^{\epsilon})$, we have obtained a reversible Hamiltonian system with linear boundary conditions.  Moreover, because $G^\epsilon$ is a near-identity mapping, the spectral properties of the linearized operator established in Lemma~\ref{hamiltonian spectral lemma} translate to the new system as well.  

Let $\X^\cs \sub \X$ be the two-dimensional center subspace associated with the eigenvalue $0$ of $\linear$. We denote by $P^\cs$ the associated spectral projection, and write $P^\hs := I-P^\cs$, $\X^\hs := P^\hs \X$. We will write $u^\cs \in P^\cs \D(\linear)$ as $u^\cs = z_1 e_1 + z_2 e_2$, where 
\begin{align} \label{def e1 and e2}
  e_1 := c_0^{-1/2} (\Phicr,0),
  \qquad 
  e_2 := c_0^{-1/2} (0,\Phicr/H_p),
\end{align}
are the eigenvector and generalized eigenvector from Lemma~\ref{hamiltonian spectral lemma}, and the normalization constant $c_0 > 0$ is given by 
\begin{align}
  \label{def c0}
  c_0 := \symp\big( (\Phicr,0), (0,\Phicr/H_p) \big)
  = \int_{-1}^0 \frac {\Phicr^2}{H_p} \, dp > 0.
\end{align}

\begin{lemma}[Center manifold reduction] \label{center manifold lemma}
  For any integer $k \ge 2$, there exists a neighborhood $\Lambda \by \U$ of the origin in $\R \by \D(\linear)$ such that, for each $\epsilon \in \Lambda$, there exists a two-dimensional manifold $\cman \sub \U$ together with an invertible coordinate map 
  \begin{align*}
    \chi^\epsilon := P^\cs |_\cman \maps \cman \to \U^\cs := P^\cs \U 
  \end{align*}
  with the following properties:
  \begin{enumerate}[label=\rm(\roman*)]
  \item 
    Defining $\Psi^\epsilon \maps \U^\cs \to \U^\hs := P^\hs\U$ by 
    \begin{align*}
      u^\cs + \Psi^\epsilon(u^\cs) := (\chi^\epsilon)^{-1}(u^\cs),
    \end{align*}
    the map $(\epsilon,u) \mapsto \Psi^\epsilon(u)$ is $C^k(\Lambda \by \U^\cs, \U^\hs)$. Moreover $\Psi^\epsilon(0) = 0$ for all $\epsilon \in \Lambda$ and $D\Psi^0(0) = 0$.
  \item Every initial condition $u_0 \in \cman$ determines a unique solution $u$ of $u_q = \linear u + N^\epsilon(u)$ which remains in $\cman$ as long as it remains in $\U$.
  \item If $u$ solves $u_q = \linear u + N^\epsilon(u)$ and lies in $\U$ for all $q$, then $u$ lies entirely in $\cman$.
  \item If $u^\cs \in C^1((a,b),\U^\cs)$ solves the reduced system
    \begin{align}
      \label{reduced ode}
      u^\cs_q = 
      f^\epsilon(u^\cs) :=
      L u^\cs + P^\cs N^\epsilon ( u^\cs + \Psi^\epsilon(u^\cs)),
    \end{align}
    then $u = (\chi^\epsilon)^{-1} (u^\cs)$ solves the full system $u_q = \linear u + N^\epsilon(u)$.
  \item 
    With $u^\cs$ and $u$ as above, if $\dot u^\cs \in C^1(\R,\U^\cs)$ solves the linearized reduced equation $\dot u^\cs_q = Df^\epsilon(u^\cs)\dot u^\cs$, then $\dot u = \dot u^\cs + D_u \Psi^\epsilon(u^\cs) \dot u^\cs$ solves the full linearized system $\dot u_q = \linear \dot u + D_u N^\epsilon(u)\dot u$.
  \item \label{best reduced}
    The reduced system \eqref{reduced ode} can be transformed via a $C^{k-1}$ change of variables into a Hamiltonian system $(U^\cs,\csymp,K^\epsilon)$, where  $U^\cs$ is a neighborhood of the origin in $\mathbb{R}^2$,  $\csymp$ is the canonical symplectic form 
    \begin{align*}
      \csymp((z_1,z_2),(z_1^\prime ,z_2^\prime)) := z_1 z_2^\prime - z_1^\prime z_2, \qquad (z_1, z_2), \, (z_1^\prime, z_2^\prime) \in \mathbb{R}^2,
    \end{align*}
    and the reduced Hamiltonian is
    \begin{align}
      \label{def cham}
      K^\epsilon(z_1, z_2) := \ham^\epsilon(z_1 e_1 + z_2 e_2 +\Theta^\epsilon(z_1 e_1 + z_2 e_2 )). 
    \end{align}
    Here $(\epsilon, u^\cs) \mapsto \Theta^\epsilon(u^\cs)$ is of class $C^{k-1}(\Lambda \times \U^\cs, \U)$ and satisfies $\Theta^\epsilon(0) = 0$ for all $\epsilon \in \Lambda$, and $D_{u^\cs} \Theta^0(0) = 0$. 
    The system is reversible with reverser $\creverser(z_1,z_2) = (z_1,-z_2)$.
  \end{enumerate}
\end{lemma}
\begin{remark}\label{reduced remark}
  For a solution $(z_1,z_2)$ to the reduced Hamiltonian system from part (vi), the corresponding solution $(w,r)$ of the original Hamiltonian system is given by $(w,r) = z_1 e_1 + z_2 e_2 + \Theta^\epsilon(z_1 e_1 + z_2 e_2)$. 
\end{remark}
\begin{proof}   
  Observe that $L$ satisfies (H1) and (H2) of Theorem~\ref{center manifold theorem} in light of Lemma~\ref{hamiltonian spectral lemma}, 
  and the only part of its spectrum lying on the imaginary axis is the eigenvalue $0$, which has algebraic multiplicity $2$.
  The nonlinearity $N^\epsilon$ is, in fact, $C^\infty$ in its dependence on $\epsilon$ and $u$ in any small neighborhood of the origin.  Inspecting \eqref{hamiltons equation}, and keeping in mind that $G^\epsilon$ is a near-identity transformation, it is easy to confirm that $N^0(0) = 0$ and $D_u N^0(0) = 0$.    We may therefore apply Theorem~\ref{center manifold theorem} to the system $(\man,\sympep_*,\ham_*^{\epsilon})$, obtaining a two-dimensional center manifold $\cman$ satisfying (i)--(iv).  Part (v) requires a small additional argument that is discussed, for example, in \cite[Lemma~4.4]{wheeler2013solitary}.

  Now consider statement (vi). We begin by undoing the transformation $G^\epsilon$, since it is considerably simpler to work in the original variables.  Define 
  $\Upsilon^\epsilon \in C^k(\Lambda \by \U^\cs, \U)$ by the relation 
  \[ 
  u^\cs + \Upsilon^\epsilon( u^\cs) := (G^\epsilon)^{-1}\left( u^\cs + \Psi^\epsilon(u^\cs) \right).
  \]
  Then $\Upsilon^\epsilon(0) = 0$ for all $\epsilon \in \Lambda$, and $D_{u^\cs} \Upsilon^0(0) = 0$.  Here we are simply relying on the fact that $G^\epsilon$ is near-identity.  

  From Theorem~\ref{center manifold theorem}(v), we know that the center manifold $(\cman,{\sympep_*}|_{\cman}, \ham_*^\epsilon|_{\cman})$ is a symplectic submanifold of $(\man,\sympep_*,\ham_*^{\epsilon})$. Changing variables using $G^\epsilon$ we obtain a new manifold $\cmang$
  that we equip with the chart $(\operatorname{id}+\Upsilon^\epsilon)^{-1} \maps \cmang \to \U^\cs$. Working in these coordinates, the symplectic form $\sympep_*|_{\cman}$ becomes $\csympg$, which, for $u^\cs \in \U^\cs$ and $v_1^\cs, v_2^\cs \in \X^\cs$, is given by
    \begin{align}
    \label{csympg}
    \begin{aligned}
      \csympg|_{u^\cs}(v_1^\cs, v_2^\cs) &= 
      \sympep_*|_{u^\cs + \Psi^\epsilon(u^\cs)}(v_1^\cs + D_{u^\cs}\Psi^\epsilon(u^\cs)v_1^\cs, v_2^\cs + D_{u^\cs}\Psi^\epsilon(u^\cs)v_2^\cs) \\
    & = \omega(v_1^\cs, v_2^\cs) + \mathcal{O}(|\epsilon| \| u^\cs \|_{\mathcal{X}}  \| v_1^\cs \|_{\mathcal{X}}  \| v_2^\cs \|_{\mathcal{X}} ).
        \end{aligned}
  \end{align}
  Since $\chi^\epsilon$ and the reverser $\reverser$ are linear maps which commute, it is easy to check that, in these coordinates, it is given simply by $S|_{\U^\cs}$.

  Thanks to \eqref{csympg}, we can now employ a parameter-dependent Darboux transformation 
  \begin{equation}
    u^\cs \mapsto u^\cs + \Xi^\epsilon(u^\cs) 
  \end{equation}
  which is $C^{k-1}$ and transforms $\csympg$ into $\omega$ in a neighborhood of the origin (see \cite[Theorem~4]{buffoni1999multiplicity}). 
  This map is near-identity in that $\Xi^\epsilon(0) = 0$ and $D_{u^\cs} \Xi^0(0) = 0$. We now equip $\cmang$ with the new chart $(\operatorname{id}+\Theta^\epsilon)^{-1} \maps \cmang \to \U^\cs$ where $\Theta^\epsilon$ is defined by 
  \begin{equation}
    \id +
    \Theta^\epsilon :=  
    (\operatorname{id}+\Upsilon^\epsilon) \circ 
    (\id + \Xi^\epsilon)^{-1},
    \label{def darboux theta} 
  \end{equation}
  and $\id + \Theta^\epsilon$ is also near-identity. In these coordinates, our Hamiltonian system becomes $(\U^\cs, \omega, \cham)$, where the Hamiltonian $\cham$ is
  \begin{align*}
    \cham(u^\cs) = \ham^\epsilon(u^\cs + \Theta^\epsilon(u^\cs)).
  \end{align*}
  The Darboux transformation $\id + \Xi^\epsilon$ can be chosen so that the action of the reverser in these coordinates is still given by $S|_{\U^\cs}$; see the arguments leading to \cite[Theorem~5.17]{mielke1991book}.

  We identify $\U^\cs$ with a neighborhood of the origin $U^\cs \subset \mathbb{R}^2$ via the mapping $U^\cs\ni (z_1, z_2)\mapsto z_1 e_1 + z_2 e_2 \in  \U^\cs$.  Notice that 
  \begin{align*} 
    \omega(z_1 e_1 + z_2 e_2, z_1^\prime e_1 + z_2^\prime e_2) &= z_1 z_2^\prime \omega(e_1, e_2) + z_1^\prime z_2 \omega(e_2, e_1) \\
    & = \csymp((z_1, z_2), (z_1^\prime, z_2^\prime)),
  \end{align*}
  for all $(z_1, z_2), (z_1^\prime, z_2^\prime) \in \mathbb{R}^2$, 
  and that 
  \begin{align*}
    S(z_1 e_1 + z_2 e_2) = z_1 e_2 - z_2 e_2.
  \end{align*}
  Thus, we obtain the reversible and canonical Hamiltonian system
  $(U^\cs,\gamma,K^\epsilon)$, proving (vi).
\end{proof}

Thanks to the clever choice of coordinates outlined in \cite{groves2008vorticity} and leading to the system described in Lemma~\ref{center manifold lemma}(vi), we can Taylor expand the reduced system \eqref{reduced ode} by using \eqref{def cham} directly, and avoid dealing with the implicitly defined intermediate Hamiltonian $\ham_*^\epsilon$ entirely. The result of this calculation, which is presented in Appendix~\ref{appendix: reduced system calc}, is that the reduced system \eqref{reduced ode} is equivalent to the following ODE set in $\mathbb{R}^2$: 
\begin{equation*}
  \left\{ 
  \begin{aligned}
    {z}_{1q}  &= z_2 + \mathcal{R}_1(z_1, z_2, \epsilon) \\
    {z}_{2q}  &= \epsilon c_0^{-1} c_1 z_1 - \frac{3}{2} c_0^{-3/2}  c_2 z_1^2 + \mathcal{R}_2(z_1, z_2, \epsilon), 
  \end{aligned} 
  \right.
\end{equation*}
where $c_0$ was defined in \eqref{def c0},
\begin{align}
  \label{def c1 c2}
  c_1 := \rho(0) \Phicr(0)^2 - \int_{-1}^0 \rho_p \Phicr^2 \, dp,
  \qquad
  c_2 := \int_{-1}^0 \frac{(\partial_p \Phicr)^3}{H_p^4} \, dp,
\end{align}
and
\begin{align*} 
  \mathcal{R}_1(z_1,z_2,\epsilon) &= \mathcal{O}(|(z_1, z_2)|^2 + |z_2| |(\epsilon, z_1, z_2)|^2 + |\epsilon|  \lvert(z_1, z_2)\rvert ), \\
  \mathcal{R}_2(z_1,z_2,\epsilon) &= \mathcal{O}(|z_1||(\epsilon, z_2)|^2 + |z_2| |(\epsilon, z_1)|), 
\end{align*}
are higher order remainder terms.  
The reversal symmetry implies that $\mathcal{R}_1$ is odd in $z_2$ and $\mathcal{R}_2$ is even in $z_2$.   We can simplify things even further by introducing the scaled variables $(Z_1, Z_2)$ and $Q$ defined as
\begin{gather} 
  \label{ode scaling}
    z_1 =: |\epsilon| c_0^{1/2} c_1 c_2^{-1} Z_1 ,  
    \qquad
    z_2 =: |\epsilon|^{3/2} c_1^{3/2} c_2^{-1}  Z_2, 
    \qquad
    q =: |\epsilon|^{-1/2} c_0^{1/2} c_1^{-1/2} Q.
\end{gather}
This transforms the system into 
\begin{equation}
  \label{rescaled ode}
  \left\{ 
  \begin{aligned}
    Z_{1Q}  
    &= Z_2 + \mathcal{R}_3(Z_1, Z_2, \epsilon) \\ 
    Z_{2Q}  
    &=  Z_1 - 
    \frac{3}{2} (\signum \epsilon) Z_1^2 + \mathcal{R}_4(Z_1, Z_2, \epsilon), 
  \end{aligned} 
  \right.
\end{equation}
where $\mathcal{R}_3$ and $\mathcal{R}_4$ are new remainder terms.  From the estimates of $\mathcal{R}_1, \mathcal{R}_2$ and the change of variable formulas above, it is clear that 
\begin{align*}
  \mathcal R_3, 
  ~ 
  \mathcal R_4,
  ~ 
  D_{(Z_1,Z_2)}\mathcal R_3,
  ~ 
  D_{(Z_1,Z_2)}\mathcal R_4 
  = 
  \mathcal{O}(|\epsilon|^{1/2}).
\end{align*}

These calculations lead directly to the following result.

\begin{lemma}[Existence of $w^\epsilon$] \label{small existence lemma}  
  There exists $\epsilon_* > 0$ such that, for each $\epsilon \in (0, \epsilon_*)$, there is a solution  $(w^\epsilon,F^\epsilon) \in X \times \R$ to the height equation \eqref{w equation}.  
  Moreover, the map $\epsilon \mapsto w^\epsilon$ is continuous $(0,\epsilon_*) \to X$, and $\n{w^\epsilon}_X \to 0$ as $\epsilon \to 0$.
\end{lemma}
\begin{proof}  
  \begin{figure}
    \includegraphics[scale=1.1]{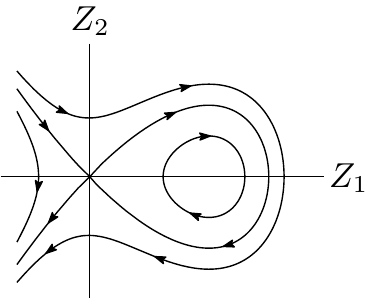} 
    \caption{Phase portrait for the reduced system \eqref{rescaled ode} when $\epsilon=0$.}
    \label{phase figure}
  \end{figure}
  Observe that, when $\epsilon = 0$, the reduced and rescaled system \eqref{rescaled ode} becomes 
  \begin{align}
    \label{rescaled z}
    -Z_1 + \frac{3}{2} Z_1^2 + Z_{1 QQ} = 0,
  \end{align}
  whose phase portrait is shown in Figure~\ref{phase figure}. This is exactly the equation satisfied by the Korteweg--de Vries soliton $Z_1 = \sech^2(Q/2)$ (with unit wave speed).  
  We conclude, therefore, that at $\epsilon = 0$, \eqref{rescaled ode} has an orbit $(Z_1^0,Z_2^0)$ homoclinic to $0$.  The reversibility of the system guarantees that there is a nearby reversible homoclinic orbit $(Z_1^\epsilon,Z_2^\epsilon)$
  for $\epsilon > 0$ sufficiently small
  (see, e.g., \cite[Proposition~5.1]{kirchgassner1988resonant} and the surrounding comments). 
  By choosing the $k$ in Lemma~\ref{center manifold lemma} appropriately, we can ensure that the map $\epsilon \mapsto (Z_1^\epsilon,Z_2^\epsilon)$ is continuous 
  with values in $C^4_0(\R,\R^2)$.

  By Lemma~\ref{center manifold lemma}, for each $\epsilon > 0$ sufficiently small, there is a corresponding solitary wave solution \[ (w^\epsilon, r^\epsilon, F^\epsilon) \in \left(C_0^4(\mathbb{R}, \mathcal{X}) \cap C_0^3(\mathbb{R}, \U) \right) \times \R\] of the full system \eqref{hamiltons equation}.
  By Lemma~\ref{change of vars lem}(ii), there exists a corresponding solution  $(w^\epsilon, F^\epsilon) \in X \times \mathbb{R}$ of the original height equation \eqref{w equation}.  Thus, there is an $\epsilon_* > 0$ and a local curve of solutions $\cm_\loc = \{(w^\epsilon, F^\epsilon) : \epsilon \in (0, \epsilon_*)\}$.  
  Finally, undoing the scaling and the various changes of variable, we see that $\epsilon \mapsto w^\epsilon$ is continuous $(0,\epsilon_*) \to X$ with $\n{w^\epsilon}_X \to 0$ as $\epsilon \to 0$.
\end{proof}

\begin{remark}\label{periodic remark}
  The waves constructed in Lemma~\ref{small existence lemma} are solitary waves. As in \cite{groves2008vorticity}, for $\epsilon < 0$ with $\abs \epsilon$ sufficiently small, a similar argument guarantees the existence of a family of symmetric periodic waves of cnoidal type with period $\mathcal O(\abs\epsilon^{1/2})$.
\end{remark}

\subsection{Uniqueness and elevation}

In this subsection we show that the waves $(w^\epsilon,F^\epsilon)$ constructed in Lemma~\ref{small existence lemma} are waves of elevation.  Moreover, they are the unique waves of elevation with $w$ and $\epsilon$ sufficiently small. 

\begin{lemma}[Elevation] \label{small elevation lemma}
  After possibly shrinking the range of $\epsilon$, all of the waves $w^\epsilon$ constructed above are waves of elevation in that $w^\epsilon > 0$ on $R \cup T$.
  \begin{proof}
    For each $\epsilon \geq 0$, consider the rescaled solution $(Z_1^\epsilon,Z_2^\epsilon)$ of \eqref{rescaled ode}. Since $Z_1^0(Q) > 0$ for all $Q$, the continuous dependence of $Z_1^\epsilon$ as well as the stable and unstable manifolds of \eqref{rescaled ode} at the origin on $\epsilon$ imply that $Z_1^\epsilon(Q) > 0$ for all $Q$ when $\epsilon$ is sufficiently small. Since 
    \begin{align}
      \label{ratio limit}
      \lim_{Q \to \pm\infty} \frac{Z_2^\epsilon(Q)}{Z_1^\epsilon(Q)}
      =
      \pm 1 + \mathcal{O}(\epsilon^{1/2}),
    \end{align}
    (see Figure~\ref{phase figure}) this continuous dependence also means that, shrinking $\epsilon$ further,
    \begin{align*}
      \abs{Z_2^\epsilon(Q)}
      \le C \abs{Z_1^\epsilon(Q)} \qquad \textrm{for all } Q,
    \end{align*}
    where $C$ is independent of $\epsilon$. Undoing the scaling in \eqref{ode scaling}, we find that $z_1^\epsilon(q) > 0$ and
    \begin{align}
      \label{ratio bounds}
      \abs{z_2^\epsilon(q)}
      \le C \epsilon^{1/2} \abs{z_1^\epsilon(q)}.
    \end{align}

    Going back through the many changes of variables (see Remark~\ref{reduced remark}), we see that we can write \begin{align}
      \label{unwinding changes}
      w^\epsilon(q,p) &= c_0^{-1/2} z_1^\epsilon(q)\Phicr(p) + \mathcal R(q,p),
    \end{align}
    where the remainder term satisfies
    \begin{align*}
      \n{\mathcal R(q,\placeholder)}_{H^2}
      \le C \big(\epsilon + \abs{z_1^\epsilon(q)} + \abs{z_2^\epsilon(q)}\big)\big(\abs{z_1^\epsilon(q)}+\abs{z_2^\epsilon(q)}\big).
    \end{align*}
    In particular, taking $\epsilon$ small enough, \eqref{ratio bounds} ensures that
    \begin{align}
      \label{better R bound}
      \n{\mathcal R_p(q,\placeholder)}_{L^\infty}
      \le \frac {c_0^{-1/2}}2 \min[ (\Phicr)_p]  z_1^\epsilon(q),
    \end{align}
    where we recall from Lemma~\ref{lem_minimality}\ref{lem_minimality_pos} that $(\Phicr)_p > 0$. Differentiating \eqref{unwinding changes} with respect to $p$ then yields 
    \begin{align*}
      w^\epsilon_p(q,p) \ge
      c_0^{-1/2} 
      z_1^\epsilon(q)(\Phicr)_p(p)
      - \frac {c_0^{-1/2}}2 \min[ (\Phicr)_p]  z_1^\epsilon(q) > 0
    \end{align*}
    for all $(q,p) \in \overline R$. Integrating with respect to $p$ using the fact that $w=0$ on $B=\{p=-1\}$, we conclude that $w^\epsilon > 0$ on $R \cup T$ as desired.
  \end{proof}
\end{lemma}

\begin{lemma}[Uniqueness] \label{uniqueness lemma}
  Suppose that $(w,F^\epsilon) \in X \times \R$ is a solution of \eqref{w equation} with $F^\epsilon$ defined as in \eqref{def F epsilon}. If $\epsilon > 0$ and $\n w_{C^2(\Omega)}$ are sufficiently small and $w > 0$ on $T$, then $w=w^\epsilon$.
  \begin{proof}
    Suppose that we have a solution $(w, F^\epsilon)$ of \eqref{w equation} with $\epsilon + \n w_{C^2(\Omega)} < \delta$, where $\delta > 0$ is to be determined, and assume that $w \not\equiv w^\epsilon$. We will show that $w$ is not a wave of elevation.

    By the properties of the center manifold, $w$ is determined by a homoclinic orbit $(z_1,z_2)$ of \eqref{reduced ode}. Since this equation already has a homoclinic orbit, namely $(z_1^\epsilon,z_2^\epsilon)$, and $w$ cannot be a translate of $w^\epsilon$, $(z_1,z_2)$ is not a translate of $(z_1^\epsilon,z_2^\epsilon)$. Looking at the phase portrait at the origin (see Figure~\ref{phase figure}), we conclude that $z_1 < 0$ for $\abs q$ sufficiently large, and also that (compare with the signs in \eqref{ratio limit})
    \begin{align*}
      \lim_{q \to \pm\infty} \frac{z_2(q)}{z_1(q)}
      =
      \mp \epsilon^{1/2} + \mathcal{O}(\epsilon).
    \end{align*}
    Thus, as in the proof of Lemma~\ref{small elevation lemma} above, we can shrink $\delta$ so that
    \begin{align}
      \label{unwinding changes again}
      w(q,p) &= c_0^{-1/2} z_1(q)\Phicr(p) + \mathcal R(q,p).
    \end{align}
    where the remainder term has the bound
    \begin{align}
      \label{better R bound again}
      \n{\mathcal R_p(q,\placeholder)}_{L^\infty}
      \le \frac {c_0^{-1/2}}2 \min[ (\Phicr)_p]  \abs{z_1(q)}.
    \end{align}
    Differentiating \eqref{unwinding changes again} with respect to $p$ and plugging in \eqref{better R bound again}, we conclude 
    \begin{align*}
      w_p(q,p) \le
      c_0^{-1/2} z_1(q)(\Phicr)_p(p)
      + \frac {c_0^{-1/2}}2 \min[ (\Phicr)_p]  \abs{z_1(q)} < 0
    \end{align*}
    as soon as $q$ is large enough that $z_1(q) < 0$. Integrating with respect to $p$, we find, for instance, that $w(q,0) < 0$ for such $q$, and hence in particular that $w$ cannot be a wave of elevation.
  \end{proof}
\end{lemma}

\subsection{Proof of small-amplitude existence} \label{small amplitude proof section} 

We now have all of the major components necessary to prove our main theorem.  

\begin{proof}[Proof of Theorem~\ref{small amplitude theorem}]
We already established the existence of the one-parameter family $(w^\epsilon, F^\epsilon)$ as well as (i) in Lemma~\ref{small existence lemma}. The invertibility of $\F_w(w^\epsilon, F^\epsilon)$ follows immediately from Lemma~\ref{strong invertibility lemma} when we note that $\epsilon > 0$ implies $F^\epsilon > \Fcr$.  Likewise, parts (iii) and (iv) of the theorem statement were already proved in Lemmas \ref{small elevation lemma} and \ref{uniqueness lemma}.    

All that remains, therefore, is part (v).  But, for any $\epsilon \in (0,\epsilon_*)$, $\F_w(w^\epsilon, F^\epsilon)$ is an isomorphism by the above reasoning, and moreover $\F$ is real analytic.  Applying the real analytic implicit function theorem, we find that $w^\epsilon$ depends analytically on $\epsilon \in (0,\epsilon_*)$, which completes the theorem.
\end{proof}

\section{Large-amplitude existence theory} \label{global bifurcation section}
The final step in the proof of Theorem \ref{main existence theorem} is to show that the local solution curve $\cm_\loc$ can be continued to obtain a curve of large-amplitude solution $\cm$.  This will be accomplished using an argument based on analytic global bifurcation theory.  Unfortunately, the existing literature does not immediately apply to our problem for two reasons:  (i) the point of bifurcation $(0, \Fcr)$ is singular in the sense that $\F_w(0,\Fcr)$ fails to be a Fredholm operator with index $0$; and (ii) it is far from obvious that $\F^{-1}(0)$ is locally compact.  

With that in mind, in the next subsection we provide an abstract global-bifurcation-theoretic result that does hold in this more general context.  
The price we pay for this is the appearance of a new (and undesirable) possibility for the behavior of the global solution curve (see Figure~\ref{global figure} for an illustration).  
For a large class of elliptic systems set on infinite cylinders, we show that this undesirable alternative occurs precisely when there exists a sequence of translated solutions along the curve that converges to a front.   At last, in Section \ref{proof main result section}, these results are applied to our problem and combined with the qualitative theory developed in Section \ref{qualitative section} to complete the proof of Theorem~\ref{main existence theorem}.    

\subsection{Global continuation} \label{abstract global section} 
  
\begin{figure}
  \includegraphics[scale=1.1]{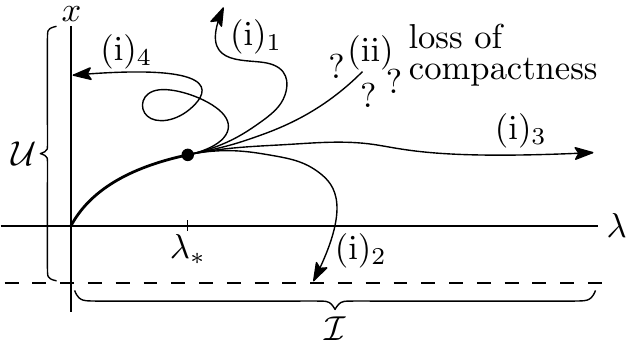}  
  \caption{Several ways for the alternatives \ref{gen blowup alternative} and \ref{gen noncompact} in Theorem~\ref{generic global theorem} to hold. In \ref{gen blowup alternative}$_1$, the first term on the right-hand side of \ref{gen global blowup} tends to infinity as $s \to \infty$. In \ref{gen blowup alternative}$_2$, it is the second term on the right-hand side of \ref{gen global blowup}, and so on. Alternative \ref{gen noncompact} holds if the curve loses compactness even while $N(s)$ remains bounded.}
  \label{global figure}
\end{figure}

Let $\genX, \genY$ be Banach spaces, $\genI$ an open interval (possibly unbounded) with $0 \in \overline{\genI}$, and $\genU \sub \genX$ an open set with $0 \in \partial \genU$. Consider the abstract operator equation
\[ \genG(x,\lambda) = 0, \]
where $\genG \maps \genU \by \genI \to \genY$ is an analytic mapping. Assume that for any $(x,\lambda) \in \genU \times \genI$ with $\genG(x,\lambda) = 0$, the Fr\'echet derivative $\genG_x(x,\lambda) \maps \genX \to \genY$ is Fredholm with index $0$.

\begin{theorem}[Global continuation] \label{generic global theorem}
  Suppose that there exists a continuous curve $\cm_\loc$ of solutions to $\genG(x,\lambda) = 0$, parametrized as
  \begin{align*}
    \cm_\loc :=  \{(\tilde x(\lambda),\lambda) : 0 < \lambda < \lambda_*\}
    \sub \genG^{-1}(0)
  \end{align*}
  for some $\lambda_* > 0$ and continuous $\tilde x \maps (0,\lambda_*) \to \genU$.  If
  \begin{align}
    \label{gen limit inv}
    \lim_{\lambda \searrow 0} \tilde x(\lambda) = 0 \in \dell \genU,
    \qquad 
    \genG_x(\tilde x(\lambda),\lambda) \maps \genX \to \genY
    \textup{ is invertible for all $\lambda$},
  \end{align}
  then $\cm_\loc$ is contained in a curve of solutions $\cm$, parametrized as
  \begin{align*}
    \cm := \{(x(s),\lambda(s)) : 0 < s < \infty\} \sub \genG^{-1}(0)
  \end{align*}
  for some continuous $(0,\infty) \ni s \mapsto (x(s),\lambda(s)) \in \genU \by \genI$, with the following properties.
  \begin{enumerate}[label=\rm(\alph*)]
  \item \label{gen alternatives} One of the following alternatives holds:
    \begin{enumerate}[label=\rm(\roman*)]
    \item \textup{(Blowup)} \label{gen blowup alternative}
      As $s \to \infty$,
      \begin{align}
        \label{gen global blowup}
        N(s):= \n{x(s)}_\genX + \frac 1{\dist(x(s),\dell \genU)} +
        \lambda(s) + \frac 1{\dist(\lambda(s),\dell \genI)} \to \infty.
      \end{align}
    \item \label{gen noncompact} \textup{(Loss of compactness)} There exists a sequence $s_n \to \infty$ such that $\sup_n N(s_n) < \infty$ but 
      $\{x(s_n)\}$ has no subsequences converging in $\genX$.
    \end{enumerate}
  \item \label{gen reparam} Near each point $(x(s_0),\lambda(s_0)) \in \cm$, we can reparametrize $\cm$ so that $s\mapsto (x(s),\lambda(s))$ is real analytic.
  \item \label{gen reconnect} $(x(s),\lambda(s)) \not\in \cm_\loc$ for $s$ sufficiently large.
  \end{enumerate}
  \begin{proof}
    \begin{figure}
      \includegraphics[scale=1.1]{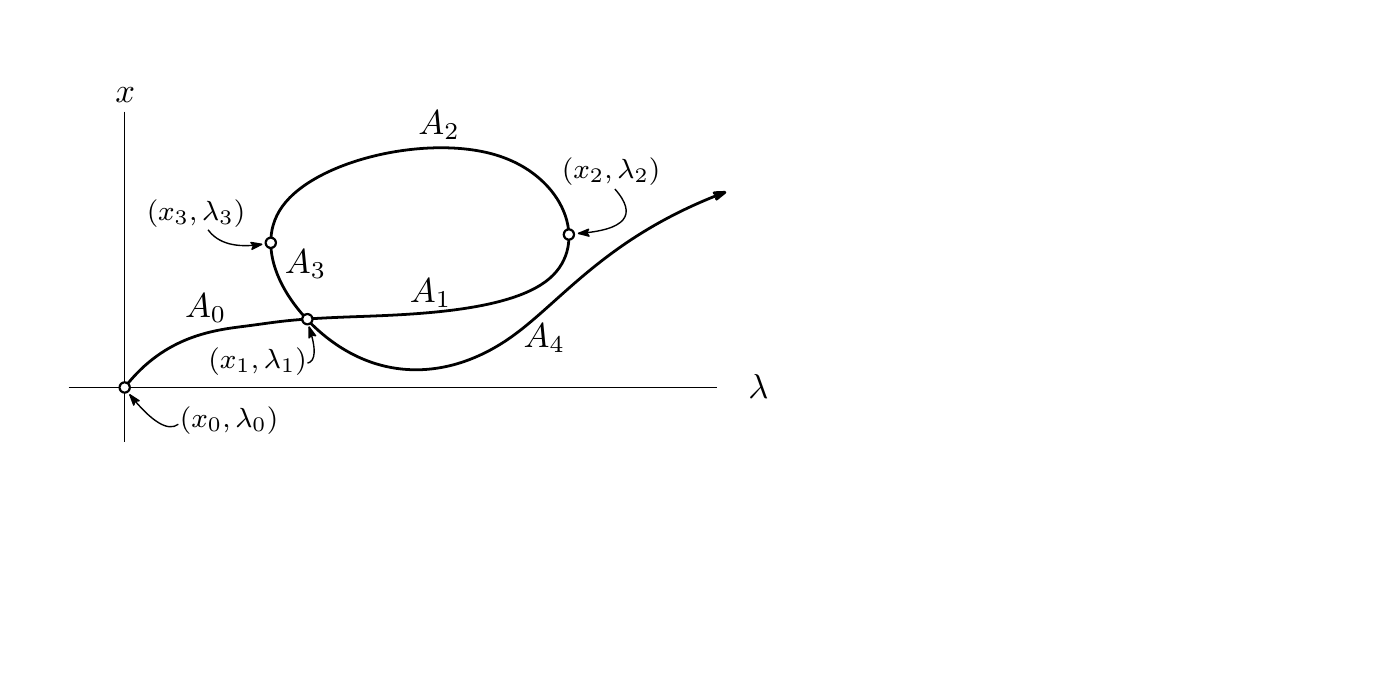}  
      \caption{The distinguished arcs in the proof of Theorem~\ref{generic global theorem}.}
      \label{arcs figure}
    \end{figure}

    Set $(x_0,\lambda_0) := (0,0) \in \dell (\genU \by \genI)$. Since we are not guaranteed that $\genG(x_0,\lambda_0)$ exists or that $\genG_x(x_0,\lambda_0)$ is Fredholm, we cannot start our continuation argument there. Instead, we will start our argument at $(x_{1/2},\lambda_{1/2}):=(\tilde x(\lambda_*/2),\lambda_*/2)$.

    Since $\genG$ is real analytic and $\genG_x(x_{1/2},\lambda_{1/2})$ is invertible by \eqref{gen limit inv}, $(x_{1/2},\lambda_{1/2})$ lies in some connected component $A_0$ of
    \begin{align*}
      \mathcal A := \Big\{ (x,\lambda) \in \genU \times \genI  : \genG(x,\lambda) = 0,
      \ \text{$\genG_x(x,\lambda)$ is invertible} \Big\}.
    \end{align*}
    Following \cite{buffoni2003analytic}, we will call such connected components \emph{distinguished arcs}. The analytic implicit function theorem guarantees that, like all distinguished arcs, $A_0$ is a graph:
    \begin{align*}
      A_0 = \{ (\til x_0(\lambda),\lambda) : \lambda  \in I_0 \},
    \end{align*}
    where $I_0 \subset \genI$ is a (possibly unbounded) open interval and $\til x_0 \maps I_0 \to \genX$ is analytic. For notational convenience, we will also reparametrize $A_0$ as 
    \begin{align*}
      A_0 = \{ (x(s),\lambda(s)) : 0 < s < 1 \}
    \end{align*}
    where $\lambda(s)$ is increasing.

    An easy argument using the implicit function theorem shows that $\cm_\loc$ lies entirely in $A_0$. Thus $I_0 = (0,\lambda_1)$ for some $\lambda_1 \ge s_*$, and $A_0$ has ``starting endpoint''
    \begin{align*}
      \lim_{s \searrow 0} (x(s),\lambda(s)) = 
      (x_0,\lambda_0) = (0,0)
    \end{align*}
    by \eqref{gen limit inv}. Next we turn our attention to the limit $s \nearrow 1$. We will show that either $\cm := A_0$ satisfies \ref{gen blowup alternative} or \ref{gen noncompact} (after reparametrization) or that $A_0$ connects to another distinguished arc $A_1$. If $N(s) \to \infty$ as $s \nearrow 1$, then, after a trivial reparametrization mapping $(0,1)$ to $(0,\infty)$, \ref{gen blowup alternative} occurs and we are done. So assume that \eqref{gen global blowup} does not hold as $s \nearrow 1$, in which case we can find a sequence $\{ s_n \} \subset (1/2,1)$ with $s_n \nearrow 1$ so that $N(s_n) \le M < \infty$ for all $n$. Without loss of generality we can assume that $\lambda(s_n) \to \lambda_1 \in \genI$. If $\{ x(s_n) \}$ has no convergent subsequences in $\genX$, then \ref{gen noncompact} occurs, and again we are done after a trivial reparametrization. Assume instead, passing to a subsequence, that we have $x(s_n) \to x_1 \in \genU$.

    By continuity, $\genG(x_1,\lambda_1) = 0$, and hence $\genG_x(x_1,\lambda_1)$ is Fredholm with index $0$.  Let $1 \leq k < \infty$ be the dimension of its null space.  Using the real-analytic Lyapunov--Schmidt procedure as in \cite[Theorem~8.2.1 and Remark~8.2.2]{buffoni2003analytic}, the problem of solving $\genG(x, \lambda) = 0$ near $(x_1, \lambda_1)$ can be reduced to a finite-dimensional system.  That is, there exists a real-analytic bifurcation function $\mathcal B \maps \R^k \by \R \to \R^k$, and a real-analytic bijection mapping the zero-set of $\mathcal{B}$ in neighborhood of the origin to the zero-set of $\genG$ near $(x_1, \lambda_1)$.   
    Complexifying $\mathcal{B}$ in the natural way, we see that its zero-set represents a complex-analytic variety, which implies a great deal about its structure; see \cite[Theorem 7.4.7]{buffoni2003analytic}. Arguing as in Steps~1 and 2 of the proof of \cite[Theorem 9.1.1]{buffoni2003analytic}, we find that $(x(s),\lambda(s)) \to (x_1,\lambda_1)$ as $s \nearrow 1$, and that, at $(x_1,\lambda_1)$, the distinguished arc $A_0$ is uniquely paired to another (different) distinguished arc $A_1$, parametrized as    
    \begin{align*}
      A_1 = \{ (x(s),\lambda(s)) : 1 < s < 2 \}
    \end{align*}
    and with $(x(s),\lambda(s)) \to (x_1,\lambda_1)$ as $s \searrow 1$. Moreover, we can choose $s \mapsto (x(s),\lambda(s))$ to be analytic in a neighborhood of $s=1$, 
    perhaps at the cost of giving up analyticity for other values of $s$.
    Note that, by the implicit function theorem, the closures $\overline {A_0}$ and $\overline{A_1}$ can only meet at $(x_1,\lambda_1)$ and perhaps also at the other endpoint $(x_0,\lambda_0)$. In particular, $\overline{A_1} \cap \cm_\loc = \varnothing$.

    We can now argue for $A_1$ as we did for $A_0$. We find that either (after a trivial reparametrization) $\cm := A_0 \cup A_1$ satisfies \ref{gen blowup alternative} or \ref{gen noncompact}, or that $A_1$ terminates in a second endpoint $(x_2,\lambda_2)$ as $s \nearrow 2$, and that at this point $A_1$ is uniquely paired with another distinguished arc $A_2$; see Figure~\ref{arcs figure}. The parametrization can be chosen to be analytic near $(x_2,\lambda_2)$, and $\overline{A_1 \cup A_2} \cap \cm_\loc = \varnothing$.
    Continuing in this way, we can assume that we have an infinite sequence of distinguished arcs 
    \begin{align*}
      A_n = \{ (x(s),\lambda(s)) : n < s < n+1 \}
    \end{align*}
    with (two-sided) limits $(x(s),\lambda(s)) \to (x_n,\lambda_n)$ as $s \to n$. The unique way consecutive arcs are paired at the endpoints $(x_n,\lambda_n)$, together with the fact that the only ``unpaired'' endpoint $(x_0,\lambda_0)$ lies in $\dell (\genU \times \genI)$ and not in $\genU \times \genI$, guarantees that all of the $A_n$ are distinct. Setting $\cm := \cup_n A_n$, the implicit function theorem guarantees as before that 
    $\cup_{n \ge 1}\overline{A_n} \cap \cm_\loc = \varnothing$, which implies \ref{gen reconnect}.

    It remains to show, in this case, that either \ref{gen blowup alternative} or \ref{gen noncompact} holds. Seeking a contradiction, suppose that neither holds, so that we can find a sequence $s_n \to \infty$ with $N(s_n) \le M < \infty$ and where $(x(s_n),\lambda(s_n)) \to (x_\infty,\lambda_\infty) \in \genU \by \genI$ and $\genG(x_\infty,\lambda_\infty) = 0$. Therefore any neighborhood of the point $(x_\infty,\lambda_\infty)$ intersects infinitely many distinguished arcs. 
    But, applying the analytic Lyapunov--Schmidt reduction near $(x_\infty,\lambda_\infty)$ along with the structure theorem for real-analytic varieties as in Step~4 of the proof of \cite[Theorem 9.1.1]{buffoni2003analytic}, we see that sufficiently small neighborhoods of $(x_\infty,\lambda_\infty)$ only meet a finite number of distinguished arcs, a contradiction.
    
    Finally, we point out that this construction also yields (b): in the interior of a distinguished arc, one can apply the analytic implicit function theorem to reparameterize, and at the end points, we get it from the analytic Lyapunov--Schmidt reduction instead (cf.\ \cite[Theorem 9.1.1(d)]{buffoni2003analytic}).
  \end{proof}   
\end{theorem}
\begin{remark} \label{no loop remark} 
  This theorem represents a somewhat atypical global bifurcation result. More commonly, one begins with a local curve that bifurcates from a trivial solution and perhaps lies inside some open set $\mathcal{O}$.  Provided that the necessary compactness hypotheses are met, then one expects to obtain a statement of the form:  there exists a global continuation of the local curve that either is unbounded, contains a sequence limiting to $\partial \mathcal{O}$, or is a closed loop.  

  The situation here is quite different as  
  the trivial solution is not the initial point of bifurcation, 
  and in fact it lies in $\partial ( \genU  \times \genI)$.  It is therefore impossible to have a closed loop in this case, since $\cm = \bigcup_{n} A_n$, where each distinguished arc $A_n$ is \emph{distinct}. 
\end{remark}

Next let us look more closely at alternative \ref{noncompact} of Theorem \ref{generic global theorem}.  
In doing so, we will restrict our attention to the case of an elliptic PDE that generalizes the height equation \eqref{w equation}.  Fix an integer $k \ge 0$, and let $\Omega := \R \by B$ be an infinite cylinder whose base $B \subset \R^{n-1}$ is a bounded $C^{k+2+\alpha}$ domain. We will denote points in $\Omega$ as $(x,y)$ where $x \in \R$ and $y\in B$. 
Partition the components of $\dell B$ as $\dell B = \dell_1 B \cup \dell_2 B$ (either may be empty). Consider a nonlinear elliptic problem
\begin{equation}
 \label{lem:compact:main} \left\{ 
 \begin{alignedat}{2}
   \mathcal F(y,u,Du,D^2u,\lambda)&=0 && \text{ in } \Omega,\\
   \mathcal G(y,u,Du,\lambda)&=0 && \text{ on } \R \by \partial_1 B ,\\
   u &=0 && \text{ on } \R \by \partial_2 B,
 \end{alignedat} \right. 
\end{equation}
  where the parameter $\lambda \in \R^m$ and where $\mathcal F$ and $\mathcal G$ have the regularity
  \begin{align*}
    \mathcal F \in C^{k+1+\alpha}_\bdd(\overline{ B} \by \R \by \R^n \by \mathbb S^{n\by n} \by \R^m),
    \qquad 
    \mathcal G \in C^{k+2+\alpha}_\bdd(\R \by \partial_1 B \by \R \by \R^n \by \R^m).
  \end{align*}
  Moreover, assume that $\mathcal F$ is uniformly elliptic in that
  \begin{align}
    \label{lem:compact:elliptic}
    \mathcal F_{r^{ij}}(y,z,\xi,r,\lambda) \eta_i \eta_j \ge c \abs\eta^2
  \end{align}
  for all $y,z,\xi,r,\eta,\lambda$ while $\mathcal G$ is uniformly oblique in that 
  \begin{align}
    \label{lem:compact:oblique}
    \mathcal G_{\xi^i}(y,z,\xi,\lambda) \nu_i > c
  \end{align}
  for all $z,\xi,\lambda$ and $y\in \dell_1 B$, where here $\nu$ is the outward pointing normal to $\dell\Omega$.
\begin{lemma}[Compactness or front]\label{lem:compact:gen}
    If $\{(u_n,\lambda_n)\}$ is a sequence of solutions to \eqref{lem:compact:main} that is uniformly bounded in $C^{k+2+\alpha}_\bdd(\overline \Omega) \by \R^m$, with the additional monotonicity property
  \begin{align}
    \label{lem:compact:gen:monotone}
    \textup{$u_n(x,y)$ is even in $x$ and has $\dell_x u_n \le 0$ for $x \ge 0$}
  \end{align}
  for each $n$ as well as the asymptotic condition
  \begin{align}
    \label{lem:compact:gen:asymptotic}
    \lim_{\abs x \to \infty} u_n(x,y) = U(y)
    \quad
    \textup{uniformly in $y$}
  \end{align}
  for some fixed function $U \in C^{k+2+\alpha}_\bdd(\overline B)$, then either
  \begin{enumerate}[label=\rm(\roman*)]
  \item \label{lem:compact:gen:cpt} we can extract a subsequence so that $u_n \to u$ in $C^{k+2+\alpha}_\bdd(\overline\Omega)$; or
  \item \label{lem:compact:gen:bore} we can extract a subsequence and find $x_n \to +\infty$ so that the translated sequence 
    $\{\til u_n\}$ defined by $\til u_n = u_n(\placeholder+x_n, \placeholder)$
    converges in $C^{k+2}_\loc(\overline\Omega)$ to some $\til u \in C^{k+2+\alpha}_\bdd(\overline\Omega)$ which solves \eqref{lem:compact:main} and has $\til u \not \equiv U$ and
    $\dell_x \til u \le 0$.
  \end{enumerate}
  \begin{proof}
    Without loss of generality we can assume that $\lambda_n \to \lambda \in \R^m$. Suppose first that 
    \begin{align}
      \label{eqn:equidecay}
      \lim_{\abs x\to\infty} \sup_n \sup_y \abs{u_n(x,y)-U(y)} = 0. 
    \end{align}
    We will show that alternative \ref{lem:compact:gen:cpt} occurs. Using Arzel\`a--Ascoli, \eqref{eqn:equidecay}, and a diagonalization argument, we can extract a subsequence so that $u_n \to u$ in $C^{k+2}_\loc(\overline \Omega)$ and $C^0_\bdd(\overline \Omega)$ for some $u \in C^{k+2+\alpha}_\bdd(\overline\Omega) \cap C^0_0(\overline\Omega)$. It remains to show that $u_n \to u$ in $C^{k+2+\alpha}_\bdd(\overline \Omega)$. For this we observe that $v_n := u_n-u$ satisfies a linear elliptic equation
    \begin{equation}
      \label{eqn:vn12k} \left\{     \begin{alignedat}{2} 
        a^{ij}_n D_{ij} v_n + b^i_n D_i v_n + c_n v_n &= 0 &\qquad& \text{in } \Omega,\\
        \beta^i D_i v_n + \mu_n v_n &= 0 && \text{on } \R \by \partial_1 B,\\
        v_n &= 0 && \text{on } \R \by \partial_2 B,
      \end{alignedat} \right. 
    \end{equation}
    where the coefficients $a^{ij}_n, b^i_n, c_n, \beta^i_n, \mu_n$ are defined in terms of the convex combinations $u_n^{(s)} := su_n + (1-s)u$ and $\lambda^{(s)} := s\lambda_n + (1-s)\lambda$ by
    \begin{align*}
      a^{ij}_n
      &:= 
      \int_0^1 \mathcal F_{r^{ij}}(y, u_n^{(s)}, Du_n^{(s)}, D^2 u_n^{(s)},\lambda^{(s)})\, ds,
      &
      b^i_n
      &:= 
      \int_0^1 \mathcal F_{\xi^{i}}(y, u_n^{(s)}, Du_n^{(s)}, D^2 u_n^{(s)},\lambda^{(s)})\, ds,
      \\
      c_n
      &:= 
      \int_0^1 \mathcal F_z(y, u_n^{(s)}, Du_n^{(s)}, D^2 u_n^{(s)},\lambda^{(s)})\, ds,
      &
      \beta^i_n
      &:= 
      \int_0^1 \mathcal G_{\xi^{i}}(y, u_n^{(s)}, Du_n^{(s)},\lambda^{(s)})\, ds,
      \\
      \mu_n
      &:= 
      \int_0^1 \mathcal G_z(y, u_n^{(s)}, Du_n^{(s)},\lambda^{(s)})\, ds.
    \end{align*}
    From the assumptions on $\mathcal F,\mathcal G$ and the uniform bounds on $u_n$ and $u$ in $C^{k+\alpha}_\bdd(\overline\Omega)$, the $C^{k+\alpha}$ norms of $a^{ij}_n,b^i_n,c_n$ as well as the $C^{k+1+\alpha}$ norms of $\beta^i_n,\mu_n$ are bounded uniformly in $n$. From \eqref{lem:compact:elliptic} and \eqref{lem:compact:oblique} we also see that $a^{ij}_n \eta_i \eta_j \ge c \abs\eta^2$ and $\beta^i \nu_i > c$ so that so that \eqref{eqn:vn12k} is uniformly elliptic with uniformly oblique boundary condition. Thus we have a Schauder estimate
    \begin{align*}
      \n{v_n}_{C^{k+2+\alpha}(\Omega)} \le C \n{v_n}_{C^0(\Omega)} 
    \end{align*}
    where the constant $C$ is independent of $n$. Since $u_n \to u$ and hence $v_n \to 0$ in $C^0_\bdd(\overline\Omega)$, this proves that $u_n \to u$ in $C^{k+2+\alpha}_\bdd(\overline\Omega)$ as desired.
    
    Now assume that \eqref{eqn:equidecay} does not hold; we will show that  \ref{lem:compact:gen:bore} occurs. We can find a sequence $\{(x_n,y_n)\} \subset \Omega$ with $x_n \to +\infty$ and $\varepsilon > 0$ so that \begin{align*}
      \abs{u_n(x_n,y_n)-U(y_n)} \ge \varepsilon 
    \end{align*}
    for all $n$. Extracting a subsequence we can assume that $y_n \to y_\infty \in \overline B$. Consider the translated sequence
    $\{\til u_n\}$ defined by
    \begin{align*}
      \til u_n(x,y) := u_n(x+x_n,y).
    \end{align*}
    Thanks to the uniform bounds on $u_n$ and hence $\til u_n$, we can extract a further subsequence so that $\til u_n \to \til u$ in $C^{k+2}_\loc(\overline \Omega)$ for some $\til u \in C^{k+2+\alpha}_\bdd(\overline \Omega)$. Since $\mathcal F$ and $\mathcal G$ have no explicit dependence on $x$, the $\til u_n$ are also solutions to \eqref{lem:compact:main}, and therefore the $C^{k+2}_\loc$ limit implies that 
    $\til u$ also solves \eqref{lem:compact:main}.
   
    By \eqref{lem:compact:gen:monotone}, we have $\dell_x \til u_n \le 0$ for $x \ge -x_n$, and hence $\dell_x \til u \le 0$ on $\overline \Omega$. Finally, to see that $\til u \not \equiv U$, we simply note that
    \begin{gather*}
      \abs{\til u(0,y_\infty)-U(y_\infty)}
      = \lim_{n \to \infty} \abs{u_n(x_n,y_n)-U(y_n)}
      \ge \varepsilon > 0.
      \qedhere
    \end{gather*}
  \end{proof}
\end{lemma}

\begin{remark}
  The above proof can easily be genearlized to the case where the $u_n$ are not necessarily even and where \eqref{lem:compact:gen:monotone} is replaced by the monotonicity of the $u_n$ for $\abs x > M$ for some fixed $M > 0$.
\end{remark}

\subsection{Proof of the main result} \label{proof main result section}

Recall from Section~\ref{subsec_spaces} that we can formulate the height equation \eqref{w equation} as a nonlinear operator equation $\F(w,F) = 0$ with $\F \maps U \to Y$ given in \eqref{operator equation} and \eqref{F domain}. 
Obviously $\F$ is real analytic on $U$, and from Lemma~\ref{fredholm lemma} it follows that $\F_w(w,F)$ is Fredholm with index $0$ whenever $(w,F) \in U$.
By Theorem~\ref{small amplitude theorem}, we know that there is a local curve 
\begin{align*}
  \cm_\loc = \{ (w^\epsilon,F^\epsilon) : 0 < \epsilon < \epsilon_* \} \subset U
\end{align*}
of nontrivial symmetric and monotone waves of elevation with $F$ slightly larger than $\Fcr$. (Recall from \eqref{def F epsilon} that $F^\epsilon = ( 1/\Fcr^2 - \epsilon )^{-1/2}$.) Moreover, $\F_w$ is invertible along $\cm_\loc$. 

Applying Theorem~\ref{generic global theorem} to our nonlinear operator $\F \maps U \to Y$, we obtain the following.
\begin{theorem}[Global continuation]\label{global theorem}
  The local curve $\cm_\loc$ is contained in a continuous curve of solutions, parametrized as
  \begin{align*}
    \cm = \{(w(s),F(s)) : 0 < s < \infty\} \subset U
  \end{align*}
  with the following properties.
  \begin{enumerate}[label=\rm(\alph*)]
  \item \label{alternatives} One of two alternatives must hold: either
    \begin{enumerate}[label=\rm(\roman*)]
    \item \textup{(Blowup)} \label{blowup alternative}
      as $s \to \infty$,
      \begin{align}
        \label{global blowup}
        N(s):= \n{w(s)}_X + \frac 1{\inf_R (w_p(s)+H_p)} +
        F(s) + \frac 1{F(s)-\Fcr} \to \infty;~\textrm{ or}
        \hspace*{-2em}
      \end{align}
    \item \textup{(Loss of compactness)} \label{noncompact} there exists a sequence $s_n \to \infty$ such that $\sup_n N(s_n) < \infty$ but $\{w(s_n)\}$ has no subsequences converging in $X$.
    \end{enumerate}
  \item \label{reparam} Near each point $(w(s_0),F(s_0)) \in \cm$, we can reparametrize $\cm$ so that the mapping $s\mapsto (w(s),F(s))$ is real analytic.
  \item \label{reconnect} $(w(s),F(s)) \not\in \cm_\loc$ for $s$ sufficiently large.
  \end{enumerate}
\end{theorem}

We will now use the qualitative results from Section~\ref{qualitative section} to pare down the alternatives in Theorem~\ref{global theorem} until we are left with only $\inf_R(w_p(s)+H_p) \to 0$, proving Theorem~\ref{main existence theorem}. 

First we consider alternative~\ref{noncompact}. Here we would like to apply Lemma~\ref{lem:compact:gen}, but first we need to know that these solutions have the required monotonicity properties \eqref{lem:compact:gen:asymptotic}. 
\begin{lemma}\label{global nodal}
  The nodal properties \eqref{elevation nodal properties} hold along the global bifurcation curve $\cm$. 
  \begin{proof}
    First we claim that \eqref{elevation nodal properties} holds along the local bifurcation curve $\cm_\loc$. By Theorem~\ref{small amplitude theorem}\ref{small amplitude theorem elevation}, any $(w,F) \in \cm_\loc$ is a wave of elevation in that $w > 0$ on $R \cup T$. Thus we can apply Theorem~\ref{symmetry theorem} to get $w_q < 0$ on $R \cup T$, and hence by Lemma~\ref{nodal: monotone implies nodal lemma} that \eqref{elevation nodal properties} holds.
    
    Let $V \sub \cm$ denote the set of all $(w,F) \in \cm$ satisfying \eqref{elevation nodal properties}. Since $\cm$ is a continuous curve, it is connected in $X \times \R$. By Lemmas~\ref{nodal open lemma} and \ref{nodal closed lemma}, $V \sub \cm$ is both relatively open and relatively closed. Since $\cm_\loc \sub V$, $V$ is nonempty, and we conclude that $V = \cm$ as desired.
  \end{proof}
\end{lemma}

Now we are ready to eliminate alternative~\ref{noncompact} in Theorem~\ref{global theorem}. We state the lemma in a slightly more general form for later convenience.
\begin{lemma}\label{lem:compact}
  Given a sequence of solutions $\{(w_n,F_n)\} \sub \cm$ to $\F(w,F) = 0$ with $\n{w_n}_X$ uniformly bounded, we can extract a subsequence so that $(w_n,F_n)$ converges in $X \by \R$ to a solution $(w,F)$ of $\F(w,F) = 0$. In particular, alternative~\ref{noncompact} in Theorem~\ref{global theorem} cannot occur.
  \begin{proof}
    We will apply a slight variant of Lemma~\ref{lem:compact:gen} to \eqref{height equation} in non-divergence form, setting $k :=1$, $\lambda := F$, $u := h$, $U := H$, 
  \[ 
      B := [-1,0],
      \quad
      \dell_1 B := \{0\},
      \quad
      \dell_2 B:= \{-1\},\]
    and 
    \begin{align}
      \label{nondivergence}
      \begin{aligned}
        \mathcal F(p,z,\xi,r,F)  
        &:=  (1+\xi_1^2) r_{22} - 2\xi_1 \xi_2 r_{12} + \xi_2^2 r_{11} 
        + \Big( \frac 1{2H_p^2} \Big)_p \xi_2^3 + \frac 1{F^2} \rho_p (H-z) \xi_2^3, \\
        \mathcal G(z,\xi,F) &:=  \frac{1+\xi_1^2}{2 \xi_2^2} - \frac 1{2H_p(0)^2} + \frac{1}{F^2} \rho(0) (z - 1).
      \end{aligned}
    \end{align}
    
    Given a sequence $\{(w_n,F_n)\}$ as in the statement of the lemma, Lemma~\ref{finite lemma} at once furnishes a uniform bound
    \begin{align*}
      \n{h_n}_X  +
      \frac 1{\inf_R \dell_p h_n}
      +
      F_n  \le M < \infty.
    \end{align*}
    Thus, for all $n$ and at any $(q,p) \in \overline\Omega$, $(p,h_n,Dh_n,D^2h_n,F_n) \in \mathscr D_M$, where
    \begin{align*}
      \mathscr D_M := \big\{
      (p,z,\xi,r,F) \in [-1,0] \by \R \by \R^2 \by \mathbb S^{2\by 2} \by \R :
      \xi_2 \ge \tfrac 1M,\ \Fcr \le F \le M \big\}
    \end{align*}
    Here $\mathbb{S}^{2 \by 2}$ is the set of symmetric $2\times 2$ real matrices.      We easily check that $\mathcal F, \mathcal G$ satisfy the requirements of Lemma~\ref{lem:compact:gen} when restricted to $\mathscr D_M$. Since $\mathscr D_M$ is convex, the proof of Lemma~\ref{lem:compact:gen} is easily extended to this setting. We conclude that, after extracting a subsequence, we can arrange for $F_n \to F \ge \Fcr$ and find $q_n \to +\infty$ so that $\tilde h_n(q_n+\placeholder,\placeholder)$ converges in $C^3_\loc(\overline\Omega)$ to a solution $\tilde h \in C^{3+\alpha}_\bdd(\overline\Omega)$ of \eqref{height equation} with $\tilde h_q \le 0$, 
    \begin{align}
      \label{contradict this}
      \tilde h \not \equiv H,
    \end{align}
    and $\tilde h_p \ge 1/M$. Moreover, since each $h_n \ge H$, we have $\tilde h \ge H$, and since $\flowforce(h_n)=\flowforce(H)$, $\flowforce(\tilde h) = \flowforce(H)$.

    Because $\tilde h_q \le 0$, and $\tilde h$ is bounded, for all $p$ we must have pointwise limits 
    \[ H_\pm(p) := \lim_{q \to \pm\infty} \tilde h(q,p).\]
     As $\tilde h \ge H$, we have 
    \begin{align}
      \label{going to pinch}
      H \le H_- \le \tilde h \le H_+.
    \end{align}
    Moreover, since $\flowforce(\tilde h)=\flowforce(H)$, $\flowforce(H_-)=\flowforce(H_+)=\flowforce(H)$. Combining these facts with Corollary~\ref{bores corollary}, we conclude that $H_+ \equiv H_- \equiv H$. But then \eqref{going to pinch} forces $\tilde h \equiv H$, contradicting \eqref{contradict this}.
  \end{proof}
\end{lemma}

Now we turn to alternative~\ref{blowup alternative} in Theorem~\ref{global theorem}, and clarify the way in which the first and last terms on the left-hand side of \eqref{global blowup} may be unbounded.  First we apply Corollary~\ref{bound on w and w_q cor} and Theorem~\ref{upper bound on F theorem} to bound the middle two terms in \eqref{global blowup} by the first term. 
\begin{lemma}[Finite Froude number and velocity]\label{finite lemma}
  Let $(w,F) \in U$ solve $\F(w,F) = 0$. If $\n w_{C^1(\Omega)} \le K$, there is a constant $C$ depending only on $K$ so that 
  \begin{align*}
    \frac 1{\inf_R (w_p+H_p)} + F  < C.
  \end{align*}
  \begin{proof}
    Corollary~\ref{bound on w and w_q cor} immediately gives $\inf_R (w_p + H_p) > \delta_*$, while Theorem~\ref{upper bound on F theorem} gives
    \begin{gather*}
      F^2 
      \le 
      C
      \n{h_p}_{L^\infty}
      \le C \n{H_p+w_p}_{L^\infty}
      \le C (1 + \n w_{C^1(\Omega)})
      \le C.
      \qedhere
    \end{gather*}
  \end{proof}
\end{lemma}
  
Next we deal with the fourth term in \eqref{global blowup} by showing that
uniform bounds on $\| w(s) \|_X$ imply that $F(s)$ is bounded away from $\Fcr$.
\begin{lemma}[Asymptotic supercriticality] \label{asymptotical supercritical lemma}  
  If $\n{w(s)}_X$ is uniformly bounded along $\cm$, then
  \begin{equation}
    \liminf_{s \to \infty} F(s) >  F_{\mathrm{cr}}. \label{liminf F > Fcr} 
  \end{equation}
\end{lemma}
\begin{proof} 
  Arguing by contradiction, suppose that there exists a sequence $s_n \to \infty$ with
  \begin{equation} \label{w bounded F to Fcr} 
    \limsup_{n \to \infty} \| w(s_n) \|_X < \infty 
    \qquad \text{and} \qquad  
    F(s_n) \to \Fcr.
  \end{equation} 
  Applying Lemma~\ref{lem:compact}, we can extract a subsequence so that $\{(w(s_n),F(s_n))\}$ converges in $X \by \R$ to a solution $(w^*,F^*)$ of $\F(w,F) = 0$ with critical Froude number $F^* = \Fcr$. By Theorem~\ref{froude lower theorem}\ref{no critical}, we know that this can only happen if $w^* = 0$. Consequently, $\| w(s_n) \|_X \to 0$. Since, by Lemma~\ref{global nodal}, each $w(s_n)$ is a wave of elevation, we have $(w(s_n), F(s_n)) \in \cm_\loc$ for $n$ sufficiently large by Lemma~\ref{uniqueness lemma}. But by Theorem~\ref{global theorem}\ref{reconnect}, $(w(s),F(s)) \not\in\cm_\loc$ for $s$ sufficiently large, so this is a contradiction.
\end{proof}

Together, Lemmas~\ref{lem:compact}, \ref{finite lemma}, \ref{asymptotical supercritical lemma},
and Theorem~\ref{global theorem} show that $\| w \|_X$ is necessarily unbounded along $\cm$.  The next result explains more precisely which derivatives of $w$ are growing in the limit.  
\begin{theorem}[Uniform regularity] \label{uniform regularity theorem}  
  For each $K > 0$ there exists a constant $C = C(K) > 0$ such that, if $(w, F) \in \cm$ with $\| w_p \|_{C^0({R})} < C$, then $\| w \|_{C^{3+\alpha}({R})} < K$.  
\end{theorem}
In other words, if $\|w(s)\|_X \to \infty$, then $\| w_p(s) \|_{C^0({R})} \to \infty$ as well. This is a consequence of the structure of the height equation and elliptic regularity theory.  Statements of this type are well known in the context of steady water waves (see, e.g.,  \cite[Section~6]{constantin2004exact},  \cite[Section~6]{walsh2009stratified}, \cite[Section~5]{wheeler2013solitary}, and \cite[Section~5]{walsh2014global}), so we only provide a sketch of the argument and relegate it to Appendix~\ref{appendix: large amplitude}. The crucial first step will be to apply Corollary~\ref{bound on w and w_q cor}.

At last, we are prepared to prove the main result.  Most of the work has already been done, and so all that remains is to assemble the various pieces laid out above.

\begin{proof}[Proof of Theorem~\ref{main existence theorem}]
  Let $\cm$ be given as in Theorem~\ref{global theorem}.  Applying Lemma~\ref{lem:compact}, we conclude that alternative~\ref{blowup alternative} in Theorem~\ref{global theorem} holds, i.e.,
  \begin{align*}
    \n{w(s)}_X + \frac 1{\inf_R (w_p(s)+H_p)} +
    F(s) + \frac 1{F(s)-\Fcr} \to \infty
  \end{align*}
  as $s \to \infty$. Using Lemma~\ref{finite lemma}, we can simplify this to 
  \begin{align*}
    \n{w(s)}_X + \frac 1{F(s)-\Fcr} \to \infty
  \end{align*}
  and using Lemma~\ref{asymptotical supercritical lemma} we conclude $\| w(s) \|_{X} \to \infty$. Theorem~\ref{uniform regularity theorem} now yields
  \begin{align*}
    \| w_p(s) \|_{C^0(\overline R)} \to \infty.
  \end{align*}
  Switching to our dimensionless Eulerian variables using \eqref{h to uv}, we find, for $s$ sufficiently large,
  \begin{align}
    \label{dimensionless stagnation}
    \inf_{\tilde\Omega(s)} (\tilde c-\tilde u(s)) 
    = \frac 1{\sup_R \abs{\sqrt\rho h_p}}
    \le \frac 1{\min \sqrt\rho} \frac 1{\sup_R \abs{w_p} - \max{H_p}}
    \to 0,
  \end{align}
  where $\tilde\Omega(s)$ is the dimensionless fluid domain corresponding to $(w(s),F(s))$, and $(\tilde u(s), \tilde v(s))$ is the dimensionless velocity field. To prove a similar statement 
  for the dimensional horizontal velocity 
  \begin{align*}
    u -c = \frac m{\sqrt{\rho_0}d}(\tilde u - \tilde c)
    = F \sqrt{gd} (\tilde u - \tilde c),
  \end{align*}
  we combine \eqref{dimensionless stagnation} with the bound the bound \eqref{eqn:hurray} from Theorem~\ref{upper bound on F theorem},
  \begin{equation*}
    F^2(s) \le \frac C{\inf_{\Omega(s)} (\tilde c-\tilde u(s))},
  \end{equation*}
  to get
  \begin{align*}
    \inf_{\Omega(s)} (c-u(s))
    = F \sqrt{gd} \inf_{\tilde\Omega(s)}(\tilde c-\tilde u)
    \le C \sqrt{\inf_{\tilde\Omega(s)}(\tilde c-\tilde u)}
    \to 0.
  \end{align*}

  The regularity statements about the parameterization are inherited from those of Theorem~\ref{global theorem}.
\end{proof}

\appendix

\section{Proofs and calculations} \label{appendix calculations}

This appendix collects some of the more technical proofs and straightforward calculations, organized by section.  

\subsection{Proofs from Section~\ref{fredholm section}}
Using the reformulation \eqref{v equation} of $\F_w(0,F)\dot w = (f_1,f_2)$, we first show, as a preliminary step, that $\F_w(0,F)\dot w = f$ is uniquely solvable in the Hilbert space 
\begin{align}
  \mathbb{H} := \left\{ \dot w \in H^1(R) : \dot w|_B = 0 \textrm{ in the trace sense}\right\}.
\end{align}
\begin{lemma}[Weak solvability] \label{weak invertibility lemma} 
  If $F > \Fcr$, then, for each $(f_1, f_2) \in L^2(R) \by L^2(T)$, there exists a unique $\dot w \in \mathbb H$ solving $\F_w(0,F)\dot w = (f_1,f_2)$ in the sense of distributions. Equivalently, $v = w/\Phia$ is a weak solution to \eqref{v equation}.
  \begin{proof}
    By definition, a weak solution of \eqref{v equation} must satisfy 
    \begin{equation} 
      \mathscr{B}[v, u] = - \int_{R} \frac{f_1}{\Phia} u \, dq \, dp - \int_T \frac{f_2}{\Phia} u \, dq\label{weak v problem} 
    \end{equation}
    for each $u \in \mathbb{H}$, where the bilinear form $\mathscr{B}\colon \mathbb{H} \times \mathbb{H} \to \R$ is given by
    \begin{align*}
      \mathscr{B}[ v, u] := \int_{R} \left[ \frac{1}{H_p^3} v_p u_p  + \frac{1}{H_p} v_q u_q \right] \, dq \, dp +
      \int_{T} \left( \frac{\Phia_p}{H_p^3} - \frac{1}{F^2}\rho \Phia \right) v u \, dq - \int_{T} \frac{1}{H_p} v_q u\, dq.
    \end{align*}
    The boundedness of $\mathscr{B}$ is obvious, so consider the question of its coercivity.  We estimate 
    \begin{align*} \mathscr{B}[ v, v] &= \int_{R} \left[ \frac{1}{H_p^3} v_p^2  + \frac{1}{H_p} v_q^2 \right] \, dq \, dp +
      \int_{T} \left( \frac{\Phia_p}{H_p^3} - \frac{1}{F^2}\rho \Phia \right) v^2 \, dq - \int_{T} \frac{1}{H_p} \partial_q ( \frac{1}{2} v^2)\, dq \\
      & \geq C \| v \|_{\dot{H}^1(R)}^2.
    \end{align*}
    To derive the second line, we have used \eqref{Phia top} to conclude the second term is nonnegative, while observing the third term is an integral of pure derivative in $q$.   Because the problem is set on a strip with homogeneous Dirichlet condition on the bottom $\{ p= -1\}$, we may apply Poincar\'e's inequality to conclude that the $\dot{H}^1$ norm is equivalent to the full $H^1$ norm.  This proves that $\mathscr{B}$ is coercive.  The unique (weak) solvability of \eqref{weak v problem} then follows directly from Lax--Milgram.   \end{proof}
\end{lemma}

Lemma~\ref{weak invertibility lemma} shows that $\F_w(0,F)$ is invertible in the Sobolev setting, but we need more than this as we wish to work with the domain $\Xb$ whose elements need not vanish at infinity.  To pass from one regime to the other will be accomplished in two steps.  First, we prove that $\F_w(0,F)$ is injective as a mapping from $\Xb$ to $\Yb$.  This will imply that $\F_w(0,F)$ is locally proper by a translation argument (cf., e.g., \cite{volpert2003degree}).  Using this fact, we will be able to infer the surjectivity of $\F_w(0,F)$ from its weak invertibility and a limiting argument.

\begin{lemma}[Strong injectivity] \label{strong injectivity lemma} 
  For $F > \Fcr$, there are no nontrivial solutions $\dot w \in \Xb$ of $\F_w(0,F)\dot w = 0$.
  \begin{proof}  
    Let $\dot w \in \Xb$ be a solution to $\F_w(0,F)\dot w = 0$, and let $v = \dot w/\Phia$ be the corresponding solution of \eqref{v equation} with $f_1 = f_2 = 0$. It suffices to show that $v \equiv 0$. For any $\delta > 0$ the function $u := \sech(\delta q) v$ lies in $\mathbb H$. Writing down the equation solved by $u$, we can argue as in the proof of Lemma~\ref{weak invertibility lemma} to conclude $u \equiv 0$, provided that $\delta$ is sufficiently small. This in turn forces $v \equiv 0$ and hence $\dot w \equiv 0$ as desired.
    \end{proof} 
\end{lemma}

\begin{lemma}[Local properness] \label{local properness lemma} 
  For $F > \Fcr$, $\F_w(0,F) \colon \Xb \to \Yb$ is locally proper.  That is, for any compact set $K \subset \Yb$ and any closed and bounded set $D \subset \Xb$, $\F_w(0,F)^{-1}(K) \cap D$ is compact in $\Xb$.
\end{lemma}
\begin{proof}
  Since the coefficients of $\F_w(0,F)$ are independent of the horizontal variable $q$, this follows immediately from  \cite[Lemma~A.7]{wheeler2013solitary}, which is proved using Schauder estimates and a translation argument.
\end{proof}
  
\begin{proof}[Proof of Lemma~\ref{strong invertibility lemma}]
  We have already confirmed in Lemma~\ref{strong injectivity lemma} that $\F_w(0,F)$ is injective between these spaces, so it suffices to show that it is surjective. Fix $(f_1,f_2) \in \Yb$. We will construct a solution $\dot w \in \Xb$ of $\F_w(0,F)\dot w = (f_1,f_2)$.

  First, since $\F_w(0,F)$ has trivial kernel by Lemma~\ref{strong injectivity lemma} and is locally proper by Lemma~\ref{local properness lemma}, a standard argument shows that it enjoys an improved Schauder estimate
  \begin{align}
    \label{schauderlater}
    \n{\dot w}_\Xb \le C\n{\F_w(0,F)\dot w}_\Yb
  \end{align}
  with no $\n{\dot w}_{C^0}$ term on the right-hand side.

  For $\delta > 0$, define
  \begin{align*}
    f_{1,\delta} := \sech(\delta q) f_1 \in L^2(R),
    \qquad 
    f_{2,\delta} := \sech(\delta q) f_2 \in L^2(T).
  \end{align*}
  By Lemma~\ref{weak invertibility lemma}, there exist unique weak solutions $\dot w_\delta \in \mathbb H$ to $\F_w(0,F) \dot w_\delta = (f_{1,\delta},f_{2,\delta})$, and by standard elliptic regularity theory $\dot w_\delta \in \Xb$. Thus \eqref{schauderlater} yields the bound
  \begin{align*}
    \n{\dot w_\delta}_\Xb
    \le C\n{(f_{1,\delta},f_{2,\delta})}_\Yb
    \le C\n{(f_1,f_2)}_\Yb.
  \end{align*}
  In particular, $\dot w_\delta$ is bounded in $\Xb$ uniformly in $\delta$, so that we can extract a subsequence converging in $C^3_\loc(\overline R)$ to a function $\dot w \in \Xb$. As $f_{1,\delta} \to f_1$ in $C^1_\loc(\overline R)$ and $f_{2,\delta} \to f_2$ in $C^2_\loc(T)$, we conclude that $\F_w(0,F) \dot w = (f_1,f_2)$, completing the proof.
\end{proof}
  
\begin{proof}[Proof of Lemma~\ref{fredholm lemma}]
  Fix $(w,F) \in U$, and notice that, since $w \in X$, the coefficients of $\F_w(w,F)$ tend to those of $\F_w(0,F)$ as $\abs q \to \infty$. Because $\F_w(0,F) \maps \Xb \to \Yb$ is invertible by Lemma~\ref{strong invertibility lemma}, the statement then follows from \cite[Lemmas~A.12 and A.13]{wheeler2015pressure}.
\end{proof}

\subsection{Proofs and calculations from Section~\ref{small-amplitude section}} \label{appendix: small amplitude} 

\begin{proof}[Proof of Lemma~\ref{hamiltonian spectral lemma}\ref{spectral bound}]
  Let $u = (w, r) \in \D(\linear)$, $\xi\in \R$ be given, and denote $(f,g) := (\linear-i\xi)u$,
  that is, 
  \begin{equation}
    \label{resolvent}  
    f = H_p r - i\xi w , \qquad 
    g = \displaystyle -\left( {w_p \over H^3_p} \right)_p + {1\over \Fcr^2} \rho_p w - i \xi r.  
  \end{equation}
  Recall also that the definition of $\D(\linear)$ implies that 
  \begin{equation}
    \label{resolvent bc} 
    w(-1) = r(-1) = 0, \qquad
    \displaystyle \left.\LC -{w_p \over H^3_p} + {1\over \Fcr^2} \rho w \RC\RV_{p = 0} = 0. 
  \end{equation}

  From the bottom boundary conditions, Sobolev embedding theorem, Poincar\'e inequality, and an interpolation argument, we have 
  \begin{equation*}
    \|w\|_{L^\infty} \lesssim \|w\|_{H^1} \lesssim \|w_p\|_{L^2}, \quad \|r\|_{L^\infty} \lesssim \|r\|_{H^1} \lesssim \|r_p\|_{L^2}, \qquad \| w \|_{H^2} \lesssim \| w_{pp} \|_{L^2}.
  \end{equation*}
  Following \cite[Lemma~3.4]{groves2007spatial} and \cite[Lemma~4.9]{wheeler2013solitary}, we compute 
  \begin{align*}
    {|f_p|^2 \over H^3_p} + H_p |g|^2 &=  {\LV (rH_p)_p \RV^2 \over H^3_p} + H_p \LV  \LC{w_p \over H^3_p} \right)_p - { \rho_p\over \Fcr^2} w\RV^2 + |\xi|^2 \LC {|w_p|^2\over H^3_p} + H_p|r|^2 \RC \\
    & \qquad + 2\xi \imagpart {\left\{{-(rH_p)_p \bar w_p \over H^3_p} + \bar rH_p \LB \LC {w_p\over H^3_p} \RC_p - {\rho_p w  \over \Fcr^2} \RB\right\}}.
  \end{align*}
  Integrating over $[-1,0]$, applying integration by parts, and using the facts that $H\in C^{3+\alpha}$ and $H_p>0$, we eventually arrive at the estimate  
  \begin{align*}
    C(\|f\|^2_{H^1} + \|g\|^2_{L^2}) & \geq \|w_{pp}\|^2_{L^2} + \|r_p\|^2_{L^2} + |\xi|^2 (\|w_p\|^2_{L^2} + \|r\|^2_{L^2}) \\
    & \quad - (|\xi| + C) \LC \|w\|^2_{H^1} + \|r\|^2_{L^2} \RC - C|\xi| |r(0) \bar w(0)|.
  \end{align*}
  Note that $|w(0)| \leq C\|w_p\|_{L^2}$. To control $|r(0)|$, we use the first equation in \eqref{resolvent} and the boundary condition \eqref{resolvent bc} to conclude that
  \begin{align*}
    \LV H_p(0)r(0) \RV^2 & = 2 \int^0_{-1} \realpart \LC H_p \bar{r} (H_p r)_p \RC \, dp = 2 \realpart \int^0_{-1} H_p \bar r (f_p + i\xi w_p) \, dp \\
    & \leq \delta^2 \LC \|f_p\|^2_{L^2} + |\xi|^2\|w_p\|^2_{L^2} \RC + {1\over \delta^2} \|r\|^2_{L^2},
  \end{align*}
  where $\delta > 0$ is to be determined.

  Putting all of the above together, and choosing $|\xi|$ sufficiently large and $\delta$ sufficiently small, we have
  \begin{align*}
    C(\|f\|^2_{H^1} + \|g\|^2_{L^2}) & \geq \|w_{pp}\|^2_{L^2} + \|r_p\|^2_{L^2} + |\xi|^2 (\|w_p\|^2_{L^2} + \|r\|^2_{L^2}) \\
    & \gtrsim \|w\|^2_{H^2} + \|r\|^2_{H^1} + |\xi|^2 (\|w\|^2_{H^1} + \|r\|^2_{L^2}),
  \end{align*}
  which proves the second part of the theorem.
\end{proof}

\begin{proof}[Proof of Lemma~\ref{change of vars lem}\ref{dynamic to elliptic}]
  The argument is similar to the proof of \cite[Lemma~4.3]{wheeler2013solitary}. Note that $u \in C^4_0(\R,\X) \cap C^3(\R,\U)$ implies $w \in C^4_0(\R,H^1) \cap C^3_0(\R,H^2)$ and hence, since $\alpha \le 1/2$, $w \in C^4_0(\R,C^\alpha) \cap C^3_0(\R,C^{1+\alpha})$. This is the only place in the paper where the assumption $\alpha \le 1/2$ is used.
  We claim that $w \in C^{3+\alpha}_0(\overline R) = C^{3+\alpha}_\bdd(\overline R) \cap C^3_0(\overline R)$. To see this, first observe that $w \in C^4_0(\R,C^\alpha) \cap C^3_0(\R,C^{1+\alpha})$ implies 
  \begin{align}
    \label{W uglier}
    w,w_q,w_{pq},w_{qq},
    w_{pqq},w_{qqq}
    \in C^\alpha_0(\overline R).
  \end{align}
  Introducing the notation
  \begin{align}
    W &:= (H,H_p,H_{pp},w,w_p,w_{pq},w_{qq}),
  \end{align}
  we can abbreviate \eqref{W uglier} as 
  $W, W_q \in C^\alpha_0(\overline R)$. It remains to show that $w_{pp}, w_{ppq}, w_{ppp} \in C^\alpha_0(\overline R)$. Using the equation, we can solve explicitly for $w_{pp}$ in terms of $W$, say $w_{pp} = f(W)$. From the form of $f$ we immediately discover that 
  \begin{align}
    \label{W ugliest}
    w_{pp} = f(W) \in C^\alpha_0(\overline R),
    \qquad 
    w_{ppq} = f_W(W) W_q \in C^\alpha_0(\overline R),
  \end{align}
  so the last thing to verify is that $w_{ppp} \in C^\alpha_0(\overline R)$. But, from \eqref{W uglier}, \eqref{W ugliest}, 
  and the fact that $H \in C^{3+\alpha}([-1,0])$, we know that $W_p \in C^\alpha_0(\overline R)$. Differentiating $w_{pp} = f(W)$ with respect to $p$ thus yields
  $w_{ppp} = f_W(W)W_p \in C^\alpha_0(\overline R)$ as desired.
\end{proof}

\subsection{Calculation of the reduced system} \label{appendix: reduced system calc}
In this subsection, we present the computation of the leading order part of the reduced system in Lemma~\ref{center manifold lemma}\ref{best reduced}.

Let us now record the variations of the Hamiltonian that will be required to derive the reduced system. We will write $u = (w, r)$, and similarly for variations $\dot{u} = (\dot{w}, \dot{r})$ and so on.  First note that for an arbitrary $u$ and variation $\dot{u}$, we have 
\begin{align} 
  \label{Hu} 
  \ham_u^\epsilon(u)[ \dot{u} ] & = \int_{-1}^0 \left[ r \dot{r} + \frac{\dot{w}_p}{(H_p+w_p)^3} + \frac{1}{(F^\epsilon)^2} \rho \dot{w} \right] (w_p + H_p) \, dp \\
  \notag
  & \qquad + \int_{-1}^0 \left[ \int_0^p  \frac{1}{(F^\epsilon)^2} \rho H_p \, dp^\prime - \frac{1}{2H_p^2} + \frac{1}{2} r^2 - \frac{1}{2(H_p+w_p)^2} + \frac{1}{(F^\epsilon)^2}  \rho w \right] \dot{w}_p \, dp.
\end{align}
Taking a second and third derivative in $u$ yields
\begin{align} 
  \notag
  \ham_{uu}^\epsilon(u) [\dot{u}, \, \ddot{u}] & = \int_{-1}^0 \left( \dot{r} \ddot{r} - \frac{3 \dot{w}_p \ddot{w}_p}{(H_p + w_p)^4} \right) (w_p + H_p) \, dp \\
  \label{Huu}
  & \qquad + \int_{-1}^0 \left( r \dot{r} + \frac{\dot{w}_p}{(H_p + w_p)^3} + \frac{1}{(F^\epsilon)^2} \rho \dot{w} \right) \ddot{w}_p \, dp \\
  \notag
  & \qquad + \int_{-1}^0 \left( r \ddot{r} + \frac{\ddot{w}_p}{(H_p + w_p)^3} + \frac{1}{(F^\epsilon)^2} \rho \ddot{w} \right) \dot{w}_p \, dp,\\
  \label{Huuu} 
  \ham_{uuu}^\epsilon(u)[ \dot{u},\ddot{u} , \dddot{u}] & = 3 \int_{-1}^0 \frac{\dot{w}_p \ddot{w}_p \dddot{w}_p}{(H_p+w_p)^4} \, dp  +  \int_{-1}^0 \left( \dot{r} \ddot{r} \dddot{w}_p + \dot{r} \dddot{r} \ddot{w}_p + \ddot{r} \dddot{r} \dot{w}_p \right) \, dp.
\end{align}
Finally, differentiating \eqref{Huu} with respect to $\epsilon$ and recalling the definition \eqref{def F epsilon} of $F^\epsilon$ leads to
\begin{equation} \label{Huuepsilon}
\ham_{uu\epsilon}^\epsilon[ \dot{u}, \ddot{u}] = -\int_{-1}^0  \rho \partial_p ( \dot{w} \ddot{w} )  \, dp.
\end{equation}

In Lemma~\ref{hamiltonian spectral lemma}(i), we found that the center space is spanned by the eigenvector $(\Phicr, 0)$ and generalized eigenvector $(0,  \Phicr/H_p)$ corresponding to $0$, 
where $\Phicr$ is given in Lemma~\ref{lem_minimality}.  Evaluating \eqref{Hu}, \eqref{Huu}, and \eqref{Huuu} at $u = 0$ and with 
\begin{equation}
  \dot{u} = \ddot{u} = \dddot{u}  
  = z_1 e_1 + z_2 e_2 =
  \Big( c_0^{-1/2} z_1 \Phicr, \, c_0^{-1/2} z_2 \frac{1}{H_p} \Phicr\Big) =: \ucenter,
\end{equation}
gives 
\begin{align}
  \ham_u^\epsilon(0)[\ucenter] &= \ham_{u\epsilon}^\epsilon(0)[\ucenter] = 0,  \label{Hu(0) on center}  \\
  \label{Huu(0) on center}
  \ham_{uu}^\epsilon(0)[ \ucenter, \ucenter] &= 
    z_2^2,  \\
    \ham_{uu\epsilon}^\epsilon(0)[\ucenter,\ucenter] & =  - c_0^{-1} c_1 z_1^2, \label{Huuuep(0) on center} \\ 
  \ham_{uuu}^\epsilon(0) [\ucenter, \ucenter , \ucenter]   &= 3 
  c_0^{-3/2} c_2 
  z_1^3 + 3 c_0^{-1/2} z_1 z_2^2, \label{Huuu(0) on center}
\end{align}
where $c_0$ is defined in \eqref{def c0} and $c_1,c_2$ are defined in \eqref{def c1 c2}.

Now consider the Taylor expansion of the Hamiltonian at $0$ taking variations only in the center directions.  That is, for each $\ucenter$ in the center space and 
$\epsilon \in [0, \epsilon_*)$, 
we consider the quantity 
\begin{equation} \label{def Kepsilon} 
  \cham(\ucenter) := \ham^\epsilon(\ucenter + \Theta^\epsilon(\ucenter)),
\end{equation}
where $\Theta^\epsilon$ is the reduction function of Lemma \ref{center manifold lemma}\ref{best reduced}.  It is helpful to study the related function $\twodcham \in C^\infty(\R^2, \R)$ defined by
\begin{equation}
  \twodcham(z_1, z_2) := \cham(z_1 e_1 + z_2 e_2),
\end{equation}
which will serve as the Hamiltonian for the reduced system.  

As $\cham$ is smooth and vanishes at $0$, we have the Taylor expansion
\begin{equation} \label{preliminary expansion K}
\cham(\ucenter)  =  \cham_{\ucenter}(0)[\ucenter]  + \frac{1}{2}   \cham_{\ucenter \ucenter}(0)[\ucenter, \ucenter]  + \frac{1}{6} \cham_{\ucenter\ucenter\ucenter}(0)[\ucenter, \ucenter, \ucenter] + \mathcal{O}(\| \ucenter \|^4).
\end{equation}
The derivatives of $\cham$ above can be computed from \eqref{def Kepsilon} and our previous calculations of the variations of $\ham^\epsilon$.   Indeed,
\[
\cham_{\ucenter}(\ucenter)[\varuc]  =  \ham_u^\epsilon(\ucenter + \Theta^\epsilon(\ucenter)) (1+\Theta_{\ucenter}^\epsilon(\ucenter))[\varuc],
\]
and hence 
\[ \cham_{\ucenter}(0)[ \varuc ] = \ham_u^\epsilon(0) (1+\Theta_{\ucenter}^\epsilon(0))[\varuc] = 0.\]
This implies that the linear term in \eqref{preliminary expansion K} vanishes.  For the quadratic term we will also expand in $\epsilon$ near $\epsilon = 0$.  With that in mind, we compute
\begin{align*}
\chamzero_{\ucenter \ucenter}(0)[\varuc, \varuc] & = \ham_{uu}^0(0)[ \varuc, \varuc], \\
\chamzero_{\epsilon \ucenter \ucenter}(0)[\varuc, \varuc] & = \ham_{\epsilon u u}^0(0)[ \varuc, \varuc] + 2 \ham_{uu}^0(0)[ \varuc, \, \Theta_{\epsilon \ucenter}^0(0) \varuc].
\end{align*}  
However, in light of \eqref{Huu}, we see that for any variation $\dot{u}$ 
\begin{align*}
 \ham_{uu}^0(0)[(\Phicr,0), \dot{u}]  &= \int_{-1}^0 \left[ -\frac{\partial_p \Phicr }{H_p^3} \dot{w}_p + \mu_{\textrm{cr}} \rho \partial_p (\dot{w} \Phicr) \right] \, dp \\
 & = \int_{-1}^0 \left[ \left( \frac{\partial_p \Phicr }{H_p^3} \right)_p - \frac 1{\Fcr^2} \rho_p \Phicr  \right] \dot{w} \, dp + \left[ -\frac{\partial_p \Phicr }{H_p^3} + \mu_{\textrm{cr}} \rho \Phicr \right] \dot{w} \bigg|^0 = 0,\end{align*}
given the equation satisfied by $\Phicr$.  Thus,
\begin{align*}
  \chamzero_{\epsilon \ucenter \ucenter}(0)[\varuc, \varuc] & = 
  -c_0^{-1} c_1
  z_1^2 
   + 2 \ham_{uu}^0(0) \left[ e_2, \Theta_{\epsilon \ucenter}^0(0) e_1\right] z_1 z_2 \\
  & \qquad +  2 \ham_{uu}^0(0)\left[ e_2,
  \Theta_{\epsilon \ucenter}^0(0) e_2 \right]  z_2^2.
\end{align*}
The quadratic terms in \eqref{preliminary expansion K} can then be written as 
\begin{align}
  \cham_{\ucenter \ucenter}(0)[ \ucenter, \ucenter] & = \ham_{uu}^0(0)[ \ucenter, \ucenter] + \epsilon \left( \ham_{\epsilon u u}^0(0)[ \ucenter, \ucenter] + 2 \ham_{uu}^0(0)[ \ucenter, \, \Theta_{\epsilon \ucenter}^0(0) \ucenter] \right)  \nonumber \\
  &\qquad+ \mathcal{O}(\epsilon^2 \| \ucenter \|^2 ) \nonumber  \\
  & = z_2^2 - \epsilon c_0^{-1} c_1  z_1^2 
  + \mathcal{O}(|\epsilon| |(z_1, z_2)| |z_2| )
  + \mathcal{O}(\epsilon^2 |(z_1, z_2)|^2) 
  . 
  \label{quad term K taylor}
\end{align}

Finally, for the cubic term we see that
\begin{align*}
  \cham_{\ucenter \ucenter \ucenter}(\ucenter)[\varuc, \vvaruc, \vvvaruc] & =  \ham^{\epsilon}_{uuu}(u) [ (1+ \Theta_{\ucenter}^\epsilon(\ucenter)) \varuc, \,  (1+ \Theta_{\ucenter}^\epsilon(\ucenter)) \vvaruc, \, (1+ \Theta_{\ucenter}^\epsilon(\ucenter)) \vvvaruc ] \\
& \qquad +   \ham_{uu}^\epsilon(u) [ \Theta_{\ucenter \ucenter}^\epsilon(u) [\vvaruc, \vvvaruc],  \, (1+\Theta_{\ucenter}^\epsilon(\ucenter)) \varuc ] \\
& \qquad +  \ham_{uu}^\epsilon(u) [ (1+\Theta_{\ucenter}^\epsilon(\ucenter)) \vvaruc, \,   \Theta_{\ucenter \ucenter}^\epsilon(u)[\varuc, \vvvaruc ]]. 
\end{align*}
Evaluating this at $u = 0$ and along the diagonal, we obtain
\begin{equation} 
  \cham_{\ucenter \ucenter \ucenter}(0) [\ucenter, \ucenter, \ucenter]  =  \ham_{uuu}^\epsilon(0) [\ucenter, \ucenter, \ucenter] + 2 \ham_{uu}^\epsilon(0) [ \ucenter, \,  \Theta_{\ucenter \ucenter}^\epsilon(0) [\ucenter, \ucenter] ].
\label{K cubic terms at 0}
\end{equation}
Here we have used the fact that $\Theta_{\ucenter}^\epsilon(0) = 0$.  The first term on the right-hand side above can be found via \eqref{Huuu(0) on center}; 
for the second, we again note that $\ham_{uu}(0)$ vanishes when one variation is taken in the $(\Phicr, 0)$ direction.  
In total,  the contribution of the cubic terms in \eqref{preliminary expansion K} is thus 
\begin{equation} \label{cubic terms K taylor} 
\begin{split}
  \cham_{\ucenter \ucenter \ucenter}(0) [\ucenter, \ucenter, \ucenter]  &= 3 c_0^{-3/2} c_2 z_1^3 + \mathcal{O}(|z_2| |(z_1, z_2)|^2).
\end{split}
\end{equation}

Combining \eqref{preliminary expansion K}, \eqref{quad term K taylor}, and \eqref{cubic terms K taylor}, we arrive at the following expansion for the reduced Hamiltonian $\twodcham$:
\begin{equation} \label{expansion K} 
  \begin{split}
    \twodcham(z_1, z_2) &= 
    \frac{1}{2} z_2^2 - \frac{1}{2} \epsilon c_0^{-1} c_1 z_1^2 
    +  \frac{1}{2} c_0^{-3/2} c_2 z_1^3\\ 
    &\qquad 
    + \mathcal{O}(|z_2| |(z_1, z_2)|^2) 
    + \mathcal{O}(|\epsilon| |z_2| |(z_1, z_2)|)
    + \mathcal{O}(|(\epsilon, z_1, z_2)|^2 |(z_1, z_2)|^2) 
    .
  \end{split}
\end{equation}

\subsection{Proofs from Section~\ref{global bifurcation section}} \label{appendix: large amplitude}

\begin{proof}[Proof of Theorem~\ref{uniform regularity theorem}]
Let $K > 0$ be given.  Throughout the proof, we let $C > 0$ denote a generic constant that depends only on $\| w_p \|_{C^0({R})}$.  In light of Corollary~\ref{bound on w and w_q cor}, we already know that $\| w \|_{C^1({R})}$ can be controlled by $\| w_p \|_{C^0({R})}$.  It remains now to bound the higher order derivatives.  

First, we establish uniform H\"older norm estimates for the gradient.  Note that the height equation \eqref{height equation} can be written abstractly as 
\[ \mathcal{F}(p, h, Dh, D^2 h, F) = 0 \textrm{ in } R, \qquad \mathcal{G}(h, Dh, F) = 0 \textrm{ on } T, \qquad h = 0 \textrm{ on } B.\]
where 
\[ \mathcal{F} \colon [-1,0] \times \R \times \R^2 \times \mathbb{S}^{2 \times 2} \by \R_+ \to \R, \qquad \mathcal{G} \colon  \R \times \R^2 \times \R_+ \to \R\]
are defined by \eqref{nondivergence}. 
We are interested in deriving bounds that are uniform for 
\begin{equation}
  \sup_{{R}} h_p > \delta_* > 0, \qquad \| h \|_{C^1({R})} + F + \frac 1F < C.
\end{equation}
Translating this to the notation above, this means that one should consider the restriction of $\mathcal{F}$ and $\mathcal{G}$ to sets of the form 
\[ V := \left\{ (z, \xi, F) \in  \R \times \R^2 \times \R_+ : ~  \xi_2 > \delta_*, ~ z > 0, ~ z +  |\xi_1| + |\xi_2| + F + \frac 1F < C \right\}.\]
This can be achieved in the usual way by using cutoff functions.  For $(z,\xi,F) \in V$, $p \in [-1,0]$, and $r \in \mathbb{S}^{2 \times 2}$, it is easy to confirm that 
\begin{gather*} 
  c_1  I \leq \mathcal{F}_r(p, z, \xi, r, F)\leq c_1 c_2 I, \qquad |\mathcal{G}_\xi(z, \xi,F)| > c_3,  \\
  |\mathcal{F}(p, z, \xi, 0, F)| < c_1 c_4 |p|^{\alpha-1}, 	\\ 
  (1+|r|) | \mathcal{F}_\xi(p, z, \xi, r, F) | + |\mathcal{F}_z(p,z, \xi, r, F)| + |\mathcal{F}_p(p, z, \xi, r, F)| \leq  c_1c_5(|r|^2 + |p|^{\alpha-2}), 
\end{gather*}
where $I$ is the $2 \times 2$ identity matrix and $c_1$, $c_2$, $c_3$, $c_4$, and $c_5$ are positive constants depending only on $C$ and $\delta_*$.  
Moreover,  for any $(z, \xi, F)$, $(z', \xi', F) \in V$ there exists a positive constant $c_6 > 0$, depending only on $C$ and $\delta_*$, such that
\[ | \mathcal{G}(z, \xi,F) - \mathcal{G}(z',\xi',F) | \leq c_3 c_6 \left( |z - z^\prime|^{\alpha} + |\xi - \xi^\prime|^{\alpha} \right).\]
These structural properties permit us to apply quasilinear elliptic estimates up to the boundary as in \cite[Theorem~1]{lieberman1987nonlinear} to conclude that 
\[ \| h \|_{C^{1+\alpha^\prime}({R})} < C,\]
for some $\alpha^\prime \in (0, \alpha]$.  Here we have also used that the fact that $h \in X \subset W_\loc^{3,2}(R)$ and $h$ is uniformly bounded in the local Lipshitz norm by $C$.  

Next, we consider the higher-order derivatives.    For this we yet again exploit the height equation's translation invariance in $q$ to quasi-linearize it by applying $\partial_q$.    That is, $h_q$ is the solution of a uniformly elliptic second-order divergence form PDE with a uniformly oblique boundary condition \eqref{hq height equation}.  Our efforts thus far show that the coefficients of this PDE are uniformly bounded in $C^{\alpha^\prime}(\overline{R})$, thus linear Schauder estimates are enough to get control of $h_q$ in $C^{1+\alpha^\prime}(\overline{R})$ (see, e.g., \cite[Theorem~3]{constantin2011discontinuous}).  Lastly, to bound $h_{pp}$ in $C^{\alpha^\prime}(\overline{R})$, we use the full height equation \eqref{height equation} to express it in terms of $h$, $h_q$, $h_p$, $h_{qq}$, and $h_{qp}$.  

Thus, $h$ is uniformly controlled in $C^{2+\alpha^\prime}(\overline{R})$.  But then it is in particular bounded uniformly in $C^{1+\alpha}(\overline{R})$.  Repeating the same argument above, we see that the coefficients of the linear PDE for $h_q$ are in $C^{\alpha}(\overline{R})$, hence $h_q$ is controlled uniformly in $C^{3+\alpha}(\overline{R})$.  
As before, this is enough to conclude that
\[ \| h \|_{C^{2+\alpha}({R})} < C.\]
It is straightforward to continue in this fashion and obtain uniform bounds of $h$ in $C^{3+\alpha}(\overline{R})$, which finishes the proof. 
\end{proof}

\section{Quoted results} \label{appendix quotes}
For the convenience of the reader, this appendix contains a number of results from the literature that are drawn on in the paper.  

We begin with some essential tools from elliptic theory.  First, let us recall the maximum principle, Hopf boundary lemma, and Serrin edge point lemma \cite{serrin1971symmetry}.  In particular, note that we are using the version that allows for an adverse sign of the zeroth order term provided that the sign of the solution is known; see, for example, \cite{fraenkel2000introduction}, \cite[Lemma~1]{serrin1971symmetry}, and \cite[Lemma~S]{gidas1979symmetry}.

\begin{theorem} \label{max principle}  
  Let $\Omega \subset \mathbb{R}^n$ be a connected, open set (possibly unbounded), and consider the second-order operator $L$ given by
  \begin{equation}
    L := \sum_{i,j = 1}^n a_{ij}(x) \partial_i \partial_j + \sum_{i=1}^n b_i(x) \partial_i + c(x) \label{appendix: def L} 
  \end{equation}
  where $\partial_i := \partial_{x_i}$ and the coefficients $a_{ij}, b_i, c$ are of class $C^0(\overline{\Omega})$.  We assume that $L$ is uniformly elliptic in the sense that there exists $\lambda > 0$ with 
  \begin{equation}
    \sum_{ij} a_{ij}(x) \xi_i \xi_j \geq \lambda |\xi|^2, \qquad \textup{for all } \xi \in \mathbb{R}^n, \, x \in \overline{\Omega},
  \end{equation}
  and that $a_{ij}$ is symmetric.  Let $u \in C^2(\Omega) \cap C^0(\overline{\Omega})$ be a classical solution of $Lu = 0$ in $\Omega$.  

  \begin{enumerate}[label=\rm(\roman*)]
  \item \label{strong max principle} {\rm (Strong maximum principle)} Suppose $u$ attains its maximum value on $\overline{\Omega}$ at a point in the interior of $\Omega$.  If $c \leq 0$ in $\Omega$, or if $\sup_{\Omega} u = 0$, then $u$ is a constant function.  

  \item \label{hopf lemma} {\rm (Hopf boundary lemma)} Suppose that $u$ attains its maximum value on $\overline{\Omega}$ at a point $x_0 \in \partial \Omega$ for which there exists an open ball $B \subset \Omega$ with $\overline{B} \cap \partial\Omega = \{ x_0 \}$.  Assume that either $c \leq 0$ in $\Omega$, or else $\sup_B u = 0$.  Then $u$ is a constant function or 
    \[ \nu \cdot \nabla u(x_0) > 0,\]
    where $\nu$ is the outward unit normal to $\Omega$ at $x_0$.
  \item \label{edge point} {\rm (Serrin edge point lemma)} 
    Let $x_0 \in \partial\Omega$ be an ``edge point" in the sense that near $x_0$ the boundary $\partial \Omega$ consists of two transversally intersecting $C^2$ hypersurfaces $\{\gamma(x) = 0\}$ and $\{\sigma(x) = 0\}$. Suppose that $\gamma, \sigma < 0$ in $\Omega$. If $u\in C^2(\overline{\Omega})$, $u>0$ in $\Omega$ and $u(x_0) = 0$. 
    Assume further that $a_{ij} \in C^2$ in a neighborhood of $x_0$,
    \begin{equation}\label{bluntness2}
      B(x_0) = 0, \quad \text{and } \quad \partial_\tau B(x_0) = 0
    \end{equation}
    for every differential operator $\partial_\tau$ tangential to $\{\gamma=0\} \cap \{\sigma=0\}$ at $x_0$. Then for any unit vector 
    $s$
    outward from $\Omega$ at $x_0$, either
    \begin{equation*}
      {\partial_s u}(x_0) < 0 \ \text{or }\ {\partial^2_s u}(x_0) < 0.
    \end{equation*}
  \end{enumerate}
\end{theorem} 

Next, we present a version of the classical Schauder estimates that applies to unbounded domains; see, for example, the discussion in \cite[Appendix A.1]{wheeler2013solitary}.

\begin{theorem} \label{schauder theorem} 
  For $n > 1$, let $\Omega := \mathbb{R}^{n-1} \times (0,1)$, and denote $\partial_1\Omega := \mathbb{R}^{n-1} \times \{1\}$, $\partial_0 \Omega := \mathbb{R}^{n-1} \times \{0\}$.  Consider the elliptic problem 
  \begin{equation}
    Lu = f \textrm{ in } \Omega, \qquad Bu = g \textrm{ on } \partial_1 \Omega, \qquad u = 0 \textrm{ on } \partial_0 \Omega,  \label{appendix: schauder problem} 
  \end{equation} 
  where $L$ is a second-order uniformly elliptic operator of the form \eqref{appendix: def L}, and $B$ is a uniformly oblique boundary operator: 
  \[ Bu := \beta(x) u + \sum_{i=1}^{n} \gamma_i(x) \partial_i u, \qquad |\gamma_n| \geq \mu > 0. \]
  Fix $\alpha \in (0,1)$ and $k \geq 0$.  We assume that the coefficients have the regularity 
  \[ \| a_{ij}, b_i, c\|_{C^{k+\alpha}({\Omega})}, \, \| \beta, \gamma\|_{C^{k+\alpha}(\partial_1 \Omega)} < \nu.\]   

  Suppose that $u \in C_{\mathrm{b}}^0(\overline{\Omega}) \cap C^{2+\alpha}(\overline{\Omega})$ solves \eqref{appendix: schauder problem} for $f \in C^{k+\alpha}(\overline{\Omega})$ and $g \in C^{k+1+\alpha}(\partial_1 \Omega)$.  
  Then $u \in C_{\mathrm{b}}^{k+2+\alpha}(\overline{\Omega})$ satisfies the Schauder estimate 
  \begin{equation}
    \| u \|_{C^{k+2+\alpha}({\Omega})} \leq C \left( \| u \|_{C^0({\Omega})} + \| f \|_{C^{k+\alpha}({\Omega})} + \| g \|_{C^{k+1+\alpha}(\partial_1\Omega)} \right)  \label{appendix: schauder estimate} 
  \end{equation}
  for a constant $C = C(n, \alpha, k, \lambda, \mu, \nu) > 0$.
\end{theorem}

We quote below the center manifold reduction theorem that forms the basis of the small-amplitude existence theory in Section~\ref{small-amplitude section} (cf.\ \cite{mielke1988reduction}
and \cite{haragus2011book} for a general discussion). The version that we use is specifically designed to take advantage of the Hamiltonian structure of the system. 
 
\begin{theorem}[Buffoni, Groves, and Toland \cite{buffoni1996plethora}] \label{center manifold theorem} 
  Suppose that $(\mathcal{X}, \omega^\epsilon, \mathcal{H}^\epsilon)$ is a one-parameter family of reversible Hamiltonian systems, where $\mathcal{X}$ is a Hilbert space, $\omega^\epsilon$ a symplectic form on $\mathcal{X}$, and $\mathcal{H}^\epsilon$ the Hamiltonian.  Write the corresponding Hamilton equation in the form
  \begin{equation}
    u_q = Lu + N^\epsilon(u), \label{appendix:hamilton equation} 
  \end{equation}
  where $u(q)$ is assumed to lie in $\mathcal{X}$ for each $q$.  We assume that $L \colon \mathcal{D}(L) \subset \mathcal{X} \to \mathcal{X}$ is a densely defined, closed linear operator.   Suppose that $0$ is an equilibrium for \eqref{appendix:hamilton equation} at $\epsilon = 0$ and that the following conditions hold.
  \begin{itemize}
  \item[\emph{(H1)}] 
    The spectrum $\sigma(L)$ of $L$ contains at most finitely many eigenvalues on the imaginary axis, each of which has finite multiplicity.  Moreover, $\sigma(L) \cap i \mathbb{R}$ is separated from $\sigma(L) \setminus i \mathbb{R}$ in the sense of Kato.  Let $P^{\mathrm{c}}$ denote  the spectral projection corresponding to $\sigma(L) \cap i \mathbb{R}$ and put $\mathcal{X}^{\mathrm{c}} := P^{\mathrm{c}}\mathcal{X}$, $\mathcal{X}^{\mathrm{su}} := (1- P^{\mathrm{c}}) \mathcal{X}$. We let $n$ be the (finite) dimension of $\mathcal{X}^{\mathrm{c}}$.
  \item[\emph{(H2)}]  
    There exists $C > 0$ such that the operator $L$ satisfies the resolvent estimate 
    \begin{equation}
      \| u \|_{\mathcal{X}} \leq \frac{C}{1+|\xi|} \| (L - i \xi I) u \|_{\mathcal{X}}, 
    \end{equation}
    for all $\xi \in \mathbb{R}$ and $u \in \mathcal{X}^{\mathrm{su}}$.  
  \item[\emph{(H3)}] 
    There exists a natural number $k$, an interval $\Lambda \subset \mathbb{R}$ containing $0$, and a neighborhood $\mathcal{U}$ of $0$ in $\mathcal{D}(L)$ such that $N$ is $C^{k+1}$ in its dependence on $(\epsilon, u)$ on $ \Lambda \times \mathcal{U}$.  Moreover, $N^0(0) = 0$ and $D_u N^0(0) = 0$.
  \end{itemize}
  Then, after possibly shrinking the interval $\Lambda$ and neighborhood $\mathcal{U}$, we have that, for each $\epsilon \in \Lambda$, there exists an $n$-dimensional local center manifold $\cman \sub \U$ together with an invertible coordinate map 
  \begin{align*}
    \chi^\epsilon := P^\cs |_\cman \maps \cman \to \U^\cs := P^\cs \U 
  \end{align*}
  with the following properties:
  \begin{enumerate}[label=\rm(\roman*)]
  \item Defining $\Psi^\epsilon \maps \U^\cs \to \U^\hs := P^\hs \U$ by 
      $u^\cs + \Psi^\epsilon(u^\cs) = (\chi^\epsilon)^{-1}(u^\cs)$,
    the map $(\epsilon,u) \mapsto \Psi^\epsilon(u)$ is $C^k(\Lambda \by \U^\cs, \U^\hs)$. Moreover $\Psi^\epsilon = 0$ for all $\epsilon \in \Lambda$ and $D_u\Psi^0(0) = 0$.
  \item Every initial condition $u_0 \in \cman$ determines a unique solution $u$ of \eqref{appendix:hamilton equation} which remains in $\cman$ as long as it remains in $\U$.
  \item If $u$ solves \eqref{appendix:hamilton equation} and lies in $U$ for all $q$, then $u$ lies entirely in $\cman$.
  \item If $u^\cs \in C^1((a,b),\U^\cs)$ solves the reduced system
    \begin{align}
      \label{appendix: reduced ode}
      u^\cs_q = 
      f^\epsilon(u^\cs) :=
      L u^\cs + P^\cs N^\epsilon ( u^\cs + \Psi^\epsilon(u^\cs)),
    \end{align}
    then $u = (\chi^\epsilon)^{-1} (u^\cs)$ solves the full system \eqref{appendix:hamilton equation}.
  \item 
    $\mathcal{M}^\epsilon$ is a symplectic submanifold of $\mathcal{X}$ when equipped with the symplectic form $\omega^\epsilon|_{\mathcal{M}^\epsilon}$ and Hamiltonian 
    $\cham(u^\cs) = \ham^\epsilon(u^\cs+\Psi^\epsilon(u^\cs))$.
    The reduced system \eqref{appendix: reduced ode} corresponds to the Hamiltonian flow for $(\mathcal{M}^\epsilon, \omega^\epsilon|_{\mathcal{M}^\epsilon}, \cham)$.  In fact, it is reversible and coincides with the restriction of the full Hamiltonian to the center manifold.
  \end{enumerate}
\end{theorem}

\bibliographystyle{siam}
\bibliography{projectdescription}

\end{document}